\documentclass{amsart}

\usepackage[T1]{fontenc}
\usepackage[utf8]{inputenc}

\usepackage[colorlinks,citecolor=blue,urlcolor=blue]{hyperref}

\usepackage[backend=bibtex,doi=true,isbn=false,url=false,
            style=alphabetic,maxbibnames=99,maxcitenames=99,
            maxalphanames=99]{biblatex}
\usepackage{amsthm,amsmath,amsfonts,amssymb}
\usepackage{microtype}
\usepackage{xspace}
\usepackage{mathtools}
\usepackage[config, labelfont={normalsize}]{caption,subfig}
\usepackage[backgroundcolor=yellow]{todonotes}
\usetikzlibrary{decorations.pathreplacing}
\usetikzlibrary{decorations.pathmorphing}
\usetikzlibrary{decorations.markings}
\usetikzlibrary{colorbrewer}
\usetikzlibrary{arrows.meta}
\usetikzlibrary{patterns}
\usepackage{subfiles}

\colorlet{colA}{Paired-B}
\colorlet{colB}{Paired-F}
\colorlet{colC}{Paired-D}
\colorlet{colD}{Paired-J}

\colorlet{colGhostA}{Paired-A}
\colorlet{colGhostB}{Paired-E}
\colorlet{colGhostC}{Paired-C}
\colorlet{colGhostD}{Paired-I}

\definecolor{colGhost}{HTML}{666666}

\colorlet{colOrange}{Paired-H}
\colorlet{colGhostOrange}{Paired-G}
\colorlet{colBrown}{Paired-L}

\colorlet{colBC}{colOrange}
\colorlet{colP}{colD}
\colorlet{colHeir}{colD}

\colorlet{colGone}{colGhostD}
\colorlet{colGtwo}{colGhost}

\colorlet{colMatrixHeir}{YlGn-A}
\colorlet{colMatrixPlus}{Oranges-D}
\colorlet{colMatrixMinus}{Blues-G}

\tikzset{
  particleA/.style={colA, line width=2pt, solid},
  particleB/.style={colB, line width=0.8pt, double, double distance=2pt},
  particleC/.style={colC, line width=1.5pt, decorate,
    decoration={zigzag, segment length=4pt, amplitude=1pt}},
  particleD/.style={colD, line width=2pt, solid},
  ghostA/.style={colGhostA, line width=1.5pt, dashed},
  ghostB/.style={colGhostB, line width=1.5pt, dotted},
  ghostC/.style={colGhostC, line width=1.5pt, densely dashed},
  styleGhost/.style={colGhost, line width=1pt, densely dotted},
  particleAg/.style={colA, line width=2pt, dashed},
  styleHeir/.style={colHeir, line width=2.5pt, solid,
    postaction={decorate, decoration={markings,
      mark=between positions 2mm and 1 step 5mm with {\draw[-] (0,-3pt) -- (0,3pt);}}}},
}

\makeatletter
\@ifclasswith{amsart}{externalize}%
{
  \ifdefined\ExtendedVersion
    \newcommand{\tikzexternaldisable}{}
    \newcommand{\tikzexternalenable}{}
  \else
    \usetikzlibrary{external}
    \tikzexternalize[
      optimize=false,
      prefix=tikz/,
      mode=list and make]
    \makeatletter
    \renewcommand{\todo}[2][]{\tikzexternaldisable\@todo[#1]{#2}\tikzexternalenable}
    \makeatother
  \fi
}%
{
  \newcommand{\tikzexternaldisable}{}
  \newcommand{\tikzexternalenable}{}
}
\makeatother

\usepackage{aliascnt}
\usepackage{enumitem}
\usepackage[capitalize,noabbrev]{cleveref}

\numberwithin{equation}{section}

\DeclareMathOperator{\id}{id}

\theoremstyle{definition}
\newtheorem{definition}{Definition}[section]

\theoremstyle{plain}
\newaliascnt{theorem}{definition}
\newtheorem{theorem}[theorem]{Theorem}
\aliascntresetthe{theorem}

\newaliascnt{proposition}{definition}
\newtheorem{proposition}[proposition]{Proposition}
\aliascntresetthe{proposition}

\newaliascnt{lemma}{definition}
\newtheorem{lemma}[lemma]{Lemma}
\aliascntresetthe{lemma}

\newaliascnt{corollary}{definition}
\newtheorem{corollary}[corollary]{Corollary}
\aliascntresetthe{corollary}

\theoremstyle{remark}
\newaliascnt{example}{definition}
\newtheorem{example}[example]{Example}
\aliascntresetthe{example}

\newaliascnt{remark}{definition}
\newtheorem{remark}[remark]{Remark}
\aliascntresetthe{remark}

\newaliascnt{principle}{definition}

\aliascntresetthe{principle}

\crefname{theorem}{Theorem}{Theorems}
\Crefname{theorem}{Theorem}{Theorems}
\crefname{lemma}{Lemma}{Lemmas}
\Crefname{lemma}{Lemma}{Lemmas}
\crefname{corollary}{Corollary}{Corollaries}
\Crefname{corollary}{Corollary}{Corollaries}
\crefname{proposition}{Proposition}{Propositions}
\Crefname{proposition}{Proposition}{Propositions}
\crefname{definition}{Definition}{Definitions}
\Crefname{definition}{Definition}{Definitions}
\crefname{example}{Example}{Examples}
\Crefname{example}{Example}{Examples}
\crefname{remark}{Remark}{Remarks}
\Crefname{remark}{Remark}{Remarks}
\crefname{principle}{Principle}{Principles}
\Crefname{principle}{Principle}{Principles}

\usepackage{needspace}
\let\origsubsection\subsection
\renewcommand{\subsection}{\needspace{4\baselineskip}\origsubsection}
\let\origsubsubsection\subsubsection
\renewcommand{\subsubsection}{\needspace{3\baselineskip}\origsubsubsection}

\newcommand{\NN}{\mathbb{N}}
\newcommand{\ZZ}{\mathbb{Z}}
\newcommand{\RR}{\mathbb{R}}
\newcommand{\PP}{\mathbb{P}}

\DeclareMathOperator{\sgn}{sgn}
\newcommand{\gasgn}{\sgn_{\varepsilon}}
\newcommand{\gasgnof}[1]{\sgn_{\varepsilon}(#1)}
\DeclareMathOperator{\heir}{heir}

\newcommand{\ilo}{\mathrel{\triangleleft}}

\newcommand{\iloeps}{\mathrel{\triangleleft_{\varepsilon}}}
\newcommand{\igo}{\mathrel{\triangleright}}

\newcommand{\Actors}{\mathcal{A}}
\newcommand{\Roles}{\mathcal{R}}
\newcommand{\Ghosts}{\mathcal{G}}
\newcommand{\Heirs}{\mathcal{H}}        %
\newcommand{\Active}{\operatorname{Active}}  %
\newcommand{\Junctions}{\mathcal{J}}
\newcommand{\FinalState}{\mathcal{F}}

\newcommand{\casting}{\mathcal{C}}      %
\newcommand{\perf}{\mathcal{P}}         %
\newcommand{\paths}{\mathbf{P}}
\newcommand{\CastingSpace}{\mathbf{C}}  %

\newcommand{\Attribution}{\mathfrak{A}}
\newcommand{\Rehearsal}{\mathfrak{R}}

\newcommand{\Candidates}{\Pi_{\FinalState}}

\newcommand{\urban}{the second-named author\xspace}

\newif\ifshowleanmarkers
\showleanmarkersfalse

\providecommand{\uses}[1]{}

\newif\ifextendedversion
\ifdefined\ExtendedVersion\extendedversiontrue\else\extendedversionfalse\fi

\pgfkeys{/fnote/.cd, prose/.initial={}, lean/.initial={}, py/.initial={}}
\newcommand{\fnote}[1]{%
  \ifextendedversion
    \begingroup
    \pgfkeys{/fnote/.cd,#1}%
    \pgfkeysgetvalue{/fnote/prose}{\fnoteprose}%
    \pgfkeysgetvalue{/fnote/lean}{\fnotelean}%
    \pgfkeysgetvalue{/fnote/py}{\fnotepy}%
    \todo[color=green!10, inline=false]{%
      \textbf{Formalization}%
      \ifx\fnoteprose\empty\else\par \fnoteprose\fi
      \ifx\fnotelean\empty\else\par Lean:
        \texttt{\expandafter\detokenize\expandafter{\fnotelean}}\fi
      \ifx\fnotepy\empty\else\par Py:
        \texttt{\expandafter\detokenize\expandafter{\fnotepy}}\fi
    }%
    \endgroup
  \fi}

\ifextendedversion
  \title[Coalescing particle systems (extended)]%
  {Exact determinant formulas \\ for coalescing particle systems \\[0.35em]
   \normalsize Extended version}
\else
  \title[Coalescing particle systems]%
  {Exact determinant formulas \\ for coalescing particle systems}
\fi
\author{Piotr \'Sniady}
\address[P.~\'Sniady]{Institute of Mathematics,
  Polish Academy of Sciences,
  ul.~\'Sniadeckich 8, \mbox{00-656~Warszawa,} Poland}
\email[P.~\'Sniady]{psniady@impan.pl}

\author{\'Akos Urb\'an}
\address[\'A.~Urb\'an]{Department of Stochastics,
  Budapest University of Technology and Economics,
  Budapest, Hungary}
\address[\'A.~Urb\'an]{HUN-REN Alfr\'ed R\'enyi Institute of
  Mathematics, Re\'altanoda utca 13--15, \mbox{1053~Budapest,} Hungary}
\email[\'A.~Urb\'an]{urbana@math.bme.hu}
\email[\'A.~Urb\'an]{urban.akos@renyi.hu}

\thanks{A 12-page extended abstract of this paper will be submitted to
  the proceedings of the conference \emph{Formal Power Series and
  Algebraic Combinatorics} (FPSAC)~\cite{SU2026fpsac}.}

\begin{document}

\begin{abstract}
When particles on a line collide, they may coalesce
into one. Such systems arise in the voter model, where
boundaries between opinion clusters perform coalescing
random walks, and in reaction-diffusion theory, where
diffusing particles merge on contact. Computing exact
coalescence probabilities has been difficult because
collisions reduce the particle count, while classical
determinantal methods require a fixed number of
particles throughout. We introduce
\emph{ghost particles}: when two particles collide,
one survivor continues as usual and one invisible
ghost is created alongside it, preserving the total
count.
This restores the square matrix structure needed
for a determinantal formula. We prove that the
probability of any specified coalescence
pattern---which initial particles merge into which
survivors---is given by a determinant whose entries are
transition probabilities. Integrating out ghost
positions yields a closed-form formula for the
surviving particles alone: the \emph{coalescence
determinant}. The only assumptions are the Markov
property and nearest-neighbor transitions, so the
results apply wherever the classical non-colliding
theory does: discrete lattice paths, birth-death
chains, and continuous diffusions including Brownian
motion.
\end{abstract}



\maketitle

\vspace{1em}
\begin{center}
\textit{To Marek Bożejko, whose path crossed ours at just the right moments.}
\end{center}
\vspace{1em}

\section{Introduction}\label{sec:intro}
\subsection{The problem}
\label{sec:intro-problem}

Consider $n$ particles performing independent random walks
on a one-dimensional universe for which a version of the
\emph{`Darboux property'} holds true: two particles cannot
swap order without being at some intermediate moment in the
same state. (We postpone the formal definition to the
crossing property of \cref{def:planar} and, for continuous
time, to order preservation in \cref{sec:continuous-km}; for
now a concrete example to keep in mind is a system of simple
random walks, each taking $\pm 1$ steps and starting at even
integers; see the non-dotted lines in \cref{fig:intro-21}.)
When two particles meet, they coalesce into one. \emph{What
	is the probability that the particles end up at specified
	positions, having undergone a specified pattern of
	coalescences?}

This model appears throughout probability and statistical physics,
notably in the voter model~\cite{HolleyLiggett1975}---where tracing
ancestry backward in time produces coalescing lineages.

\medskip

For \emph{non-colliding} particles, exact probabilities are classical.
The Karlin--McGregor theorem~\cite{KM1959} expresses the probability
that $n$ particles starting at positions $x_1 \leq \cdots \leq x_n$
\emph{avoid collision} as a determinant:
\[
\PP(\text{particles reach } y_1, \ldots, y_n \text{ without colliding})
= \det \bigl( p(x_i \to y_j) \bigr)_{1 \leq i,j \leq n},
\]
where $p(x_i \to y_j)$ is the transition probability from $x_i$
to~$y_j$ for a single particle,
and the target positions $y_1 \leq \cdots \leq y_n$ are weakly increasing. 
The combinatorial version is the
Lindstr\"om--Gessel--Viennot (LGV) lemma~\cite{Lindstrom1973,GV1985}.

\medskip

When particles coalesce, the count changes. After $m$ coalescences,
only $k = n - m$ particles remain. The determinant formula breaks
down: the number of rows (initial particles) exceeds the number of
columns (final particles), leaving no square matrix.

\subsection{The ghost method}
\label{sec:intro-ghost}

Our solution is to keep the `discarded' particles walking. Whenever
two particles meet, the collision produces one \emph{heir} and one
\emph{ghost}: the heir continues as a single particle---free to meet
and merge again---while the ghost walks off as an independent random
walk that starts at the collision point and no longer interacts with
anything. Each later merger again splits off one fresh ghost, so a
single heir may absorb many particles by multiple collisions.

The total count---heirs plus ghosts---remains exactly~$n$ at every
stage of the dynamics. With $n$
initial particles and $n$ final entities, we can write a square
matrix and take its determinant. The ghosts are not physically
real from the viewpoint of coalescing random walks, but the
enlarged system admits exact determinantal formulas. We call
this the \emph{ghost particle method}. 

\Cref{fig:intro-21} illustrates this viewpoint for the
coalescence pattern~$2{+}1$: particles~$1$ and~$2$ collide,
producing one heir and one ghost, while particle~$3$ continues
undisturbed. This is the simplest case, a single merger; in general
an heir may merge repeatedly, each collision producing another ghost.

\begin{figure}[t]
\centering

\begin{tikzpicture}[scale=0.55]
  \begin{scope}
    \clip (-0.2,-0.2) rectangle (11.2,6.2);

    \foreach \x in {-1,...,12} {
      \foreach \t in {-1,...,7} {
        \pgfmathparse{mod(\x+\t,2)==0 ? 1 : 0}
        \ifnum\pgfmathresult>0
          \fill[gray!30] (\x,\t) circle (1.5pt);
        \fi
      }
    }

    \foreach \x in {-2,0,2,4,6,8,10,12} {
      \foreach \t in {-2,0,2,4,6,8} {
        \draw[gray!30, thin] (\x,\t) -- (\x-1,\t+1);
        \draw[gray!30, thin] (\x,\t) -- (\x+1,\t+1);
      }
    }
    \foreach \x in {-1,1,3,5,7,9,11} {
      \foreach \t in {-1,1,3,5,7} {
        \draw[gray!30, thin] (\x,\t) -- (\x-1,\t+1);
        \draw[gray!30, thin] (\x,\t) -- (\x+1,\t+1);
      }
    }
  \end{scope}

  \draw[->, thick] (-0.3,0) -- (11,0) node[right] {$x$};
  \draw[->, thick] (0,-0.3) -- (0,6.7) node[above] {$t$};

  \draw (-0.1, 6) -- (0.1, 6);
  \node[left] at (-0.15,6) {\small $T$};

  \fill[colA] (2,0) circle (4pt);
  \fill[colB] (4,0) circle (4pt);
  \fill[colC] (6,0) circle (4pt);

  \node[below] at (2,-0.15) {\small $x_1$};
  \node[below] at (4,-0.15) {\small $x_2$};
  \node[below] at (6,-0.15) {\small $x_3$};

  \draw[particleA] (2,0) -- (3,1) -- (4,2);

  \draw[particleB] (4,0) -- (5,1) -- (4,2);

  \draw[styleHeir] (4,2) -- (3,3) -- (2,4) -- (3,5) -- (4,6);

  \draw[styleGhost] (4,2) -- (5,3);
  \draw[styleGhost, transform canvas={shift={(0.1,-0.1)}}]
    (5,3) -- (6,4) -- (7,5);
  \draw[styleGhost] (7,5) -- (8,6);

  \draw[particleC] (6,0) -- (7,1) -- (6,2) -- (5,3) -- (6,4) -- (7,5) -- (6,6);

  \fill[black] (4,2) circle (5pt);
  \node[left=0.2] at (4,2) {\small $c$};

  \fill[colP] (4,6) circle (4pt);                                    %
  \fill[colC] (6,6) circle (4pt);                                    %
  \draw[colGone, fill=white, line width=1.5pt] (8,6) circle (4pt);   %

  \node[above] at (4,6.15) {\small $y_H$};
  \node[above] at (6,6.15) {\small $y_{H'}$};
  \node[above] at (8,6.15) {\small $y_g$};

  \begin{scope}[shift={(11,0.8)}]
    \draw[particleA] (0,5.2) -- (0.7,5.2);
    \node[right] at (0.8,5.2) {\small particle $1$};
    \draw[particleB] (0,4.4) -- (0.7,4.4);
    \node[right] at (0.8,4.4) {\small particle $2$};
    \draw[styleHeir] (0,3.6) -- (0.7,3.6);
    \node[right] at (0.8,3.6) {\small heir $H$};
    \draw[styleGhost] (0,2.8) -- (0.7,2.8);
    \node[right] at (0.8,2.8) {\small ghost $g$};
    \draw[particleC] (0,2.0) -- (0.7,2.0);
    \node[right] at (0.8,2.0) {\small particle $3$ = heir $H'$};
  \end{scope}

\end{tikzpicture}

\caption{\textbf{Coalescence on the checkerboard lattice (pattern $2{+}1$).}
Three particles start at $x_1 < x_2 < x_3$. Particles~$1$ (solid) and~$2$
(double) coalesce at~$c$, producing heir~$H$ (ticked) and a ghost (dotted).
Particle~$3$ (zigzag) does not coalesce; it is heir~$H'$.
The ghost shares two edges with heir~$H'$ (shown offset)---ghosts do not
interact and may cross any path freely. Final positions: $y_H < y_{H'} < y_g$.}
\label{fig:intro-21}

\end{figure}

\subsection{Main results}
\label{sec:intro-results}

\subsubsection{The coalescence formula}

Consider $n$ particles starting at positions $x_1 \leq \cdots \leq x_n$.
Suppose they coalesce into $k$ heirs at positions
$y_{H_1} <  \dots < y_{H_k}$, producing $m = n - k$ ghosts at positions
$y_{g_1}, \ldots, y_{g_m}$.
The \emph{coalescence pattern} is described by a \emph{composition}
$c_1{+}\cdots{+}c_k = n$, where~$c_j$ counts how many of the initial
particles merge into the heir~$H_j$.

A ghost $g$ comes to rest either to the left or to the
right of its heir, and we record which by a single bit, the
\emph{ghost sign} $\varepsilon_g \in \{-1, +1\}$: it is
$+1$ when the ghost lies to the left of its heir and $-1$
when it lies to the right. We
will often use Iverson's bracket notation
\begin{align*}
	[\varepsilon_g = +1] &=
	\begin{cases} 1 & \text{if } \varepsilon_g = +1, \\
		0 & \text{otherwise}, \end{cases} \\
	[\varepsilon_g = -1] &=
	\begin{cases} 1 & \text{if } \varepsilon_g = -1, \\
		0 & \text{otherwise}; \end{cases}
\end{align*}
in figures, where space is tight, we abbreviate these as
$[\varepsilon_g^+]$ and $[\varepsilon_g^-]$.

Each ghost is created when two adjacent particles merge and the
boundary between them dissolves. We index the ghost by its
\emph{junction index} $g \in \{2, \dots, n\}$, which also indexes the
ghost's column in the matrix below.

\fnote{%
  prose={Informal preview of the main theorem; the formal statement and its
    three Lean layers are at \Cref{thm:coalescence}.},
  lean={main_theorem_layer1, main_theorem_layer2, main_theorem_layer3},
  py={tests/test_determinant_identity.py}}%
\begin{theorem}[Coalescence formula, informal version of
  \cref{thm:coalescence}]\label{thm:intro-coalescence}
The probability that the final outcome follows the prescribed
coalescence pattern, with the final entities at the positions above, is
the determinant
\begin{equation}
	\label{eq:coalescence}
	\PP = \det M
\end{equation}
of the $n \times n$ matrix $M$ whose rows are indexed by the initial
particles and whose columns are indexed by the final entities. The
columns are arranged in groups of $c_1, \dots, c_k$: the $i$-th group
consists of the heir~$H_i$ followed by the $c_i - 1$ ghosts created on
the way to~$H_i$. An heir column holds the usual transition
probabilities
\[ M_{i,H} = p(x_i \to y_H), \]
while a ghost column~$g$ has a signed ``staircase'' structure of
$\pm$~transition probabilities and zeros, selected by the ghost sign:
\[
M_{i,g} = \begin{cases}
-[\varepsilon_g=-1]\cdot p(x_i \to y_g)  & \text{if } i < g,
\\[0.3em]
\phantom{-}[\varepsilon_g=+1] \cdot p(x_i \to y_g)   & \text{if } i \geq g.
\end{cases}
\]
\end{theorem}

See \cref{fig:matrix-structure} for the shape of~$M$.

\begin{figure}[t]
\centering

\begin{tikzpicture}[scale=0.95]

\def\cellsize{2.4}

\colorlet{cbYellow}{colMatrixHeir}   %
\colorlet{cbOrange}{colMatrixPlus}   %
\colorlet{cbPurple}{colMatrixMinus}  %

\pgfdeclarepatternformonly{dashed ne lines}
  {\pgfpoint{-0.1cm}{-0.1cm}}{\pgfpoint{0.4cm}{0.4cm}}
  {\pgfpoint{0.3cm}{0.3cm}}
  {\pgfsetlinewidth{0.6pt}\pgfsetdash{{2pt}{2pt}}{0pt}
   \pgfpathmoveto{\pgfpoint{-0.1cm}{-0.1cm}}
   \pgfpathlineto{\pgfpoint{0.4cm}{0.4cm}}\pgfusepath{stroke}}

\pgfdeclarepatternformonly{solid nw lines}
  {\pgfpoint{-0.1cm}{-0.1cm}}{\pgfpoint{0.4cm}{0.4cm}}
  {\pgfpoint{0.3cm}{0.3cm}}
  {\pgfsetlinewidth{0.6pt}\pgfpathmoveto{\pgfpoint{-0.1cm}{0.4cm}}
   \pgfpathlineto{\pgfpoint{0.4cm}{-0.1cm}}\pgfusepath{stroke}}

\tikzset{
    cellHeir/.style={fill=cbYellow, draw=black, line width=0.5pt},
    cellPlus/.style={fill=cbOrange, postaction={pattern=dashed ne lines, pattern color=black},
                     draw=black, line width=0.5pt},
    cellMinus/.style={fill=cbPurple, postaction={pattern=solid nw lines, pattern color=black},
                      draw=black, line width=0.5pt},
    entrylabel/.style={font=\scriptsize, fill=white, inner sep=1pt},
}

\draw[cellHeir] (0*\cellsize, 3*\cellsize) rectangle (1*\cellsize, 4*\cellsize);
\node[entrylabel] at (0.5*\cellsize, 3.5*\cellsize) {$p(x_1 \to y_H)$};

\draw[cellMinus] (1*\cellsize, 3*\cellsize) rectangle (2*\cellsize, 4*\cellsize);
\node[entrylabel] at (1.5*\cellsize, 3.5*\cellsize) {$-[\varepsilon_g^-]\, p(x_1 \to y_g)$};

\draw[cellHeir] (2*\cellsize, 3*\cellsize) rectangle (3*\cellsize, 4*\cellsize);
\node[entrylabel] at (2.5*\cellsize, 3.5*\cellsize) {$p(x_1 \to y_{H'})$};

\draw[cellMinus] (3*\cellsize, 3*\cellsize) rectangle (4*\cellsize, 4*\cellsize);
\node[entrylabel] at (3.5*\cellsize, 3.5*\cellsize) {$-[\varepsilon_{g'}^-]\, p(x_1 \to y_{g'})$};

\draw[cellHeir] (0*\cellsize, 2*\cellsize) rectangle (1*\cellsize, 3*\cellsize);
\node[entrylabel] at (0.5*\cellsize, 2.5*\cellsize) {$p(x_2 \to y_H)$};

\draw[cellPlus] (1*\cellsize, 2*\cellsize) rectangle (2*\cellsize, 3*\cellsize);
\node[entrylabel] at (1.5*\cellsize, 2.5*\cellsize) {$[\varepsilon_g^+]\, p(x_2 \to y_g)$};

\draw[cellHeir] (2*\cellsize, 2*\cellsize) rectangle (3*\cellsize, 3*\cellsize);
\node[entrylabel] at (2.5*\cellsize, 2.5*\cellsize) {$p(x_2 \to y_{H'})$};

\draw[cellMinus] (3*\cellsize, 2*\cellsize) rectangle (4*\cellsize, 3*\cellsize);
\node[entrylabel] at (3.5*\cellsize, 2.5*\cellsize) {$-[\varepsilon_{g'}^-]\, p(x_2 \to y_{g'})$};

\draw[cellHeir] (0*\cellsize, 1*\cellsize) rectangle (1*\cellsize, 2*\cellsize);
\node[entrylabel] at (0.5*\cellsize, 1.5*\cellsize) {$p(x_3 \to y_H)$};

\draw[cellPlus] (1*\cellsize, 1*\cellsize) rectangle (2*\cellsize, 2*\cellsize);
\node[entrylabel] at (1.5*\cellsize, 1.5*\cellsize) {$[\varepsilon_g^+]\, p(x_3 \to y_g)$};

\draw[cellHeir] (2*\cellsize, 1*\cellsize) rectangle (3*\cellsize, 2*\cellsize);
\node[entrylabel] at (2.5*\cellsize, 1.5*\cellsize) {$p(x_3 \to y_{H'})$};

\draw[cellMinus] (3*\cellsize, 1*\cellsize) rectangle (4*\cellsize, 2*\cellsize);
\node[entrylabel] at (3.5*\cellsize, 1.5*\cellsize) {$-[\varepsilon_{g'}^-]\, p(x_3 \to y_{g'})$};

\draw[cellHeir] (0*\cellsize, 0*\cellsize) rectangle (1*\cellsize, 1*\cellsize);
\node[entrylabel] at (0.5*\cellsize, 0.5*\cellsize) {$p(x_4 \to y_H)$};

\draw[cellPlus] (1*\cellsize, 0*\cellsize) rectangle (2*\cellsize, 1*\cellsize);
\node[entrylabel] at (1.5*\cellsize, 0.5*\cellsize) {$[\varepsilon_g^+]\, p(x_4 \to y_g)$};

\draw[cellHeir] (2*\cellsize, 0*\cellsize) rectangle (3*\cellsize, 1*\cellsize);
\node[entrylabel] at (2.5*\cellsize, 0.5*\cellsize) {$p(x_4 \to y_{H'})$};

\draw[cellPlus] (3*\cellsize, 0*\cellsize) rectangle (4*\cellsize, 1*\cellsize);
\node[entrylabel] at (3.5*\cellsize, 0.5*\cellsize) {$[\varepsilon_{g'}^+]\, p(x_4 \to y_{g'})$};

\node[left] at (0, 3.5*\cellsize) {$x_1$};
\node[left] at (0, 2.5*\cellsize) {$x_2$};
\node[left] at (0, 1.5*\cellsize) {$x_3$};
\node[left] at (0, 0.5*\cellsize) {$x_4$};

\node[above] at (0.5*\cellsize, 4*\cellsize) {$H$};
\node[above] at (1.5*\cellsize, 4*\cellsize) {$g$};
\node[above] at (2.5*\cellsize, 4*\cellsize) {$H'$};
\node[above] at (3.5*\cellsize, 4*\cellsize) {$g'$};

\node[below, font=\footnotesize] at (0.5*\cellsize, 0) {heir};
\node[below, font=\footnotesize] at (1.5*\cellsize, 0) {ghost};
\node[below, font=\footnotesize] at (2.5*\cellsize, 0) {heir};
\node[below, font=\footnotesize] at (3.5*\cellsize, 0) {ghost};

\draw[line width=3pt, black]
    (0, 4*\cellsize) -- (1*\cellsize, 4*\cellsize) -- (1*\cellsize, 3*\cellsize) --
    (2*\cellsize, 3*\cellsize) -- (2*\cellsize, 2*\cellsize) --
    (3*\cellsize, 2*\cellsize) -- (3*\cellsize, 1*\cellsize) --
    (4*\cellsize, 1*\cellsize) -- (4*\cellsize, 0);

\begin{scope}[shift={(-1, -1.8)}]
  \draw[cellHeir] (0, 0) rectangle (0.5, 0.5);
  \node[right, font=\small] at (0.6, 0.25) {heir column};

  \draw[cellMinus] (4, 0) rectangle (4.5, 0.5);
  \node[right, font=\small] at (4.6, 0.25) {above step ($i <g $)};

  \draw[cellPlus] (9, 0) rectangle (9.5, 0.5);
  \node[right, font=\small] at (9.6, 0.25) {below step ($i \geq g$)};
\end{scope}

\end{tikzpicture}

\caption{\textbf{The coalescence matrix: staircase structure.}
Matrix $M$ for the $2{+}2$ coalescence pattern:
the initial particles $1$ and $2$ merge into heir~$H$;
the initial particles $3$ and $4$ merge into heir~$H'$.
Heir columns (yellow) contain plain transition probabilities
$p(x_i \to y_H)$, $p(x_i \to y_{H'})$.
Ghost columns show the staircase pattern with
the thick staircase line separating the two regions.
(With the notations of
\cref{sec:setup-coalescence,sec:setup-final}, the heir $H$
corresponds to the interval $[1,3)$, and its ghost $g$
corresponds to the junction $2$;
similarly the heir $H'$ corresponds to the
interval $[3,5)$ and its ghost $g'$ to the junction $4$,
so that the columns $H, g, H', g'$ correspond to the final
entities $[1,3),\, 2,\, [3,5),\, 4$ listed in $\min$-order, with
$\min$-values $1 < 2 < 3 < 4$.)
}
\label{fig:matrix-structure}
\end{figure}

\subsubsection{Example: coalescence pattern $2{+}1$}
\label{sec:intro-example}

Returning to \Cref{fig:intro-21}: three particles start at
$x_1 < x_2 < x_3$. Particles~$1$ and~$2$ collide and merge
into heir~$H$; particle~$3$ does not collide and is the
sole particle behind heir~$H'$. The result: two heirs at positions
$y_H < y_{H'}$ and one ghost~$g$. Heir $H$ and its unique ghost $g$
form the first group of columns, the heir $H'$ forms the second group.
The $3 \times 3$ matrix is:
\[
M = \begin{pmatrix}
p(x_1 \to y_H) & -[\varepsilon_g=-1] \cdot p(x_1 \to y_g) & p(x_1 \to y_{H'}) \\[1ex]
p(x_2 \to y_H) & \phantom{-}[\varepsilon_g=+1] \cdot p(x_2 \to y_g) & p(x_2 \to y_{H'}) \\[1ex]
p(x_3 \to y_H) &  \phantom{-}[\varepsilon_g=+1] \cdot p(x_3 \to y_g) & p(x_3 \to y_{H'})
\end{pmatrix}.
\]

\medskip

More concretely, the scenario shown on \cref{fig:intro-21}
corresponds to the case $\varepsilon_g=-1$ (i.e.,~$y_g \geq
y_{H}$), so the matrix takes the form
\begin{equation}\label{eq:intro-right}
M = \begin{pmatrix}
	p(x_1 \to y_H) & -p(x_1 \to y_g) & p(x_1 \to y_{H'}) \\[1ex]
	p(x_2 \to y_H) & 0 & p(x_2 \to y_{H'}) \\[1ex]
	p(x_3 \to y_H) & 0 & p(x_3 \to y_{H'})
\end{pmatrix};
\end{equation}
the determinant of~\eqref{eq:intro-right} expands to a difference of two
products of single-particle transition probabilities.

In the other case $\varepsilon_g=+1$ (i.e., $y_g \leq y_{H}$)
the matrix takes the form
\begin{equation}\label{eq:intro-left}
M = \begin{pmatrix}
	p(x_1 \to y_H) & 0 & p(x_1 \to y_{H'}) \\[1ex]
	p(x_2 \to y_H) & p(x_2 \to y_g) & p(x_2 \to y_{H'}) \\[1ex]
	p(x_3 \to y_H) & p(x_3 \to y_g) & p(x_3 \to y_{H'})
\end{pmatrix};
\end{equation}
the determinant of~\eqref{eq:intro-left} is a signed sum of four
products.

This asymmetry---two terms versus four---reflects the staircase
structure: the ghost's index and sign determine which rows contribute.
Such asymmetry is a general feature of the coalescence formula.

\medskip

Several boundary configurations---weakly increasing heir
positions, a ghost coinciding with its heir, and coinciding
initial positions---refine how the theorem should be read; we
treat them in \Cref{rem:formula-edge-cases}.

\subsubsection{The coalescence determinant}

For applications, one typically wants the probability of heir
positions alone, summed over all ghost positions. Integrating out
the ghosts yields the \emph{coalescence determinant}
(\Cref{sec:determinant}): a ghost-free, closed-form determinant whose heir
columns contain transition probabilities and whose ghost columns
contain cumulative distribution functions in a staircase pattern.

\subsection{Relation to classical determinantal formulas}
\label{sec:intro-lgv}
The coalescence formula \eqref{eq:coalescence} generalizes the
Karlin--McGregor / Lindstr\"om--Gessel--Viennot (LGV) determinant,
and its proof adapts the classical
segment-swap argument (\Cref{sec:classic-proof}), which pairs the
configurations where particles cross and cancels them in the signed
sum of the determinant. The argument originates with Karlin and
McGregor~\cite{KM1959}; Lindstr\"om~\cite{Lindstrom1973} independently
discovered the same cancellation in the setting of matroid theory,
Gessel and Viennot~\cite{GV1985} developed it into the LGV lemma for
lattice-path counting, and Stembridge~\cite{Stembridge1990} extended it
to acyclic digraphs, introducing $D$-compatibility as the structural
condition; see the survey of
Krattenthaler~\cite{Krattenthaler2015}.

In our setting, crossings are not forbidden but prescribed: the ghost
configuration specifies which crossings must occur. The involution has
the same structure---swap at the first wrong crossing---but ``wrong''
now means ``crossing that violates the prescribed pattern.'' The ghost
signs~$(\varepsilon_g)$ encode which crossings are required, and the
Iverson brackets in the ghost columns retain exactly those terms
where particles cross at the right places. When no coalescences occur (no
ghosts, every entity an heir), the formula reduces to the
classical LGV determinant.

\subsection{Prior work}
\label{sec:intro-prior}

Coalescing Brownian motions were constructed by
Arratia~\cite{Arratia1979}. Exact distributional results
have since been obtained through several independent
lines of work.

\subsubsection{Gap distributions and the IPDF method}

Doering and
ben-Avraham~\cite{DoeringBenAvraham1988} introduced the
\emph{inter-particle distribution function} (IPDF) method,
computing nearest-neighbor gap distributions for
diffusion-limited coalescence;
ben-Avraham~\cite{benAvraham1998} extended this to the
full hierarchy of empty-interval probabilities;
ben-Avraham and Brunet~\cite{benAvrahamBrunet2005}
extracted explicit densities for consecutive spacings.
See \Cref{sec:intro-gaps-warren} for the ghost-based
derivation.

\subsubsection{Warren's determinantal formula}
\label{sec:intro-warren-prior}

Warren~\cite{Warren2007} proved another $n \times n$
determinantal formula for coalescing Brownian motions.
His mechanism does not use ghosts: after coalescence,
merged particles simply follow the same trajectory.
Writing $Z_t(x_i)$ for the position at time~$t$ of the
survivor containing particle~$x_i$, the joint cumulative
distribution function
$\PP(Z_t(x_i) \leq y_i \text{ for all } i)$ can be expressed
as a determinant. Since merged particles share a trajectory,
the formula does not resolve which particles have
coalesced. Assiotis, O'Connell, and
Warren~\cite{AssiotisOConnellWarren2019} extended
this via intertwining relations to general diffusions,
and Assiotis~\cite{Assiotis2018,Assiotis2023} to
birth-death chains. The ghost formula
provides a finer, pattern-level resolution of Warren's
determinant (see \Cref{sec:intro-gaps-warren}).

\subsubsection{Pfaffian point processes}
\label{sec:intro-pfaffian-prior}

A different question arises for infinitely many
particles: what is the statistical structure of the
surviving positions? Tribe and
Zaboronski~\cite{TZ2011} proved that under the maximal
entrance law for coalescing Brownian motions (particles
starting from every point of~$\RR$), the surviving
positions form a Pfaffian point process with an explicit
$2 \times 2$ matrix kernel involving the complementary
error function. Garrod, Poplavskyi, Tribe, and
Zaboronski~\cite{GarrodPTZ2018} extended this to
continuous-time random walks on~$\ZZ$ with spatially
inhomogeneous rates and arbitrary deterministic
initial conditions.
Their framework also covers annihilation and mixed
coalescence-annihilation and, as shown by Tribe and
Zaboronski~\cite{TribeZaboronski2026}, all entrance
laws---but it requires a time-homogeneous Markov
generator with a specific algebraic structure
(the spin-pair identity).
Using the Pfaffian structure,
Fomichov~\cite{Fomichov2016} found the exact
distribution of the number of surviving particles,
and Glinyanaya and
Fomichov~\cite{GlinyanayaFomichov2018} proved a
central limit theorem with explicit variance.

The ghost formula and this Pfaffian approach have
complementary strengths (\Cref{sec:intro-wide-scope}); see
\Cref{sec:intro-pfaffian-companion} for the ghost-based
derivation.

\subsubsection{The P\'olya web}

A key motivating example is the \emph{P\'olya web},
introduced by \urban~\cite{Urban2025}: coalescing
chains on~$\NN^2$ whose steps follow a P\'olya urn rule.
Each walk has a Beta-distributed limiting direction, and
coalescence corresponds to equality of limits. For
pairwise coalescence, Urban shows that the joint density
of these limiting directions is a determinant whose
entries alternate between Beta density and cumulative
distribution functions---a special case of our
coalescence determinant (\Cref{sec:intro-polya-companion}).

\subsection{Scope and structure of the method}
\label{sec:intro-scope}

Beyond the specific formulas, the ghost method has
several structural features---wide applicability,
algebraic flexibility, and exactness---that enable the
results of the companion papers.

\subsubsection{Wide scope}
\label{sec:intro-wide-scope}

The ghost formula requires exactly the same assumptions
as the Karlin--McGregor theorem~\cite{KM1959}---order
preservation, identical and independent dynamics, the strong Markov
property, and
meeting times being stopping times (\Cref{sec:continuous-km})---and
therefore shares its broad scope: lattice random
walks with arbitrary inhomogeneous transition
probabilities, birth-death chains, Brownian motion,
and more generally any \emph{skip-free} Markov process
(transitions only to neighboring states, so that
particles cannot change order without first meeting).
The analytic approaches surveyed in \Cref{sec:intro-prior}
have a different reach: the Pfaffian point process
method~\cite{TZ2011,GarrodPTZ2018} requires access to the
Markov generator and its spin-pair identity, and the IPDF
method~\cite{DoeringBenAvraham1988,benAvraham1998} relies on
Brownian-motion-specific integrals; but they reach results
the ghost method does not, such as mixed
coalescence-annihilation and all entrance
laws~\cite{TribeZaboronski2026}. Neither framework subsumes
the other.

\subsubsection{Algebraic flexibility and exactness}

The ghost formula is a determinant, and an \emph{exact} one---not
an asymptotic expansion or a scaling limit. The applications in the
companion papers rest on precisely this: determinants admit row and
column operations, and exactness permits algebraic rearrangement
\emph{before} any limit is taken. Confluent limits of nearly
coinciding boundary points produce the $2 \times 2$ block structure
of the Tribe--Zaboronski Pfaffian
kernel~\cite{TZ2011,GarrodPTZ2018,Sniady2026pfaffian}; collapsing
the determinant's columns one by one recovers Warren's
formula~\cite{Warren2007}, whose cumulative-distribution entries are
revealed as the ghosts' contribution; and rearranging exact cumulant
integrals yields a combinatorial central limit theorem for basin
boundaries~\cite{Sniady2026pfaffian}.

\medskip

At a structural level, the ghost method is a minimal
extension of the Karlin--McGregor theorem: the same
sign-reversing involution applied to the same
determinantal language, with one new ingredient---the
ghost particles that restore the matrix to square
form.

\subsection{New viewpoint on the classic construction}

The classic construction of a finite coalescing system fixes a
priority order on the particles, runs their trajectories, and at each
collision keeps the highest-priority particle while discarding the
rest~\cite{Arratia1979}. Our viewpoint reclaims this waste: the
absorbed particles are reclassified as ghosts, and the classic
construction is thereby retrofitted to produce the system of $n$
particles and ghosts studied here (\cref{sec:continuous}).

\subsection{Companion papers}
\label{sec:intro-companion}

This paper is part of a series of five
(\Cref{fig:companion-papers}). The left column
develops exact combinatorial formulas (this paper for
coalescence, a companion for annihilation), while the
right column applies them to probability (gap
distributions, Pfaffian point processes, and a
detailed example). The top row is the coalescence
story (ghost particle method $\to$
wall-particle system), the bottom row starts
from annihilation (\emph{ghost pair method}) and uses
it to derive Pfaffian structure for the coalescing
system. The two rows are independent---neither relies
on the other. A fifth paper~\cite{BSTU2026} applies
both lines of work to the P\'olya web.

\begin{figure}[t]
\centering
\ifdefined\tikzexternalize\tikzset{external/export=false}\fi%
\begin{tikzpicture}[
  papernode/.style={
    draw=colB, very thick, rounded corners=4pt,
    minimum width=3.8cm, minimum height=1.6cm,
    text width=3.4cm, align=center,
    font=\small
  },
  dashednode/.style={
    papernode, densely dotted, draw=colA
  },
  implyarrow/.style={
    -{Implies}, double, double distance=2pt,
    thick
  },
  paralleline/.style={
    dashed, thick
  },
  collabel/.style={
    font=\small\itshape, text width=3.8cm,
    align=center
  }
]

\def\Wx{0}      %
\def\Ex{8.0}    %
\def\Ty{1.0}    %
\def\My{-1.5}   %
\def\By{-4.0}   %
\def\Bpad{2.3}  %
\def\Btop{2.2}  %
\def\Bbot{-5.2} %

\draw[black!30, thick, dashed, rounded corners=12pt]
  ({\Wx-\Bpad}, \Bbot) rectangle ({\Wx+\Bpad}, \Btop);
\draw[black!30, thick, dashed, rounded corners=12pt]
  ({\Ex-\Bpad}, \Bbot) rectangle ({\Ex+\Bpad}, \Btop);

\node[collabel] at (\Wx, {\Bbot-0.8})
  {exact combinatorial\\formulas};
\node[collabel] at (\Ex, {\Bbot-0.8})
  {probability\\applications};

\node[papernode] (NW) at (\Wx, \Ty)
  {this paper\\[2pt]
   ghost particle method};

\node[papernode] (NE) at (\Ex, \Ty)
  {\cite{Sniady2026coalescenceApplications}\\[2pt]
   wall-particle system};

\node[dashednode] (SW) at (\Wx, \By)
  {\cite{Sniady2026annihilation}\\[2pt]
   ghost pair method};

\node[papernode] (ME) at (\Ex, \My)
  {\cite{BSTU2026}\\[2pt]
   example: P\'olya web};

\node[papernode] (SE) at (\Ex, \By)
  {\cite{Sniady2026pfaffian}\\[2pt]
   Pfaffian structure\\of walls};

\draw[implyarrow] (NW) -- (NE)
  node[midway, above, font=\small] {application};

\draw[implyarrow] (SW) -- (SE)
  node[midway, above, font=\small,
       text width=2cm, align=center]
    {cancellative\\labeling};

\draw[implyarrow] (NE) -- (ME);
\draw[implyarrow] (SE) -- (ME);

\draw[paralleline] (NW) -- (SW)
  node[midway, left, font=\small\itshape,
       text width=1.8cm, align=right]
    {parallel\\constructions};

\end{tikzpicture}

\vspace{1.5em}
\begin{tikzpicture}[font=\small]
  \draw[colB, very thick, rounded corners=3pt]
    (0,0) rectangle (0.8,0.4);
  \node[right] at (0.9, 0.2) {coalescence};
  \draw[colA, very thick, densely dotted,
        rounded corners=3pt]
    (3.8,0) rectangle (4.6,0.4);
  \node[right] at (4.7, 0.2) {annihilation};
\end{tikzpicture}

\caption{\textbf{This paper and its four companions.} Two parallel
combinatorial constructions---the ghost particle method for coalescence
developed here and the ghost pair method for annihilation---yield exact
determinant formulas. These are applied to probability, giving gap
distributions and Warren's formula and the Pfaffian structure of walls,
and both lines converge on the P\'olya web as a worked example.}
\label{fig:companion-papers}
\end{figure}

\subsubsection{Gap distributions and Warren's formula
  {\normalfont\cite{Sniady2026coalescenceApplications}}}
\label{sec:intro-gaps-warren}

Under the maximal entrance law
(\Cref{sec:intro-pfaffian-prior}), coalescing particles come
down from infinity: at any positive time only finitely many
\emph{survivors} remain per bounded interval, each attracting a
basin of initial particles that eventually merge into it. The
boundaries between adjacent basins are the \emph{walls}
(\Cref{fig:wall-particle}).

\begin{figure}[t]
\centering

\begin{tikzpicture}[scale=0.75]

\def\W{11.5}  %
\def\H{6.0}   %

\draw[gray, thick] (-0.3, 0) -- (\W, 0);
\draw[gray, thick] (-0.3, \H) -- (\W, \H);

\node[gray] at (-0.8, 0) {$\cdots$};
\node[gray] at (\W + 0.5, 0) {$\cdots$};
\node[gray] at (-0.8, \H) {$\cdots$};
\node[gray] at (\W + 0.5, \H) {$\cdots$};

\draw[->, thick, black] (-1.2, 0.4) -- (-1.2, \H - 0.4);
\node[rotate=90, black] at (-1.6, \H/2) {time};

\node[right, font=\small] at (\W + 0.9, 0) {$t = 0$};
\node[right, font=\small] at (\W + 0.9, \H) {$t > 0$};

\draw[black!50, line width=0.4pt]
  (0.0, 0.0) -- (0.38, 0.75) -- (0.0, 1.5) -- (0.38, 2.25);
\draw[black!50, line width=1.0pt]
  (0.38, 2.25) -- (0.75, 3.0) -- (1.12, 3.75) -- (1.5, 4.5);
\draw[black!50, line width=1.4pt]
  (1.5, 4.5) -- (1.12, 5.25) -- (1.5, 6.0);
\draw[black!30, line width=0.4pt]
  (0.75, 0.0) -- (1.12, 0.75) -- (0.75, 1.5) -- (0.38, 2.25);
\draw[black!30, line width=0.4pt]
  (1.5, 0.0) -- (1.12, 0.75);
\draw[black!30, line width=0.4pt]
  (2.25, 0.0) -- (2.62, 0.75) -- (2.25, 1.5) -- (2.62, 2.25)
  -- (2.25, 3.0) -- (1.88, 3.75) -- (1.5, 4.5);

\draw[colA, line width=0.4pt]
  (3.0, 0.0) -- (3.38, 0.75);
\draw[colA, line width=0.7pt]
  (3.38, 0.75) -- (3.75, 1.5) -- (4.12, 2.25);
\draw[colA, line width=1.4pt]
  (4.12, 2.25) -- (4.5, 3.0) -- (4.12, 3.75) -- (3.75, 4.5)
  -- (3.38, 5.25) -- (3.0, 6.0);
\draw[colA!40, line width=0.4pt]
  (3.75, 0.0) -- (3.38, 0.75);
\draw[colA!40, line width=0.4pt]
  (4.5, 0.0) -- (4.88, 0.75) -- (4.5, 1.5) -- (4.12, 2.25);
\draw[colA!40, line width=0.4pt]
  (5.25, 0.0) -- (4.88, 0.75);
\draw[colA!40, line width=0.4pt]
  (6.0, 0.0) -- (5.62, 0.75) -- (5.25, 1.5) -- (4.88, 2.25)
  -- (4.5, 3.0);

\draw[colB, line width=0.4pt]
  (6.75, 0.0) -- (7.12, 0.75) -- (7.5, 1.5) -- (7.88, 2.25);
\draw[colB, line width=1.4pt]
  (7.88, 2.25) -- (7.5, 3.0) -- (7.12, 3.75) -- (6.75, 4.5)
  -- (7.12, 5.25) -- (7.5, 6.0);
\draw[colB!40, line width=0.4pt]
  (7.5, 0.0) -- (7.88, 0.75) -- (8.25, 1.5) -- (7.88, 2.25);
\draw[colB!40, line width=0.4pt]
  (8.25, 0.0) -- (8.62, 0.75) -- (8.25, 1.5);
\draw[colB!40, line width=0.4pt]
  (9.0, 0.0) -- (8.62, 0.75);

\draw[black!50, line width=0.4pt]
  (9.75, 0.0) -- (10.12, 0.75);
\draw[black!50, line width=0.7pt]
  (10.12, 0.75) -- (10.5, 1.5) -- (10.12, 2.25)
  -- (9.75, 3.0) -- (10.12, 3.75) -- (10.5, 4.5)
  -- (10.88, 5.25) -- (11.25, 6.0);
\draw[black!30, line width=0.4pt]
  (10.5, 0.0) -- (10.12, 0.75);

\foreach \px in {0.00, 0.75, 1.50, 2.25, 3.00, 3.75, 4.50,
                 5.25, 6.00, 6.75, 7.50, 8.25, 9.00, 9.75, 10.50} {
  \fill (\px, 0) circle (1.5pt);
}

\foreach \bx in {2.62, 6.38, 9.38} {
  \fill (\bx, 0.30) -- ++(-0.22, -0.42) -- ++(0.44, 0) -- cycle;
}

\foreach \hx in {1.50, 3.00, 7.50, 11.25} {
  \fill (\hx, \H) circle (5pt);
}

\end{tikzpicture}

\caption{Coalescing random walks starting from every site.
Paths merge on contact; line weight increases with each merger.
Walls (triangles at the bottom) mark the boundaries between basins of
attraction at~\mbox{$t = 0$}; survivors (large dots on top) are the
particles that remain at~$t > 0$, one per basin.}
\label{fig:wall-particle}
\end{figure}

The companion
paper~\cite{Sniady2026coalescenceApplications} applies the
coalescence determinant (\Cref{sec:determinant}) to this
\emph{wall-particle system}---the joint process of survivors and
walls---deriving gap distributions and generalizing Warren's
formula (\Cref{sec:intro-warren-prior}) to arbitrary skip-free
processes.

\subsubsection{Annihilation
  {\normalfont\cite{Sniady2026annihilation}}}
\label{sec:intro-annihilation}

The ghost method also applies to annihilation ($A + A \to
\emptyset$), where both particles are destroyed upon collision.
Each collision now produces a \emph{ghost pair}---two
independent random walks starting from the collision
point---rather than one heir and one ghost. The companion
paper~\cite{Sniady2026annihilation} uses this \emph{ghost pair
method} to derive an annihilation formula with the same
determinantal structure and sign-reversing involution; when all
$n = 2k$ particles annihilate completely, the determinant
reduces to a Pfaffian. Applications include domain wall dynamics
in the Ising--Glauber model~\cite{benAvrahamHavlin2000}.

\subsubsection{Pfaffian structure of walls
  {\normalfont\cite{Sniady2026pfaffian}}}
\label{sec:intro-pfaffian-companion}

The companion paper~\cite{Sniady2026pfaffian} proves a Pfaffian
empty-interval formula for the walls of any skip-free coalescing
system: a deterministic \emph{cancellative labeling} converts
pairwise coalescence into total annihilation, and the
annihilation formula~\cite{Sniady2026annihilation} then supplies
the Pfaffian structure, yielding explicit cumulants and a central
limit theorem for the wall count. \emph{Checkerboard duality}
identifies surviving particles with walls of a dual process,
transferring the Pfaffian structure to the survivors themselves;
specializing to Brownian motion under the maximal entrance law
recovers the Tribe--Zaboronski kernel~\cite{TZ2011,GarrodPTZ2018}.

\subsubsection{The P\'olya web
  {\normalfont\cite{BSTU2026}}}
\label{sec:intro-polya-companion}

The companion paper~\cite{BSTU2026} applies both the coalescence
determinant and the Pfaffian machinery to the P\'olya
web~\cite{Urban2025}. The P\'olya web is a natural testing ground
for the general theory: it is genuinely non-homogeneous
(transition probabilities depend on position) yet exactly
solvable, and it mixes the discrete (finitely many survivors at
each level) with the continuous (each survivor acquires a
Beta-distributed asymptotic direction). In suitable projective
coordinates the paper derives exact configuration probabilities,
an arcsine law for boundary positions, and Rayleigh spacing for
gaps between adjacent survivors; the Pfaffian kernel, built from
Beta crossing probabilities, converges to the complementary error
function kernel of coalescing Brownian motions, and the number of
survivors satisfies a central limit theorem with the
Fano factor of coalescing Brownian motions. Urban's original proof~\cite{Urban2025}
reaches the same determinant from a complementary
direction---conditioning the Karlin--McGregor determinant on
successive coalescences---and derives the distribution of the
number of survivors and edge scaling to a Yule web.

\subsection{Organization}
\label{sec:intro-organization}

\Cref{sec:setup} formalizes the coalescence model on weighted
directed acyclic graphs, defining \emph{performances} (coalescence
histories) and their weights. \Cref{sec:formula} states precisely the
main result, the coalescence formula, and gives an outline of its
proof; the proof itself occupies \Cref{sec:proof-setup,%
sec:proof-attribution,sec:rehearsal,sec:proof-bijection,%
sec:involution,sec:proof-main}, which develop, in turn, the
combinatorial framework, the attribution map, the rehearsal
algorithm, their mutual inversion, the sign-reversing involution,
and the assembly of the proof. \Cref{sec:continuous} extends the
formula to continuous time and space, covering Brownian motion and
birth-death chains, and \Cref{sec:determinant} integrates out ghost
positions, yielding the closed-form ghost-free \emph{coalescence
determinant} used in the companion papers.

\subsection{Companion code}

Two software deposits accompany this paper, providing two
different kinds of evidence.

The first deposit~\cite{Sniady2026formalization} contains a
machine-checked \emph{proof}: the discrete results of this
paper are formalized and verified in the Lean~4 proof
assistant, accompanied by a Python reference implementation
of the paper's algorithms. \ifextendedversion This
is the \emph{extended version} of the article:
\Cref{sec:lean-appendix} describes that formalization, and
notes throughout the text link each result to its Lean and
Python counterparts. \else An \emph{extended version} of
this article, distributed with the formalization, adds a
companion appendix describing it and notes linking each
result to its machine-checked counterpart. \fi

The second deposit~\cite{Sniady2026companion} contains
exact \emph{numerical tests} of the theorems: the
simulation and exact-enumeration code which was used to
discover the formulas of the present paper and of the
companion paper on
annihilation~\cite{Sniady2026annihilation}, and which
compares them, in exact arithmetic, against independent
brute-force enumeration.

\section{Setup}\label{sec:setup}

The coalescence formula holds for random walks on $\ZZ$, Brownian
motion on~$\RR$, and birth-death chains on arbitrary state spaces.
Rather than prove each case separately, we work with \emph{spacetime
graphs}---an abstraction that captures two structural properties
common to all settings:
\begin{enumerate}[label=(\roman*)]
\item \textbf{Planarity}: paths with swapped endpoints must cross,
  and non-adjacent particles cannot meet without an intermediate
  particle involved;
\item \textbf{Weight-preserving segment swap}: exchanging path
  segments at crossings preserves total weight.
\end{enumerate}
For discrete models (lattice walks), the spacetime graph is literal:
vertices are space-time points, edges are allowed transitions, and
edge weights are transition probabilities. For continuous processes
(Brownian motion, continuous-time jump processes),
the graph is a conceptual tool---the actual proof
uses measure-theoretic arguments (\Cref{sec:continuous}), but the
combinatorial structure is identical.

\subsection{Spacetime graphs}
\label{sec:setup-spacetime}

\fnote{%
  prose={The spacetime DAG: the substrate of vertices, edges, and weights.},
  lean={Coalescence.Geometry.Spacetime},
  py={geometry.dag.DAG}}%
\begin{definition}[Spacetime graph]\label{def:spacetime-graph}
A \emph{spacetime graph} is a directed acyclic graph $D = (V, E)$ equipped
with an edge weight function $w\colon E \to R$, where $R$ is a commutative
ring (e.g., $\ZZ$ for combinatorial counting, $\RR_{\geq 0}$ for
probabilities, or formal power series for generating functions).
The directed acyclic graph structure induces a
\emph{time ordering}: $u \lessdot v$ if there is a
directed path from $u$ to~$v$. This is a partial
order; for the
sign-reversing involution (\Cref{sec:involution}), we fix a linear extension $\le$
of $\lessdot$. The combinatorial proof works for any such extension. The
phrase ``first crossing'' means first in this order.
\end{definition}

\fnote{%
  prose={A path and its weight (the product of edge weights along it).},
  lean={Coalescence.Geometry.Spacetime.Path.weight},
  py={geometry.dag.DAG}}%
\begin{definition}[Paths and weights]\label{def:path-weight}
A \emph{path} from $x$ to $y$ is a sequence of vertices
$(v_0, v_1, \ldots, v_\ell)$ with $v_0 = x$, $v_\ell = y$, and each
$(v_i, v_{i+1}) \in E$. We identify a path with its set of vertices
$\{v_0, \ldots, v_\ell\}$, writing $v \in P$ when $v$ is a vertex of $P$;
thus for paths $A, C$ the intersection $A \cap C$ and union $A \cup C$ are
sets of vertices. The \emph{weight} of a path
is the product of
its edge weights:
\[
w(P) = \prod_{i=0}^{\ell-1} w(v_i \to v_{i+1}).
\]
The \emph{weight from $x$ to $y$} is the total weight of all directed
paths between them:
\[
W(x \to y) = \sum_{\substack{P \text{ path} \\ \text{from } x \text{ to } y}}
w(P).
\]
When the edge weights are the transition probabilities of a single
particle, $W(x \to y)$ is the probability that the particle, started
at $x$, reaches $y$---the transition probability written $p(x \to y)$
in the introduction. The general theory below works with the weight
$W$; the probabilistic reading is the special case.
\end{definition}

\begin{example}[Main examples]\label{ex:main-examples}
	
\ %
\begin{itemize}
\item \emph{Product spacetime} $\ZZ \times \ZZ_{\geq 0}$: vertices $(x, t)$
  with edges to $(x', t+1)$ for allowed transitions.
\item \emph{Checkerboard lattice}: vertices where $x + t$ is even, with
  edges to $(x \pm 1, t+1)$ (simple random walk, see \cref{fig:intro-21}).
\end{itemize}
\end{example}

\Cref{fig:genealogy} illustrates these concepts with a running example:
four particles coalescing on the lattice $\ZZ^2$ with North and East
steps. We will use this example throughout the paper.

\begin{figure}[t]
\centering

\tikzset{
	styleA/.style={colA, line width=2pt, solid,
		decorate, decoration={coil, segment length=3pt, amplitude=1pt}},
	styleB/.style={colB, line width=1.5pt, solid,
		decorate, decoration={zigzag, segment length=4pt, amplitude=1pt}},
	styleC/.style={colC, line width=1.5pt, solid,
		decorate, decoration={snake, segment length=5pt, amplitude=0.8pt}},
	styleD/.style={colBC, line width=2pt, solid},
	styleBC/.style={colBrown, line width=1.5pt, solid, double, double distance=1pt},
	styleP/.style={colP, line width=2.5pt, solid},
	styleGone/.style={colGone, line width=1.5pt, dashed, dash pattern={on 4pt off 2pt}},
	styleGtwo/.style={colGtwo, line width=1.5pt, dotted, dash pattern={on 1pt off 2pt}},
	leafnode/.style={circle, fill, inner sep=2.5pt},
	coalnode/.style={circle, fill=black, inner sep=3pt},
	rootnode/.style={circle, fill=colP, inner sep=3pt},
	ghostend/.style={circle, draw, line width=1pt, fill=white, inner sep=2pt},
}

\begin{tikzpicture}[scale=0.24]
	\begin{scope}
		\clip (-4.5,1.5) rectangle (28.5, 24.5);
		\draw[gray!30, thin, step=2] (-5, 1) grid (30, 27);
	\end{scope}

	\node[left, fill=white] at (-0.2,14) {$x_{I_1}$};
	\node[left, fill=white] at (1.8,10) {$x_{I_2}$};
	\node[left, fill=white] at (3.8,6) {$x_{I_3}$};
	\node[left, fill=white] at (5.8,4) {$x_{I_4}$};
	
	\node[below right=0.2, fill=white] at (8,12) {$c_3$};
	\node[below right=0.2, fill=white] at (12,16) {$c_2$};
	
	\node[above, fill=white] at (16,22.5) {$y_{[1,4)}$};
	\node[right, fill=white] at (20.3,20) {$y_2$};
	\node[right, fill=white] at (24.3,18) {$y_3$};
	\node[right, fill=white] at (24.3,12) {$y_{[4,5)}$};

	\draw[styleGone] (12,16) -- (13,16) -- (14,16) -- (15,16) -- (16,16) -- (17,16)
	-- (18,16) -- (19,16) -- (20,16) -- (20,17) -- (20,18) -- (20,19)
	-- (20,20);

	\draw[styleGtwo] (8,12) -- (8,13) -- (8,14) -- (8,15) -- (8,16) -- (8,17) -- (8,18)
	-- (9,18) -- (10,18) -- (11,18) -- (12,18) -- (13,18)
	-- (14,18) -- (15,18) -- (16,18) -- (17,18) -- (18,18) -- (19,18)
	-- (20,18) -- (21,18) -- (22,18) -- (23,18) -- (24,18);

	\draw[styleA] (0,14) -- (0,15) -- (0,16) -- (1,16) -- (2,16) -- (3,16) -- (4,16)
	-- (5,16) -- (6,16) -- (7,16) -- (8,16) -- (9,16) -- (10,16)
	-- (11,16) -- (12,16);
	\draw[styleB] (2,10) -- (2,11) -- (2,12) -- (3,12) -- (4,12) -- (5,12)
	-- (6,12) -- (7,12) -- (8,12);
	\draw[styleC] (4,6) -- (5,6) -- (6,6) -- (7,6) -- (8,6) -- (8,7) -- (8,8)
	-- (8,9) -- (8,10) -- (8,11) -- (8,12);
	\draw[styleBC] (8,12) -- (9,12) -- (10,12) -- (11,12) -- (12,12)
	-- (12,13) -- (12,14) -- (12,15) -- (12,16);
	\draw[styleP] (12,16) -- (12,17) -- (12,18) -- (12,19) -- (12,20) -- (12,21)
	-- (12,22) -- (13,22) -- (14,22) -- (15,22) -- (16,22);
	\draw[styleD] (6,4) -- (7,4) -- (8,4) -- (9,4) -- (10,4) -- (11,4) -- (12,4)
	-- (13,4) -- (14,4) -- (15,4) -- (16,4) -- (17,4) -- (18,4)
	-- (19,4) -- (20,4) -- (21,4) -- (22,4) -- (23,4) -- (24,4)
	-- (24,5) -- (24,6) -- (24,7) -- (24,8) -- (24,9) -- (24,10)
	-- (24,11) -- (24,12);

	\tikzset{
		myarrow/.style={-{Stealth[length=8pt, width=6pt]}, line width=1.2pt},
		bigarrow/.style={-{Stealth[length=12pt, width=10pt]}, line width=1.5pt},
	}
	\draw[myarrow, colGone] (14.5,16) -- (16,16);
	\draw[myarrow, colGone] (20,17.5) -- (20,19);
	\draw[myarrow, colGtwo] (8,13.5) -- (8,15);
	\draw[myarrow, colGtwo] (14.5,18) -- (16,18);
	\draw[bigarrow, colA] (4,16) -- (6.5,16);
	\draw[bigarrow, colB] (5.8,12) -- (6,12);
	\draw[bigarrow, colC] (8,8.8) -- (8,9);
	\draw[bigarrow, colBrown] (9,12) -- (11,12);
	\draw[bigarrow, colBrown] (12,13) -- (12,15);
	\draw[myarrow, colP] (12,18) -- (12,19.5);
	\draw[myarrow, colP] (13.5,22) -- (15,22);
	\draw[bigarrow, colBC] (12,4) -- (14,4);
	\draw[bigarrow, colBC] (24,8) -- (24,10);

	\node[leafnode, colA] at (0,14) {};
	\node[leafnode, colB] at (2,10) {};
	\node[leafnode, colC] at (4,6) {};
	\node[leafnode, colBC] at (6,4) {};

	\node[coalnode, fill=colBrown] at (8,12) {};
	\node[coalnode, fill=colP] at (12,16) {};

	\node[rootnode] at (16,22) {};
	\node[ghostend, draw=colGone] at (20,20) {};
	\node[ghostend, draw=colGtwo] at (24,18) {};
	\node[leafnode, colBC] at (24,12) {};

	\begin{scope}[shift={(30,6)}]
		\node[anchor=west, font=\small\bfseries] at (0,14) {Particles};
		\draw[styleA] (0,12.5) -- (2,12.5); \node[anchor=west] at (2.5,12.5) {\small $I_1=[1,2)$};
		\draw[styleB] (0,11) -- (2,11); \node[anchor=west] at (2.5,11) {\small $I_2=[2,3)$};
		\draw[styleC] (0,9.5) -- (2,9.5); \node[anchor=west] at (2.5,9.5) {\small $I_3=[3,4)$};
		\draw[styleD] (0,8) -- (2,8); \node[anchor=west] at (2.5,8) {\small $I_4=[4,5)=H'$};
		\draw[styleBC] (0,6.5) -- (2,6.5); \node[anchor=west] at (2.5,6.5) {\small $[2,4)$};
		\draw[styleP] (0,5) -- (2,5); \node[anchor=west] at (2.5,5) {\small $[1,4)=H$};
		\node[anchor=west, font=\small\bfseries] at (0,3) {Ghosts};
		\draw[styleGone] (0,1.5) -- (2,1.5); \node[anchor=west] at (2.5,1.5) {\small $2$};
		\draw[styleGtwo] (0,0) -- (2,0); \node[anchor=west] at (2.5,0) {\small $3$};
	\end{scope}
\end{tikzpicture}

\caption{\textbf{Running example: coalescence pattern $3{+}1$.}
	A~\emph{performance} records the collision structure on a spacetime graph---here
	the lattice $\ZZ^2$ with North/East steps. Four particles start at
	$x_{I_1}$, $x_{I_2}$, $x_{I_3}$, $x_{I_4}$ (leaves, colored dots). Particles
	$I_2$ and $I_3$ meet at $c_3$, merging into $[2,4)$; then $I_1$ and
	$[2,4)$ meet at $c_2$, forming heir $H=[1,4)$ which ends at the root
	$y_{[1,4)}$. Particle $I_4$ does not coalesce; it is heir $H'=[4,5)$.
	Ghost paths (dashed/dotted) emanate from merger points:
	ghost~$2$ is born at $c_2$, ghost~$3$ at~$c_3$. Note that ghost~$3$ crosses
	through particle $I_1$---ghosts are non-interacting.
	See \Cref{fig:boundary} for the labeling scheme.}
\label{fig:genealogy}

\end{figure}

\begin{figure}[t]

\begin{tikzpicture}[scale=0.9]

\def\axislen{6}
\def\height{3.5}
\def\rectheight{0.35}

\draw[->, thick] (-0.3, 0) -- (\axislen, 0);
\foreach \x in {0,1,2,3,4,5} {
	\draw (\x, -0.1) -- (\x, 0.1);
	\node[below] at (\x, -0.2) {\small $\x$};
}

\draw[->, thick] (-0.3, \height) -- (\axislen, \height);
\foreach \x in {0,1,2,3,4,5} {
	\draw (\x, \height-0.1) -- (\x, \height+0.1);
	\node[above] at (\x, \height+0.2) {\small $\x$};
}

\draw[fill=colA!30, draw=colA, thick, rounded corners=3pt]
    (1.05, 0.15) rectangle (1.95, 0.15+2*\rectheight);
\node[circle, fill=colA, inner sep=1.5pt] (I1circ) at (1.2, 0.15+\rectheight) {};
\node at (1.55, 0.15+\rectheight) {$I_1$};

\draw[fill=colB!30, draw=colB, thick, rounded corners=3pt]
    (2.05, 0.15) rectangle (2.95, 0.15+2*\rectheight);
\node[circle, fill=colB, inner sep=1.5pt] (I2circ) at (2.2, 0.15+\rectheight) {};
\node at (2.55, 0.15+\rectheight) {$I_2$};

\draw[fill=colC!30, draw=colC, thick, rounded corners=3pt]
    (3.05, 0.15) rectangle (3.95, 0.15+2*\rectheight);
\node[circle, fill=colC, inner sep=1.5pt] (I3circ) at (3.2, 0.15+\rectheight) {};
\node at (3.55, 0.15+\rectheight) {$I_3$};

\draw[fill=colBC!30, draw=colBC, thick, rounded corners=3pt]
    (4.05, 0.15) rectangle (4.95, 0.15+2*\rectheight);
\node[circle, fill=colBC, inner sep=1.5pt] (I4circ) at (4.2, 0.15+\rectheight) {};
\node at (4.55, 0.15+\rectheight) {$I_4$};

\draw[fill=colP!30, draw=colP, thick, rounded corners=3pt]
    (1.05, \height-0.15-2*\rectheight) rectangle (3.95, \height-0.15);
\node[circle, fill=colP, inner sep=1.5pt] (Hcirc) at (1.2, \height-0.15-\rectheight) {};
\node at (2.5, \height-0.15-\rectheight) {$H$};

\draw[fill=colBC!30, draw=colBC, thick, rounded corners=3pt]
    (4.05, \height-0.15-2*\rectheight) rectangle (4.95, \height-0.15);
\node[circle, fill=colBC, inner sep=1.5pt] (Hpcirc) at (4.2, \height-0.15-\rectheight) {};
\node at (4.55, \height-0.15-\rectheight) {$H'$};

\node[circle, draw=colGone, fill=white, line width=1.5pt, inner sep=2pt] (g2circ)
    at (2, \height-0.15-2*\rectheight-0.25) {};
\node[right, colGone, black] at (2, \height-0.15-2*\rectheight-0.4) {\small $2$};

\node[circle, draw=colGtwo, fill=white, line width=1.5pt, inner sep=2pt] (g3circ)
    at (3, \height-0.15-2*\rectheight-0.25) {};
\node[right, colGtwo, black] at (3, \height-0.15-2*\rectheight-0.4) {\small $3$};

\draw[black, thick, ->] (I1circ) to[out=90, in=-90] (Hcirc);
\draw[black, thick, ->] (I2circ) to[out=90, in=-120] (g3circ);
\draw[black, thick, ->] (I3circ) to[out=90, in=-90] (Hpcirc);
\draw[black, thick, ->] (I4circ) to[out=90, in=-120] (g2circ);

\end{tikzpicture}

\caption{\textbf{Interval labeling and the final state.}
Initial particles (\emph{actors}) are labeled by unit intervals; 
final entities (\emph{roles}) by intervals (heirs)
or junction points (ghosts). This diagram shows the same example as
\Cref{fig:genealogy}. Here $n=4$ (\Cref{ex:final-state}): initial
particles $\Actors = \{I_1, I_2, I_3, I_4\}$ and final entities
$\Roles = \{H, 2, 3, H'\}$. Heirs $H = [1,4)$ and $H' = [4,5)$;
ghosts appear at junctions $2$ and~$3$.
The $\min$ function is indicated by the small dots: each interval $[a,b)$
has a dot at position~$a$ (its left endpoint), so $\min$ reads off
the horizontal coordinate of the dot.
Arrows show one bijection $\pi$: $I_1 \mapsto H$,\quad $I_2 \mapsto 3$,\quad
$I_3 \mapsto H'$,\quad $I_4 \mapsto 2$ between the initial particles
and the final entities.
Under $\min$, this becomes the permutation
$1 \mapsto 1$,\quad $2 \mapsto 3$,\quad $3 \mapsto 4$,\quad $4 \mapsto 2$---the
$3$-cycle $(2\ 3\ 4)$ with sign~$+1$.}
\label{fig:boundary}

\end{figure}

\subsection{Planarity}
\label{sec:setup-planarity}

The classical LGV lemma counts \emph{non-intersecting} paths, so it only
needs one geometric condition: paths with swapped endpoints must cross.
For interacting particles, we need a second condition controlling
\emph{which} particles can collide.

\fnote{%
  prose={The source/target sets $\mathcal{X}$, $\mathcal{Y}$; the boundary
    data the theorem is stated over.},
  lean={Coalescence.Boundary},
  py={valuations.context.Boundary}}%
\begin{definition}[Source and target sets]\label{def:source-target}
The \emph{source set} $\mathcal{X} \subseteq V$ and \emph{target set}
$\mathcal{Y} \subseteq V$ are each equipped with a linear order $\prec$.
\end{definition}

\fnote{%
  prose={Planarity P1 (ordered cross-pairing paths must intersect) and P2
    (a collision of the outer two of three ordered paths forces the middle
    one in). The only geometric hypotheses.},
  lean={Coalescence.Geometry.P1, Coalescence.Geometry.P2}}%
\begin{definition}[Planar configuration]\label{def:planar}
The pair $(\mathcal{X}, \mathcal{Y})$ is \emph{planar} if:
\begin{enumerate}[label=(P\arabic*)]
\item \label{itm:P1-crossing} \textbf{Crossing property.}
  For $x \preceq x'$ in $\mathcal{X}$ and $y' \preceq y$ in $\mathcal{Y}$
  (targets swapped), every path from $x$ to $y$ intersects every path
  from $x'$ to $y'$.
\item \label{itm:P2-consecutive} \textbf{Consecutive collision property.}
  Let $A, B, C$ be paths starting at $x \preceq x' \preceq x''$ in $\mathcal{X}$
  and ending in $\mathcal{Y}$. If $v \in A \cap C$, then there is a vertex
  $w \le v$ with $w \in B \cap (A \cup C)$.
\end{enumerate}
\end{definition}

The crossing property~\ref{itm:P1-crossing} is Stembridge's
$D$-compatibility~\cite{Stembridge1990}: when paths have swapped
endpoints, they must meet somewhere (see \Cref{fig:planarity-p1}).
Stembridge's key insight was
formulating this condition abstractly for general acyclic digraphs,
rather than relying on specific lattice geometry. This abstraction is
what enables generalizations to new settings---including ours.

The consecutive collision property~\ref{itm:P2-consecutive} is our
addition for interacting particles. It ensures that non-adjacent
particles cannot collide without involving intermediate
particles---a physical constraint reflecting that particles on a line
cannot jump over each other. Classical LGV does not need this because
it forbids all crossings; we need it because collisions are allowed
but must respect the spatial ordering (see \Cref{fig:planarity-p2}).

\tikzset{
  nosign radius/.initial=0.38,
  nosign/.pic={
    \draw[red, line width=1.3pt] (0,0)
      circle[radius={\pgfkeysvalueof{/tikz/nosign radius}}];
    \draw[red, line width=1.3pt]
      (-135:{\pgfkeysvalueof{/tikz/nosign radius}})
      -- (45:{\pgfkeysvalueof{/tikz/nosign radius}});
  },
}

\begin{figure}[t]
\centering
\captionsetup[subfigure]{justification=centering}

\subfloat[]{%
\begin{tikzpicture}[scale=0.6,
    vertex/.style={circle, fill, inner sep=1.5pt}]

\draw[gray, thick] (0, 0) -- (6, 0);
\draw[gray, thick] (0, 5) -- (6, 5);
\node[right] at (6.1, 0) {$\mathcal{X}$};
\node[right] at (6.1, 5) {$\mathcal{Y}$};

\node[vertex, colB] (a0) at (1, 0) {};   \node[below] at (a0) {$A$};
\node[vertex, colB] at (5, 5) {};
\node[vertex, colA] (b0) at (5, 0) {};   \node[below] at (b0) {$B$};
\node[vertex, colA] at (1, 5) {};

\draw[particleA] (5, 0) -- (1, 5);
\draw[particleB] (1, 0) -- (2.4, 1.75);
\draw[white, line width=6pt]
  (2.4, 1.75) to[out=95, in=175] (3.0, 2.95) to[out=-5, in=140] (3.6, 3.25);
\draw[particleB]
  (2.4, 1.75) to[out=95, in=175] (3.0, 2.95) to[out=-5, in=140] (3.6, 3.25);
\draw[particleB] (3.6, 3.25) -- (5, 5);

\pic at (4.5, 2.35) {nosign};

\node[font=\scriptsize] at (1.75, 2.95) {bridge};

\end{tikzpicture}%
\label{fig:planarity-p1}%
}
\hfill
\subfloat[]{%
\begin{tikzpicture}[scale=0.6,
    vertex/.style={circle, fill, inner sep=1.5pt},
    collision/.style={circle, fill=black, inner sep=1.8pt}]

\draw[gray, thick] (0, 0) -- (6.4, 0);
\draw[gray, thick] (0, 5) -- (6.4, 5);
\node[right] at (6.5, 0) {$\mathcal{X}$};
\node[right] at (6.5, 5) {$\mathcal{Y}$};

\node[vertex, colA] (a0) at (1, 0) {};   \node[below] at (a0) {$A$};
\node[vertex, colB] (b0) at (3, 0) {};   \node[below] at (b0) {$B$};
\node[vertex, colC] (c0) at (5, 0) {};   \node[below] at (c0) {$C$};

\node[vertex, colA] at (1, 5) {};
\node[vertex, colB] at (3, 5) {};
\node[vertex, colC] at (5, 5) {};

\node[collision] (v) at (3, 2.5) {};
\node[left=2pt] at (v) {$v$};

\node[collision] (w) at (3.95, 3.6) {};
\node[left=2pt] at (w) {$w$};

\draw[particleA] (1, 0) -- (v) -- (1, 5);

\draw[particleC] (5, 0) -- (4.45, 0.69);     %
\draw[particleC] (3.95, 1.31) -- (v);        %
\draw[particleC] (v) -- (w) -- (5, 5);        %

\draw[particleB] (3, 0) -- (3.9, 0.8);       %
\draw[white, line width=6pt]
  (3.9, 0.8) to[out=60, in=180] (4.45, 1.5) to[out=0, in=120] (5.05, 1.2);
\draw[particleB]
  (3.9, 0.8) to[out=60, in=180] (4.45, 1.5) to[out=0, in=120] (5.05, 1.2);
\draw[particleB]
  (5.05, 1.2) to[out=40, in=-85] (5.5, 2.5) to[out=95, in=-25] (w);
\draw[particleB] (w) -- (3, 5);              %

\pic at (6.25, 1.45) {nosign};

\node[font=\scriptsize] at (4.5, 2.05) {bridge};

\end{tikzpicture}%
\label{fig:planarity-p2}%
}

\caption{\textbf{The two planarity conditions} (\Cref{def:planar}), each shown
as a \emph{forbidden} configuration; time runs upward, with sources in
$\mathcal{X}$ (bottom) and targets in $\mathcal{Y}$ (top).
\textbf{(a)~Crossing property~\ref{itm:P1-crossing}.} Paths $A$ and $B$ have
swapped endpoints---$A$ runs from the left source to the right target and $B$
vice versa---yet reach them \emph{without} intersecting, one bridging over the
other. Condition~\ref{itm:P1-crossing} forbids this: paths with swapped
endpoints must cross.
\textbf{(b)~Consecutive collision property~\ref{itm:P2-consecutive}.} Paths
$A$ and $C$ intersect at the shared vertex $v \in A \cap C$.
The intermediate path~$B$ meets $A\cup C$ only at the vertex marked~$w$,
which lies \emph{above}~$v$.
Condition~\ref{itm:P2-consecutive} forbids this: $B$ must share a vertex
with $A$ or~$C$ at or before~$v$.}
\label{fig:planarity}

\end{figure}

\subsection{Interval labels}
\label{sec:setup-coalescence}

\subsubsection{Actors, heirs, and ghosts}

Fix source vertices $x_1 \preceq \cdots \preceq x_n$ in
$\mathcal{X}$. The coalescence structure is tracked using interval
labels (see \Cref{fig:boundary}).

\fnote{%
  prose={An interval: a contiguous run of actors.},
  lean={Coalescence.Interval},
  py={combinatorics.domain.Interval}}%
\begin{definition}[Interval labeling]\label{def:interval-labels}
The initial particle (also called \emph{actor}) starting at $x_j$
is labeled by the unit half-open
interval $I_j = [j, j+1)$; we also write $x_{I_j}$ or $x_I$ for the
starting position of actor~$I$. The set of actors is
$\Actors = \{I_1, \ldots, I_n\}$. The \emph{junction points} are
$\Junctions = \{2, 3, \ldots, n\}$---the shared endpoints between
consecutive intervals. Throughout, interval labels are half-open, so
that the labels of distinct active particles are disjoint.

When particles $[a, b)$ and $[b, c)$ collide, both disappear and two
new entities are born: an \emph{heir} labeled $[a, c)$ (the union),
and a \emph{ghost} labeled by the junction point~$b$ (the dissolved boundary).
In particular, particles that share a starting position
coalesce instantly: the coalescence occurs at time zero,
and the heir and ghost emerge from the shared vertex.
\end{definition}

This labeling lets us trace origins: a particle labeled $[a,c)$
arose from the merger of all initial particles $I_j$ with
$a \leq j < c$. Similarly, ghost~$b$ records which coalescence created
it: the one between the particles on either side of junction~$b$.

Intuition: the heir $[a,c)$ is ``\emph{the owner of the interval $[a,c)$},''
and ghost~$b$ is ``\emph{attached to the boundary $b$ that used to
separate the particles on its left from those on its right}.''

\subsubsection{Label order}

The \emph{label order} $\ilo$ (read ``label-less-than'') orders
intervals and junctions by position, by the following cases:
\begin{itemize}
\item interval before junction: $[a,b) \ilo g$ iff $b \leq g$;
\item junction before interval: $g \ilo [a,b)$ iff $g \leq a$;
\item interval before interval: $[a,b) \ilo [a',b')$ iff $b \leq a'$;
\item junction before junction: $g \ilo g'$ iff $g < g'$.
\end{itemize}
We write $g \igo I$ (read ``label-greater'') for the reverse relation
$I \ilo g$. For example, $I_1 \ilo 2 \ilo I_2 \ilo 3 \ilo I_3$.
This is only a partial order: a junction $g$ and an interval $[a,b)$
with $a < g < b$ are $\ilo$-incomparable, since neither $b \leq g$ nor
$g \leq a$ holds. In particular, a ghost is never comparable with its
own heir.

\subsubsection{$\min$-labeling}

The function $\min$ reads off the smallest element of an entity,
identifying intervals and junctions with $[n] = \{1, \ldots, n\}$:
an interval maps to its left endpoint, $[a,b) \mapsto a$, and a
junction is its own minimum, $k \mapsto k$.
On the role set $\Roles$ of final entities (the heirs and their ghosts,
\Cref{sec:setup-final})---where each ghost junction lies strictly inside
its heir, so distinct entities have distinct minima---$\min$ is
injective, and the resulting total order, the \emph{$\min$-order}, is a
linear extension of the label order $\ilo$: it places each heir
immediately before the ghosts it contains, breaking the incomparability
noted above in favor of the heir.

\subsection{Initial state}
\label{sec:setup-initial}

The \emph{initial state} fixes the actors and where they start: the
$n$ particles $\Actors = \{I_1, \ldots, I_n\}$
(\Cref{def:interval-labels}) begin at the source positions
$x_1 \preceq \cdots \preceq x_n$ in~$\mathcal{X}$, with
$x_i$ the starting position of actor~$I_i$.

\subsection{Final state}
\label{sec:setup-final}

The \emph{final state} $\FinalState$ specifies:
\begin{itemize}
\item The \emph{ghost set} $\Ghosts \subseteq \Junctions$: junctions where
  coalescence occurred;
\item The \emph{heir set} $\Heirs$: maximal half-open intervals
  $[a,b) \subseteq [1, n+1)$ not containing any
  junction from~$\Junctions\setminus \Ghosts$ in their interior;
\item The \emph{role set} $\Roles = \Heirs \cup \Ghosts$: the final entities;
\item Final positions $y_f \in \mathcal{Y}$ for each entity $f \in \Roles$.
\end{itemize}
We assume heir positions respect label order: if $H \ilo H'$, then
$y_H \preceq y_{H'}$. This is not a restriction: by the crossing
property~\ref{itm:P1-crossing}, any physically realizable coalescence
pattern must satisfy this ordering. Writing $k = |\Heirs|$ for the
number of heirs and $m = |\Ghosts|$ for the number of ghosts, the
cardinality is $|\Roles| = k + m = n$.

\fnote{%
  prose={The heir map (junction to interval); Lean's paper-aligned match is
    \texttt{containingInterval}.},
  lean={Coalescence.IntervalPartition.containingInterval}}%
\begin{definition}[Heir function]\label{def:heir}
For ghost $g \in \Ghosts$, we write $\heir(g)$ for the unique heir
interval containing junction $g$: $\heir(g) = [a, b)$ where $a < g < b$.
\end{definition}

\begin{example}[Coalescence pattern $3{+}1$]\label{ex:final-state}
Consider $n = 4$ initial particles (\Cref{fig:boundary,fig:genealogy}):
\[
\Actors = \{I_1,\, I_2,\, I_3,\, I_4\}.
\]
Suppose $I_2$ and $I_3$ coalesce (ghost at junction~$3$), then the merged
particle $[2,4)$ coalesces with $I_1$ (ghost at junction~$2$), while
$I_4$ remains separate. The final state has:
\[
\Ghosts = \{2,\, 3\}, \quad
\Heirs = \{[1,4),\, [4,5)\}, \quad
\Roles = \{[1,4),\, 2,\, 3,\, [4,5)\}.
\]
Two heirs and two ghosts. Under the identification with $[4]$:
\begin{itemize}
\item $\Actors = \{I_1,\, I_2,\, I_3,\, I_4\} 
\longleftrightarrow \{1, 2, 3, 4\}$
      (particle $I_j$ maps to $j$);
\item $\Roles = \{[1,4),\, 2,\, 3,\, [4,5)\} \longleftrightarrow \{1,\, 2,\, 3,\, 4\}$
      (heir $[1,4) \mapsto 1$, heir $[4,5) \mapsto 4$, ghosts map to
      themselves).
\end{itemize}
Any bijection $\pi\colon \Actors \to \Roles$ is thus a permutation
of $\{1, 2, 3, 4\}$.
\end{example}

The ghost set $\Ghosts$ determines a composition
$c_1{+}\cdots{+}c_k = n$, where $c_j$ counts the initial
particles that merge into the $j$th heir. In the example above,
$\Ghosts = \{2, 3\}$ gives composition~$3{+}1$;
\Cref{fig:matrix-structure} illustrates the $2{+}2$ pattern.

\subsection{Performances}
\label{sec:setup-performances}

The classic Lindström--Gessel--Viennot lemma abstracts non-colliding random
walks into a graph-theoretic problem: counting tuples of
vertex-disjoint paths. We seek an analogous abstraction for
coalescing particles. The dynamics---particles moving, meeting,
merging---is replaced by a static combinatorial object: a forest of
genealogy trees recording which particles merged, together with ghost
paths recording where each ghost traveled (\Cref{fig:genealogy}).
This abstraction is the \emph{performance}.

\fnote{%
  prose={Not formalized: genealogy structures enter Lean only through the
    path weights they contribute; the forest itself is paper-only. See
    \Cref{sec:lean-scope}.}}%
\begin{definition}[Genealogy tree]\label{def:genealogy}
A \emph{genealogy} for an heir $H \in \Heirs$ records which initial
particles merged to form $H$. It is an oriented tree $T$ embedded in
$D$ with:
\begin{itemize}
\item \textbf{Leaves}: starting positions $x_I$ for particles $I$
      that merge into $H$;
\item \textbf{Internal vertices}: merger points where two or more
      particles merge and continue as one;
\item \textbf{Root}: the final position $y_H$ of the heir;
\item \textbf{Edges}: directed paths in $D$---particle trajectories.
\end{itemize}
Each vertex $v$ of $T$ carries a \emph{label} $I_v$: the union of the
initial intervals of all particles that have merged by $v$ (an interval,
by \Cref{prop:consecutivity}).
When two initial particles share a starting vertex, that
vertex serves as both leaf and internal vertex; the
incoming paths have length zero.
\end{definition}

\fnote{%
  prose={Not formalized: the forest structure never appears in the
    machine-checked statements; see \Cref{sec:lean-scope}.}}%
\begin{definition}[Genealogy forest]\label{def:genealogy-forest}
A \emph{genealogy forest} is a collection $\mathcal{T} = \{T_H : H \in
\Heirs\}$ of genealogies satisfying:
\begin{itemize}
\item \textbf{Partition}: Every initial particle belongs to exactly
      one tree;
\item \textbf{Non-intersection}: Trees are vertex-disjoint.
\end{itemize}
\end{definition}

\fnote{%
  prose={The merged label at each vertex is an interval; for $V_D$ this is
    the consecutive-collision property, discharged from P2.},
  lean={Coalescence.Valuation.consecutive_collision_property},
  py={valuations.protocol}}%
\begin{proposition}[Consecutivity]\label{prop:consecutivity}
	Under the planarity assumption (\Cref{def:planar}), the label $I_v$ at
	each vertex is an interval: if particles $A$ and $C$ have both reached
	$v$, then so has every particle $B$ between them.
\end{proposition}

\begin{proof}
Suppose $A$ and $C$ both reach $v$, with starting positions
$x_A \preceq x_C$, and let $B$ be any particle with
$x_A \preceq x_B \preceq x_C$. By the partition property of the
genealogy forest, each of $A$, $B$, $C$ lies on a full trajectory
running from its leaf to the root of its tree. We must show that
$B$ also reaches $v$.

Since $A$ and $C$ both pass through $v$, we have $v \in A \cap C$;
as the trees are vertex-disjoint, $A$ and $C$ then lie in the same
tree and have merged into a common entity by $v$. The starting
order $x_A \preceq x_B \preceq x_C$ places the three trajectories in
the configuration of the consecutive collision
property~\ref{itm:P2-consecutive}, which provides a vertex $w$, no
later than $v$, that lies on $B$ and on $A\cup C$. Vertex-%
disjointness again applies: sharing the vertex~$w$ forces $B$ into
the same tree, merged with $A$ or $C$ by~$w$; say $w \in A$ (the case
$w \in C$ is symmetric). Now $w$ and $v$ both lie on the directed
path~$A$, hence are comparable in the time order~$\lessdot$; since
$\le$ extends~$\lessdot$, the relation $w \le v$ forces $w \lessdot v$
or $w = v$. The portion of~$A$ from~$w$ onward therefore passes
through~$v$, and~$B$, having joined~$A$ at~$w$, follows it to~$v$.

Hence every particle between $A$ and $C$ reaches $v$, and the
label $I_v$---the union of the initial intervals of the particles
merged by $v$---is an interval.
\end{proof}

\fnote{%
  prose={Ghost paths in the spacetime graph.},
  lean={Coalescence.Geometry.Spacetime.Path}}%
\begin{definition}[Ghost paths]\label{def:ghost-paths}
For each ghost $g \in \Ghosts$, let $c_g$ be the internal vertex where
junction $g$ was dissolved---the unique earliest vertex $v$ of the
genealogy forest whose label $I_v$ contains $g$ in its interior, so that
the particles on both sides of~$g$ have merged by~$v$; a simultaneous
multi-way merger may serve as $c_g$ for several ghosts at once. The
\emph{ghost path} $\Gamma_g$ is a directed path in $D$ from $c_g$ to
$y_g$. Ghost paths are non-interacting: they may pass through any
vertices freely.
\end{definition}

\fnote{%
  prose={A performance: a sequence of merges (a performance hop chain).},
  lean={Coalescence.PerformanceHopChain},
  py={combinatorics.chains.PerformanceHopChain}}%
\begin{definition}[Performance]\label{def:performance}
A \emph{performance} $\perf$ consists of:
\begin{itemize}
\item a genealogy forest $\mathcal{T}$;
\item a ghost path $\Gamma_g$ for each ghost $g \in \Ghosts$.
\end{itemize}
The \emph{weight} of the performance is the product of all edge weights:
\[
w(\perf) = \prod_{\text{tree edges } e} w(e) \cdot
           \prod_{g \in \Ghosts} w(\Gamma_g).
\]
\end{definition}

\subsection{Ghost sign}
\label{sec:setup-sign}

\fnote{%
  prose={The ghost L/R sign $\varepsilon$.},
  lean={Coalescence.GhostSign},
  py={combinatorics.domain.GhostSign}}%
\begin{definition}[Ghost sign]\label{def:ghost-sign}
The \emph{ghost sign} $\varepsilon\colon \Ghosts \to \{+1, -1\}$ is:
\[
\varepsilon_g = \begin{cases}
+1 & \text{if } y_g \preceq y_{\heir(g)} \text{ (ghost left of heir)}, \\
-1 & \text{if } y_{\heir(g)} \preceq y_g \text{ (ghost right of heir)}.
\end{cases}
\]
In the special case $y_g = y_{\heir(g)}$, either sign may be
chosen.
\end{definition}

Two distinct orderings appear in our setup (\Cref{fig:boundary}).
Actors and final entities carry \emph{labels}---intervals and
junctions ordered by~$\ilo$. But entities also have \emph{spatial
positions}---vertices in the spacetime graph ordered by~$\prec$.

In the classical LGV lemma, these two orderings coincide: the final
entities admit a single canonical left-to-right indexing, in which
label order and spatial order agree. The permutation~$\pi$ then has
the same meaning whether interpreted as permuting labels or as
crossing paths.

With ghosts, this alignment breaks. Heirs keep a canonical order---they
are sorted by spatial position, matching their label order---but a
ghost might end up to the left or the right of its heir, regardless of
their labels. The ghost sign~$\varepsilon$ captures precisely this
discrepancy.

\subsection{Total weight}
\label{sec:setup-total-weight}

Given fixed initial and final state, the total weight
\[ Z= \sum_{\perf} w(\perf)\]
is defined as the sum of weights over all performances with
this prescribed initial and final state.
Our goal in this paper is to give a closed formula for $Z$.

\section{The main result and an overview of its proof}
\label{sec:formula}
\subsection{The matrix}
\label{sec:formula-matrix}

Fix the initial and the final state.
Define the $n \times n$ matrix $M$ with:
\begin{itemize}
\item rows indexed by particles $I \in \Actors$ (in label order);
\item columns indexed by final entities $f \in \Roles$ (in
  $\min$-order: each heir followed by its attached ghosts);
\item entries:
\[
M_{I,f} = \begin{cases}
\phantom{-}W(x_I \to y_H) & \text{if } f = H \text{ is an heir}, 
   \\[0.3em]
-[\varepsilon_g=-1] \cdot W(x_I \to y_g) & \text{if } f = g \text{ is a ghost and }
I \ilo g,
\\[0.3em]
\phantom{-}[\varepsilon_g=+1] \cdot W(x_I \to y_g) & \text{if } f = g \text{ is a ghost and }
  I \igo g;
\end{cases}
\]
\end{itemize}
see \Cref{fig:matrix-structure} for an example.

\begin{remark}[Equivalence of notations]
Under the $\min$ identification
(\Cref{sec:setup-coalescence}), the label order condition $I
\ilo g$ is equivalent to $\min I < \min g$. Since $\min I$
and $\min g = g$ are the indexes of the row and column,
respectively, this recovers the coordinate formulation from
the introduction (\Cref{sec:intro-results}).
\end{remark}

\subsection{The theorem}
\label{sec:formula-theorem}

\fnote{%
  prose={The headline identity $Z=\det M$; the Lean proof factors it
    through three statements of decreasing generality.},
  lean={main_theorem_layer1, main_theorem_layer2, main_theorem_layer3},
  py={tests/test_determinant_identity.py}}%
\begin{theorem}[Coalescence formula with ghosts]\label{thm:coalescence}
Assume that the pair $(\mathcal{X}, \mathcal{Y})$ is planar
(\Cref{def:planar}). The total weight of all coalescence
performances with prescribed initial and final state is the
determinant
\begin{equation}
	\label{eq:coalescence-formula}
	Z = \det M.
\end{equation}
\end{theorem}

\noindent
This generalizes \Cref{thm:intro-coalescence}.

\begin{remark}[Edge cases of the boundary conditions]
\label{rem:formula-edge-cases}
Three boundary configurations refine the reading of
\Cref{thm:coalescence}. \emph{Weakly increasing heir positions:} the
ordering assumption on heir positions (\Cref{sec:setup-final}) is
weak, so the theorem covers two heirs sharing a final position. Such
heirs would coalesce immediately, so the configuration does not
arise; on the matrix side their columns coincide, so both sides of
\eqref{eq:coalescence-formula} vanish. \emph{A ghost coinciding with
its heir:} \Cref{def:ghost-sign} leaves $\varepsilon_g$ free when
$y_g = y_{\heir(g)}$, and $\det M$ does not depend on the choice:
passing from $\varepsilon_g = -1$ to $\varepsilon_g = +1$ amounts to
adding the heir column to the ghost column, a column operation.
\emph{Coinciding initial positions} $x_i = x_{i+1}$: particles
starting together coalesce immediately, so any realizable pattern
has a ghost at the junction between $i$ and~$i{+}1$. Unlike in the
non-colliding case, $\det M$ then need not vanish: the junction
places the two particles on opposite sides, so the two rows agree
on every heir column but differ on that ghost column, whose sign
selects a different entry in each row; instantaneous coalescence is
a legitimate pattern of generally non-zero weight. (For a pattern
without that ghost the rows coincide and both sides vanish.)
\end{remark}

\subsection{Original proof of the non-colliding case}
\label{sec:classic-proof}

The proof which we provide for \cref{thm:coalescence} is
adapted from the original papers
\cite{KM1959,Lindstrom1973,GV1985,Stembridge1990}
which concerned the special case of calculating the
probability (respectively, weight) of $n$ particles
(respectively, paths) reaching specified final positions in
a collision-free way. We shall recall the original proof in
a compressed form.

The key idea was to consider the Leibniz expansion of the
determinant $\det M$ as a sum over bijections $\pi$ between
the initial and the final positions. Any
contributing product of matrix $M$ entries (each counting
weights of single-particle trajectories) can be interpreted
as a sum over \emph{tuples} of trajectories
$\paths=(P_1,\dots,P_n)$ each connecting specific
initial and (permuted by $\pi$) final positions. In other
words, the determinant $\det M$ is expressed as a signed sum
over pairs $(\pi, \paths)$. Since such pairs play a
prominent role in this paper, we call them
\emph{castings} (\Cref{def:casting}).
 
The goal of the classic proof is to evaluate this signed sum by
cancellation. The planarity assumption guarantees that a trajectory
tuple $\paths$ without crossings can occur only for the identity
permutation $\pi=\id$, which the determinant expansion counts with
the positive sign $\sgn \pi=+1$.

The crossing tuples are removed by a \emph{sign-reversing involution}
on the set of castings: it pairs each crossing casting
$(\pi,\paths)$ with another crossing casting
$(\pi',\paths')$ of equal weight $w(\paths)=w(\paths')$
but opposite sign $\sgn \pi=-\sgn \pi'$, so that the two
contributions cancel. After all crossing castings cancel in such
pairs, only the non-crossing tuples survive, each carrying $\pi=\id$
and positive sign. The determinant $\det M$ therefore equals the
total weight of non-crossing trajectory tuples, which is the desired
count.

\subsection{What goes wrong for coalescing particles}
\label{sec:proof-overview}

The overall strategy survives: as in \cref{sec:classic-proof}, we
express the determinant $\det M$ as a signed sum over castings and
construct a sign-reversing involution which matches \emph{failed}
castings in pairs, so that their contributions cancel. Two features
of the coalescing setting, however, force adaptations---one minor,
one essential.

\subsubsection{A minor difference: $\varepsilon$-candidate castings}
The first difference concerns which bijections $\pi$ contribute.
In the classical setting the matrix has no forced zeros, so every
permutation $\pi$ can appear in the Leibniz expansion. Our matrix, by
contrast, has \emph{structural} zeros: each ghost column vanishes
except on the side selected by its sign, so a bijection contributes
only when it assigns every ghost an actor on the correct side. We call a
casting whose bijection meets this condition an \emph{$\varepsilon$-candidate}
casting (\Cref{def:candidate}), and restrict attention to these; see
\Cref{sec:proof-candidates}.

\subsubsection{The essential difference: performances are not castings}

The two sides of \eqref{eq:coalescence-formula} count different
kinds of objects. The left-hand side counts \emph{performances}. A
performance is \textbf{role-based}: at each binary coalescence two
particles arrive, meet, and disappear, replaced by an heir and
a ghost that emerge from the junction. Nothing threads what enters the
event to what leaves it.
This description says nothing about the identities behind the heir
and the ghost---only that the coalescence produced one of each. The
discontinuity is reflected in \Cref{fig:intro-21}: the line styles
entering the coalescence differ from those leaving it, emphasizing
that identities do not persist across the event. In theatrical
terms, the script fixes where each character stands in the final
scene, but not which actor plays which character.

\begin{figure}[t]
\centering

\pgfmathsetmacro{\arcrad}{0.35}

\begin{tikzpicture}[scale=0.55]
  \begin{scope}
    \clip (-0.2,-0.2) rectangle (11.2,6.2);

    \foreach \x in {-1,...,12} {
      \foreach \t in {-1,...,7} {
        \pgfmathparse{mod(\x+\t,2)==0 ? 1 : 0}
        \ifnum\pgfmathresult>0
          \fill[gray!30] (\x,\t) circle (1.5pt);
        \fi
      }
    }

    \foreach \x in {-2,0,2,4,6,8,10,12} {
      \foreach \t in {-2,0,2,4,6,8} {
        \draw[gray!30, thin] (\x,\t) -- (\x-1,\t+1);
        \draw[gray!30, thin] (\x,\t) -- (\x+1,\t+1);
      }
    }
    \foreach \x in {-1,1,3,5,7,9,11} {
      \foreach \t in {-1,1,3,5,7} {
        \draw[gray!30, thin] (\x,\t) -- (\x-1,\t+1);
        \draw[gray!30, thin] (\x,\t) -- (\x+1,\t+1);
      }
    }
  \end{scope}

  \draw[->, thick] (-0.3,0) -- (11,0) node[right] {$x$};
  \draw[->, thick] (0,-0.3) -- (0,6.7) node[above] {$t$};

  \draw (-0.1, 6) -- (0.1, 6);
  \node[left] at (-0.15,6) {\small $T$};

  \coordinate (collision) at (4,2);
  \coordinate (beforeArc) at (5,1);   %
  \coordinate (afterArc) at (3,3);    %

  \pgfmathanglebetweenpoints{\pgfpointanchor{collision}{center}}{\pgfpointanchor{beforeArc}{center}}
  \edef\angBin{\pgfmathresult}
  \pgfmathanglebetweenpoints{\pgfpointanchor{collision}{center}}{\pgfpointanchor{afterArc}{center}}
  \edef\angBout{\pgfmathresult}
  \pgfmathsetmacro{\arcEndB}{\angBout < \angBin ? \angBout + 360 : \angBout}

  \draw[particleC] (6,0) -- (7,1) -- (6,2) -- (5,3) -- (6,4) -- (7,5) -- (6,6);

  \draw[particleA] (2,0) -- (3,1) -- (collision);
  \draw[particleAg] (collision) -- (5,3);
  \draw[particleAg, transform canvas={shift={(0.1,-0.1)}}]
    (5,3) -- (6,4) -- (7,5);
  \draw[particleAg] (7,5) -- (8,6);

  \draw[colB, double, double distance=2pt, line width=0.6pt]
    (4,0) -- (beforeArc)
    -- ([shift={(\angBin:\arcrad)}]collision)
    arc[start angle=\angBin, end angle=\arcEndB, radius=\arcrad]
    -- (afterArc) -- (2,4) -- (3,5) -- (4,6);

  \node[circle, fill=white, inner sep=0pt, minimum size=5pt,
        line width=0.8pt, draw=black] at (collision) {};
  \node[left=0.2] at (collision) {\small $c$};

  \fill[colA] (2,0) circle (4pt);
  \fill[colB] (4,0) circle (4pt);
  \fill[colC] (6,0) circle (4pt);

  \node[below] at (2,-0.15) {\small $x_1$};
  \node[below] at (4,-0.15) {\small $x_2$};
  \node[below] at (6,-0.15) {\small $x_3$};

  \fill[colB] (4,6) circle (4pt);
  \fill[colC] (6,6) circle (4pt);
  \draw[colA, fill=white, line width=1.5pt] (8,6) circle (4pt);

  \node[above] at (4,6.15) {\small $y_H$};
  \node[above] at (6,6.15) {\small $y_{H'}$};
  \node[above] at (8,6.15) {\small $y_g$};

  \begin{scope}[shift={(11,1)}]
    \draw[particleA] (0,4.6) -- (0.7,4.6);
    \draw[particleAg] (0,4.2) -- (0.7,4.2);
    \node[right, align=left, inner sep=1pt] at (0.8,4.4)
      {\small particle $1$\\[-3pt]\small $\to$ ghost at $y_g$};
    \draw[particleB] (0,2.8) -- (0.7,2.8);
    \node[right, align=left, inner sep=1pt] at (0.8,2.8)
      {\small particle $2$\\[-3pt]\small $\to$ heir $1$ at $y_H$};
    \draw[particleC] (0,1.2) -- (0.7,1.2);
    \node[right, align=left, inner sep=1pt] at (0.8,1.2)
      {\small particle $3$\\[-3pt]\small $\to$ heir $2$ at $y_{H'}$};
  \end{scope}

\end{tikzpicture}

\caption{\textbf{Successful casting: actor-based view of a performance.}
This is a casting---a bijection $\pi$ from actors to roles together with the
paths realizing it---but it also displays a \emph{performance}: the same
coalescence as \Cref{fig:intro-21}, now with each actor's identity tracked
through the crossing. Here $\pi$ sends particle~$1$ (solid) to the ghost
at~$y_g$, particle~$2$ (double) to the heir at~$y_H$, and particle~$3$ (zigzag)
to the heir at~$y_{H'}$; the line styles persist past~$c$, revealing who went
where. The casting is \emph{successful}: the crossing at~$c$ is a valid
coalescence---one of the two meeting paths is destined for the ghost, so the
crossing is not spurious---and rehearsal reconstructs exactly the performance
of \Cref{fig:intro-21}.}
\label{fig:actor-bijection}
\end{figure}

\begin{figure}[t]
\begin{tikzpicture}[scale=0.55]
  \begin{scope}
    \clip (-0.2,-0.2) rectangle (11.2,6.2);

    \foreach \x in {-1,...,12} {
      \foreach \t in {-1,...,7} {
        \pgfmathparse{mod(\x+\t,2)==0 ? 1 : 0}
        \ifnum\pgfmathresult>0
          \fill[gray!30] (\x,\t) circle (1.5pt);
        \fi
      }
    }

    \foreach \x in {-2,0,2,4,6,8,10,12} {
      \foreach \t in {-2,0,2,4,6,8} {
        \draw[gray!30, thin] (\x,\t) -- (\x-1,\t+1);
        \draw[gray!30, thin] (\x,\t) -- (\x+1,\t+1);
      }
    }
    \foreach \x in {-1,1,3,5,7,9,11} {
      \foreach \t in {-1,1,3,5,7} {
        \draw[gray!30, thin] (\x,\t) -- (\x-1,\t+1);
        \draw[gray!30, thin] (\x,\t) -- (\x+1,\t+1);
      }
    }
  \end{scope}

  \draw[->, thick] (-0.3,0) -- (11,0) node[right] {$x$};
  \draw[->, thick] (0,-0.3) -- (0,6.7) node[above] {$t$};

  \draw (-0.1, 6) -- (0.1, 6);
  \node[left] at (-0.15,6) {\small $T$};

  \coordinate (spurious) at (4,2);

  \draw[particleC] (6,0) -- (7,1) -- (6,2) -- (5,3);
  \draw[particleC, transform canvas={shift={(0.1,-0.1)}}]
    (5,3) -- (6,4) -- (7,5);
  \draw[particleC] (7,5) -- (8,6);

  \draw[particleA] (2,0) -- (3,1) -- (spurious) -- (3,3) -- (2,4) -- (3,5) -- (4,6);

  \draw[particleB] (4,0) -- (5,1) -- (spurious) -- (5,3) -- (6,4) -- (7,5) -- (6,6);

  \draw[red, line width=2pt] (spurious) circle (0.4);
  \node[font=\large\bfseries, red, right=0.2] at (spurious) {!};
  \node[left=0.3] at (spurious) {\small $c$};

  \fill[colA] (2,0) circle (4pt);
  \fill[colB] (4,0) circle (4pt);
  \fill[colC] (6,0) circle (4pt);

  \node[below] at (2,-0.15) {\small $x_1$};
  \node[below] at (4,-0.15) {\small $x_2$};
  \node[below] at (6,-0.15) {\small $x_3$};

  \fill[colA] (4,6) circle (4pt);
  \fill[colB] (6,6) circle (4pt);
  \draw[colC, fill=white, line width=1.5pt] (8,6) circle (4pt);

  \node[above] at (4,6.15) {\small $y_H$};
  \node[above] at (6,6.15) {\small $y_{H'}$};
  \node[above] at (8,6.15) {\small $y_g$};

  \begin{scope}[shift={(11,0.5)}]
    \draw[particleA] (0,5.0) -- (0.7,5.0);
    \node[right, align=left, inner sep=1pt] at (0.8,5.0)
      {\small particle $1$\\[-3pt]\small $\to$ heir $1$ at $y_H$};
    \draw[particleB] (0,3.4) -- (0.7,3.4);
    \node[right, align=left, inner sep=1pt] at (0.8,3.4)
      {\small particle $2$\\[-3pt]\small $\to$ heir $2$ at $y_{H'}$};
    \draw[particleC] (0,1.8) -- (0.7,1.8);
    \node[right, align=left, inner sep=1pt] at (0.8,1.8)
      {\small particle $3$\\[-3pt]\small $\to$ ghost at $y_g$};
  \end{scope}

\end{tikzpicture}

\caption{\textbf{Failed casting: a spurious crossing.}
The same endpoints as \Cref{fig:actor-bijection}, realized by different paths
and a different bijection: here $\pi$ sends particle~$1$ to the heir at~$y_H$,
particle~$2$ to the heir at~$y_{H'}$, and particle~$3$ to the ghost at~$y_g$.
Scanning crossings in chronological order, rehearsal reaches~$c$, where the
paths of particles~$1$ and~$2$ meet; but neither is destined for the ghost
role---both end at heir positions---so the crossing cannot be read as a
coalescence. It is \emph{spurious}. Rehearsal therefore terminates and
produces no performance: its only output in this failed case is the spurious
crossing~$c$ together with the two actors that triggered it, particles~$1$
and~$2$. Segment swap takes exactly this data, exchanging the two actors' path
suffixes at~$c$ to pair the failed casting with a sign-opposite partner
(\Cref{sec:proof-swap}).}
\label{fig:failed-casting}

\end{figure}

The right-hand side, $\det M$, counts a different kind of object.
As in the classical argument (\cref{sec:classic-proof}), its Leibniz
expansion is a signed sum over castings $(\pi,\paths)$; 
but a casting is \textbf{actor-based}, tracking
each particle's complete trajectory, so in \Cref{fig:actor-bijection}
the line styles persist through the coalescence, recording who went
where.

The proof bridges the two descriptions. It establishes a
weight-preserving bijection between performances and the
\emph{successful} castings---those that reproduce a valid
coalescence pattern (\Cref{def:successful-casting})---while the
remaining \emph{failed} castings cancel in signed pairs
(\Cref{fig:failed-casting}). The remainder of this overview
describes the machinery that realizes this bijection and
cancellation.

\subsection{Adapting the proof}
\label{sec:formula-adapting}

Two maps bridge the two descriptions. \emph{Attribution}
(\Cref{sec:proof-attribution}) turns a performance into a casting by
tracking which actor ends where; it never fails, is injective, and its
output is always an $\varepsilon$-candidate casting
(\Cref{prop:attribution-candidate}). \emph{Rehearsal}
(\Cref{sec:rehearsal}) tests the converse: scanning the crossings of
an $\varepsilon$-candidate casting in chronological order, it either
processes every crossing as a coalescence and reproduces a
performance---the casting is \emph{successful}---or halts at a
\emph{spurious} crossing and reports its \emph{failure pair}---the
casting is \emph{failed}. The failed castings cancel in signed pairs:
the \emph{segment swap} (\Cref{sec:proof-swap}), applied at the
failure pair, yields the involution~$\iota$
(\Cref{def:global-involution}) whose fixed points are exactly the
successful castings (\Cref{thm:involution}), and these are in
weight-preserving bijection with performances
(\Cref{prop:performance-casting-bijection}).
\Cref{sec:proof-setup,sec:proof-attribution,sec:rehearsal,%
sec:proof-bijection,sec:involution,sec:proof-main} develop these
components in turn.

\section{Expanding the determinant}
\label{sec:proof-setup}
\label{sec:proof-castings}

Throughout the proof, we fix a final state~$\FinalState$
(\Cref{sec:setup-final}): the ghost set~$\Ghosts$, the heir
set~$\Heirs$, and final positions~$y_f$ for each $f \in \Roles$.
The total weight $Z = Z_{\FinalState}$ counts
\emph{performances} for~$\FinalState$. From
\Cref{sec:proof-attribution} onward we also assume, as in
\Cref{thm:coalescence}, that the pair $(\mathcal{X}, \mathcal{Y})$
is planar (\Cref{def:planar}); the present section is
planarity-free.

\Cref{sec:formula} named the two sides of \eqref{eq:coalescence-formula} as
performances and castings; this section makes the casting side precise
and equips each casting with the right sign. The Leibniz formula offers
the usual permutation sign $\sgn\pi$, but that sign is bound to the
arbitrary order we placed on the columns; we replace it by the
\emph{ghost-adjusted sign} $\gasgnof{\pi}$, read off the label order and the
ghost signs alone---the intrinsic data of the final state, blind to any
left--right convention. The two constructions combine in one clean
identity, the \emph{restricted Leibniz expansion}
(\Cref{prop:restricted-leibniz}): $\det M$ is the $\gasgn$-signed sum
over $\varepsilon$-candidate castings.

\subsection{The Leibniz expansion}

The Leibniz formula expands the determinant as a sum over bijections
$\pi\colon \Actors \to \Roles$ from actors to roles:
\[
\det M = \sum_{\pi} \sgn \pi \prod_{I \in \Actors} M_{I, \pi(I)}.
\]
Each matrix entry $M_{I,f}$ is either a weight
$W(x_I \to y_f)$ (for heir columns) or
$\pm [\varepsilon_g = \pm 1] \cdot W(x_I \to y_g)$ (for ghost
columns). Expanding:
\[
\det M = \sum_{\pi} \sgn \pi \sum_{\paths} (\text{sign factors})
\cdot w(\paths),
\]
where $\paths = \{P_I\}_{I \in \Actors}$ is a family of paths with
$P_I$ running from $x_I$ to $y_{\pi(I)}$. 

\fnote{%
  prose={A casting: a bijection $\pi$ together with a path family.},
  lean={Coalescence.Casting},
  py={combinatorics.casting.Casting}}%
\begin{definition}[Casting]\label{def:casting}
A \emph{casting} $(\pi, \paths)$ consists of:
\begin{itemize}
\item A bijection $\pi\colon \Actors \to \Roles$ from actors
  (initial particles) to roles (final entities);
\item A path family $\paths = \{P_I\}_{I \in \Actors}$ where $P_I$
  goes from $x_I$ to $y_{\pi(I)}$.
\end{itemize}
The \emph{weight} is the product of the path weights,
\[w(\paths) = \prod_I w(P_I).\]
The sign a casting carries is the ghost-adjusted sign defined below
(\Cref{eq:ghost-adjusted-sign}), not the bare permutation sign.
\end{definition}

Crucially, the path-family layer is \emph{pure geometry}: the
paths are non-interacting. They may cross freely---a crossing
is simply a point where two paths share a vertex, with no
physical consequence. There is no coalescence physics yet: the
determinant gives us $n$ non-interacting walkers, not an
interacting particle system.

\subsection{$\varepsilon$-candidate bijections}
\label{sec:proof-candidates}

Each ghost column carries forced zeros, and these decide which
bijections contribute. A ghost entry $M_{I,g}$ is non-zero only when
two conditions align. The first is \emph{geometric}: through the
bracket $[\varepsilon_g = \pm 1]$, the ghost sign selects one side of
the junction~$g$ and zeros every entry on the other side. The second
is \emph{combinatorial}: which side a given row lies on is fixed by
the label relation between the row's interval~$I$ and the
junction~$g$. A term in the Leibniz expansion therefore survives only
if, for every ghost, the \emph{performer} $\pi^{-1}(g)$---the actor
that $\pi$ casts in the role~$g$---lands on the side left non-zero by
the sign. We record this as a condition on~$\pi$.

\fnote{%
  prose={$\varepsilon$-candidate bijections: every performer lands on the
    sign-selected side of its junction. Lean:
    \texttt{Coalescence.PlainAssignment.Is}$\varepsilon$\texttt{Candidate}.},
  py={valuations.predicates}}%
\begin{definition}[$\varepsilon$-candidate bijection]\label{def:candidate}
A bijection $\pi\colon \Actors \to \Roles$ is an \emph{$\varepsilon$-candidate} if,
for every ghost $g \in \Ghosts$, the performer $\pi^{-1}(g)$ falls on
the side of the junction dictated by the ghost sign:
\[
g \ilo \pi^{-1}(g) \ \text{ if } \varepsilon_g = +1,
\qquad
\pi^{-1}(g) \ilo g \ \text{ if } \varepsilon_g = -1
\]
(equivalently, $g \ilo \pi^{-1}(g) \iff \varepsilon_g = +1$). We write
$\Candidates$ for the set of $\varepsilon$-candidate bijections.
\end{definition}

The ghost sign $\varepsilon_g$ compares the ghost's final
\emph{spatial position} to its heir's; $\varepsilon$-candidacy is the matching
condition on the performer's interval \emph{label} relative to the
junction.

The ghost-column brackets select exactly the $\varepsilon$-candidate bijections:
each ghost column vanishes off its sign-selected side, so only the
terms with $\pi \in \Candidates$ survive.

\subsection{The ghost-adjusted sign}

\subsubsection{Factoring out the ghost signs}

The $\varepsilon$-candidacy condition depends only on $\pi$, not on the paths
$\paths$, so the determinant expansion splits into a sign and a weight.
Fix an $\varepsilon$-candidate $\pi$. In each heir column the selected entry is
$W(x_I \to y_H)$; in each ghost column $\varepsilon$-candidacy places the performer
on the side left non-zero by the bracket $[\varepsilon_g = \pm 1]$,
which contributes the factor $\varepsilon_g$, so the selected entry is
$\varepsilon_g\, W(x_I \to y_g)$. Hence
\[
\prod_{I \in \Actors} M_{I, \pi(I)}
= \Bigl(\prod_{g \in \Ghosts} \varepsilon_g\Bigr)
  \prod_{I \in \Actors} W(x_I \to y_{\pi(I)}),
\]
and the Leibniz expansion gathers each casting's permutation sign and
these ghost signs into a single sign---the one the determinant attaches
to the casting. We define that sign without reference to any order on
the columns.

\subsubsection{The tournament $\iloeps$}

On the final entities, the label order $\ilo$ is only partial: by
\Cref{sec:setup-coalescence} its only incomparable pairs are an heir $H$
and a ghost $g$ lying in its interior (so $H = \heir(g)$). We extend it to a \emph{total}
relation $\iloeps$ by deciding exactly those
incomparable pairs from the side on which the ghost comes to rest
relative to its heir,
\[
g \iloeps H \ \text{ if } \varepsilon_g = -1,
\qquad \qquad
H \iloeps g \ \text{ if } \varepsilon_g = +1,
\]
while $f \iloeps f'$ agrees with $f \ilo f'$ on every comparable pair.
This placement runs against the spatial picture: a ghost coming to rest
to the \emph{left} of its heir ($\varepsilon_g = +1$) is ordered
\emph{after} it in $\iloeps$, while a ghost resting to the \emph{right}
($\varepsilon_g = -1$) is ordered \emph{before} it---each ghost sits, in
$\iloeps$, on the side of its heir \emph{opposite} to where it comes to
rest spatially.
The ghost sign $\varepsilon_g$ enters only as a name for ``which
side''; which side one calls positive will not matter. Every pair of
roles is now comparable, so $\iloeps$ is total; it need not, however, be
transitive---$\iloeps$ is a \emph{tournament}, not a linear order
(\Cref{rmk:tournament}).

\subsubsection{The sign as an inversion count}

Counting inversions needs only that every pair is comparable, not
transitivity. For a bijection $\pi$, an \emph{inversion} is a pair of
actors $I \ilo I'$ whose assigned roles $\iloeps$ reverses, that is, with
$\pi(I') \iloeps \pi(I)$; write $N(\pi)$ for their number,
\[
N(\pi) = \#\bigl\{\, I \ilo I' : \pi(I') \iloeps \pi(I) \,\bigr\}.
\]
Most are ordinary $\ilo$-inversions between comparable roles. The rest
come at most one per ghost: $g$ contributes an inversion exactly when its
performer and resting place fall on the same side---both left, or both
right---of its heir's performer and of the heir. One may read such an
inversion as a placement gone wrong; the inversion-free placement is
the \emph{far-side principle} that attribution will follow
(\Cref{sec:proof-attribution}).

The \emph{ghost-adjusted sign} of $\pi$ is then
\fnote{%
  prose={The ghost-adjusted sign $\gasgn$, defined as a tournament
    inversion count.},
  lean={Coalescence.PlainAssignment.sgnEps}}%
\begin{equation}\label{eq:ghost-adjusted-sign}
\gasgnof{\pi} \;:=\; (-1)^{N(\pi)}.
\end{equation}

\ifextendedversion
\begin{example}[The running example]\label{ex:gasgn-running}
Return to \Cref{ex:final-state} (\Cref{fig:genealogy,fig:boundary}): the
pattern $3{+}1$, with heirs $H = [1,4)$ and $H' = [4,5)$ and ghosts $2$
and~$3$, both interior to $H$. In \Cref{fig:genealogy} each ghost comes to
rest to the \emph{right} of its heir, so $\varepsilon_2 = \varepsilon_3 = -1$;
the tournament then places both ghosts before $H$, and $\iloeps$ is the linear
order
\[
2 \iloeps 3 \iloeps H \iloeps H'.
\]
For the bijection
$\pi\colon I_1 \mapsto H,\ I_2 \mapsto 3,\ I_3 \mapsto H',\ I_4 \mapsto 2$
of \Cref{fig:boundary}, the actor pairs $I \ilo I'$ with
$\pi(I') \iloeps \pi(I)$ are
\[
(I_1, I_2), \quad (I_1, I_4), \quad (I_2, I_4), \quad (I_3, I_4),
\]
four in all, so $N(\pi) = 4$ and $\gasgnof{\pi} = (-1)^{4} = +1$. This count
reads off the order $\iloeps$ alone---no $\min$, no permutation sign.
\Cref{lem:gasgn-eval} below gives the $\min$-based cross-check
$\gasgnof{\pi} = \sgn \pi \prod_g \varepsilon_g$: here $\pi$ is the $3$-cycle
$(2\,3\,4)$, of sign $+1$, and $\varepsilon_2 \varepsilon_3 = +1$, in
agreement.
\end{example}
\fi

\fnote{%
  prose={Evaluation $\gasgnof{\pi}=\sgn\pi\prod_g\varepsilon_g$; Python
    brute-forces it over every bijection.},
  lean={Coalescence.PlainAssignment.sgnEps_eq_sgn_mul_eps},
  py={tests/test_sgn_eps.py}}%
\begin{lemma}[Evaluating the ghost-adjusted sign]\label{lem:gasgn-eval}
For every bijection $\pi\colon \Actors \to \Roles$,
\[
\gasgnof{\pi} = \sgn \pi \prod_{g \in \Ghosts} \varepsilon_g.
\]
\end{lemma}

\begin{proof}
Let $N_{\min}(\pi) = \#\{I \ilo I' : \min \pi(I') < \min \pi(I)\}$ count
the inversions of $\pi$ in the $\min$-order, so that
$\sgn \pi = (-1)^{N_{\min}(\pi)}$. The $\min$-order is a linear
extension of $\ilo$, so $\iloeps$ and the $\min$-order agree on every
comparable pair; they differ only on the incomparable pairs, and there
only for a ghost $g$ with $\varepsilon_g = -1$, where $g \iloeps \heir(g)$
while $g > \min \heir(g)$ places $g$ after its heir. Each such ghost
flips exactly one inversion, so
$N(\pi) \equiv N_{\min}(\pi) + \#\{g : \varepsilon_g = -1\} \pmod 2$.
Hence $\gasgnof{\pi} = (-1)^{N(\pi)} = (-1)^{N_{\min}(\pi)}
\prod_{g} \varepsilon_g = \sgn \pi \prod_{g} \varepsilon_g$.
\end{proof}

\subsubsection{The restricted Leibniz expansion}

\fnote{%
  prose={The restricted Leibniz expansion $\det M$ over
    $\varepsilon$-candidate castings; the algebraic core of layer~2.},
  lean={Coalescence.MainTheorem.main_theorem_layer2, Coalescence.MainTheorem.M_V}}%
\begin{proposition}[Restricted Leibniz expansion]\label{prop:restricted-leibniz}
For fixed initial and final state, the determinant is a signed sum over
$\varepsilon$-candidate castings, with the ghost-column signs absorbed into the
ghost-adjusted sign:
\begin{equation}\label{eq:restricted-leibniz}
\det M
= \sum_{\pi \in \Candidates} \gasgnof{\pi}
  \prod_{I \in \Actors} W(x_I \to y_{\pi(I)})
= \sum_{\substack{(\pi, \paths) \\ \pi \in \Candidates}}
  \gasgnof{\pi} \cdot w(\paths).
\end{equation}
\end{proposition}

\begin{proof}
The Leibniz formula sums over all bijections~$\pi$, but each ghost column
vanishes off its sign-selected side, so only the $\varepsilon$-candidates survive
(\Cref{def:candidate}). For an $\varepsilon$-candidate the column product factors as
$\prod_I M_{I,\pi(I)} = \bigl(\prod_g \varepsilon_g\bigr)
\prod_I W(x_I \to y_{\pi(I)})$, and $\sgn\pi \cdot \prod_g \varepsilon_g =
\gasgnof{\pi}$ by \Cref{lem:gasgn-eval}; expanding each weight
$W$ as a sum over paths turns the product into $\sum_\paths w(\paths)$.
\end{proof}

This identity is the starting point for the rest of the proof, which
evaluates the signed sum by cancellation.

\fnote{%
  prose={The tournament $\iloeps$ is total but not transitive.},
  lean={Coalescence.Entity.ltEps}}%
\begin{remark}[A tournament, not an order]\label{rmk:tournament}
The failure of transitivity is concrete. For the heir $H = [1,4)$
carrying ghosts $2, 3$ with $\varepsilon_2 = +1$ and
$\varepsilon_3 = -1$, the comparisons $3 \iloeps H \iloeps 2 \iloeps 3$
close a cycle, so $\iloeps$ is a genuine tournament with no underlying
linear order. No linear order of the roles reproduces $\gasgn$; what is
well defined is the inversion count $N(\pi)$, which tallies pairs and
needs no transitivity, and with it the sign
\eqref{eq:ghost-adjusted-sign}. This is why one cannot
reorder the columns to remove the ghost signs, and why the $\min$-order
of the introduction is only a device for writing the determinant down.
That device costs symmetry: $\min$ selects the left endpoint of each
interval, a left--right choice, whereas $\iloeps$ is built only from the
label order and the side each ghost rests on, which simply reverse when
the line is reflected. The restricted Leibniz expansion is in this sense
even-handed---it privileges neither end of the line, naming no positive
side---where the determinant, to gain a compact formula, had to.
\end{remark}

Finally, $\gasgn$ inherits the one property of the permutation sign that
the sign-reversing involution of \Cref{sec:involution} relies on.

\fnote{%
  prose={Composing with a transposition negates the ghost-adjusted sign.},
  lean={Coalescence.PlainAssignment.sgnEps_composeWithActorSwap}}%
\begin{corollary}[Transposition rule]\label{cor:gasgn-transposition}
Composing $\pi$ with a transposition negates the ghost-adjusted sign:
\[
\gasgnof{(I\;J) \circ \pi} = -\gasgnof{\pi}.
\]
\end{corollary}

\begin{proof}
Since $\prod_g \varepsilon_g$ is a constant, \Cref{lem:gasgn-eval} gives
$\gasgn = \sgn \cdot \prod_g \varepsilon_g$, and a transposition negates
$\sgn$.
\end{proof}

\section{Attribution: performance to casting}
\label{sec:proof-attribution}

A performance fixes the genealogy of the
coalescences and the ghost paths, but it does not record which initial
actor ends up at which final entity. \emph{Attribution} supplies that
missing correspondence---the casting underlying the performance.

Concretely, attribution follows each initial particle through the
performance one coalescence at a time. At each binary coalescence two
incoming particles meet; the rule below sends one onward as the heir
and the other to the ghost just created. Carrying out this rule at
every coalescence both names the final entity each particle
reaches---the bijection $\pi$ that the casting carries---and glues, in
parallel, the path each particle travels. We describe the rule at a
single coalescence, iterate it to define attribution and read off that
it yields an $\varepsilon$-candidate bijection.

\subsection{The two-particle case}
\label{sec:two-particle}

We first describe what happens at a single coalescence of two
particles, then explain how to combine these local operations.

By consecutivity (\Cref{prop:consecutivity}), coalescing particles
are always adjacent intervals. Write $I^-$ and $I^+$ for the left
and right incoming intervals at a coalescence, $H$ for the heir,
and $g$ for the
ghost. (In interval notation: $I^- = [a,g)$, $I^+ = [g,c)$, $H = [a,c)$.)

Four path segments meet at the coalescence vertex
(\Cref{fig:simplified-coalescence-a,fig:simplified-coalescence-neg-a}):
two incoming and two outgoing (one to the ghost position $y_g$,
one continuing as the heir toward further coalescences or the final
position).
The \emph{far-side principle} connects each incoming particle to the
outgoing segment on the opposite side: a particle arriving from the
left interval~$I^-$ leaves to the right, and one arriving from the
right interval~$I^+$ leaves to the left.

The ghost sign $\varepsilon_g$ encodes whether the ghost ends left
or right of the heir, determining the gluing
(\Cref{fig:simplified-coalescence-b,fig:simplified-coalescence-neg-b}):
\begin{itemize}
\item $\varepsilon_g = +1$ (ghost left of heir): the right interval
  $I^+$ becomes the ghost;
\item $\varepsilon_g = -1$ (ghost right of heir): the left interval
  $I^-$ becomes the ghost.
\end{itemize}
This defines a \emph{local bijection} $\lambda\colon\{I^-, I^+\}\to \{H, g\}$
(\Cref{fig:simplified-coalescence-label,fig:simplified-coalescence-neg-label}).
The two cases look asymmetric, but each is the inversion-free choice for
the tournament~$\iloeps$: the far-side principle draws the ghost from the
side of its heir opposite where it comes to rest, the placement that
contributes no $\iloeps$-inversion at~$g$. So the local bijection counts
zero inversions and, by \eqref{eq:ghost-adjusted-sign}, a single
coalescence is ghost-adjusted-sign-neutral either way:
\[
\gasgnof{\lambda} = +1.
\]
Across all the coalescences these neutral factors multiply to
$\gasgnof{\pi} = +1$ for the whole casting
(\Cref{prop:attribution-positive}).

\subfile{../figures/fig-simplified-coalescence}

\subsection{From local to global}

\subsubsection{The attribution map}

\fnote{%
  prose={Attribution: the casting underlying a performance (the inverse of
    casting), read off hop by hop.},
  lean={Coalescence.attribute_chain},
  py={combinatorics.attribution.attribute_chain}}%
\begin{definition}[Attribution]\label{def:attribution}
\emph{Attribution} constructs a casting $(\pi, \paths)$ from a
performance by applying the local gluing rule at each coalescence.
\end{definition}

We describe attribution assuming every coalescence is binary;
\Cref{sec:proof-high-indegree} reduces the general case to this one.
For each initial particle $I$, follow its path through the binary
coalescences, applying the local gluing rule at each.
The path
terminates either at an heir position or at a ghost position. This
process produces:
\begin{itemize}
\item the endpoint $\pi(I)$ (the final entity where $I$ ends up);
\item the glued path $P_I$ from $x_I$ to $y_{\pi(I)}$.
\end{itemize}
Iterating over all initial particles yields the output of attribution:
a casting $(\pi, (P_I)_{I \in \Actors})$.

\begin{example}[Single coalescence: \Cref{fig:intro-21}]
In \Cref{fig:intro-21}, particles $I_1$ and $I_2$ coalesce at~$c$.
The ghost ends at $y_g > y_H$, so $\varepsilon_g = -1$ (ghost right of
heir). By the far-side principle, the left particle $I_1$ becomes the ghost.
The output is the casting $(\pi, (P_{I_1}, P_{I_2}, P_{I_3}))$ where:
\begin{itemize}
\item $P_{I_1}$: from $x_1$ to $c$, then to $y_g$; endpoint $\pi(I_1) = g$;
\item $P_{I_2}$: from $x_2$ to $c$, then to $y_H$; endpoint $\pi(I_2) = H$;
\item $P_{I_3}$: from $x_3$ to $y_{H'}$ (no coalescence);
  endpoint $\pi(I_3) = H'$.
\end{itemize}
See \Cref{fig:actor-bijection}.
\end{example}

\subsubsection{High-indegree vertices}
\label{sec:proof-high-indegree}
It remains to lift the binary assumption. When $r > 2$ intervals meet
at one vertex they form a consecutive run $J_1, \ldots, J_r$
(consecutivity); process it by nested pairings in label order,
$((J_1\, J_2)\, J_3) \cdots J_r$, each pairing producing one ghost and a
continuing heir, which completes the definition of attribution. This
left-to-right order is also the one rehearsal uses
(\Cref{sec:rehearsal-algorithm}), so attribution and rehearsal agree.

\begin{example}[Full coalescence: \Cref{fig:n3-pi1,fig:n3-pi2}]
\label{ex:attribution-n3}
Three particles $I_1$, $I_2$, $I_3$ all coalesce into one heir
$H = [1,4)$, with ghost signs $\varepsilon_2 = +1$ (ghost~$2$ left of
heir) and $\varepsilon_3 = -1$ (ghost~$3$ right of heir).

In \Cref{fig:n3-pi1}, junction~$3$ fires first: $I_2$ and $I_3$ meet.
Since $\varepsilon_3 = -1$, the left particle $I_2$ becomes the ghost.
Next, junction~$2$ fires: $I_1$ meets the merged interval (carried by
$I_3$). Since $\varepsilon_2 = +1$, the right particle ($I_3$) becomes
the ghost. The output is the casting $(\pi_1, (P_{I_1}, P_{I_2}, P_{I_3}))$:
\begin{itemize}
\item $P_{I_1}$: from $x_1$ to junction~$2$, then to $y_H$;
  endpoint $\pi_1(I_1) = H$;
\item $P_{I_2}$: from $x_2$ to junction~$3$, then to $y_3$;
  endpoint $\pi_1(I_2) = 3$;
\item $P_{I_3}$: from $x_3$ to junction~$3$, to junction~$2$, then to $y_2$;
  endpoint $\pi_1(I_3) = 2$.
\end{itemize}

In \Cref{fig:n3-pi2}, the same final state arises from a different
coalescence order (junction~$2$ fires first), producing a different
casting $(\pi_2, (P'_{I_1}, P'_{I_2}, P'_{I_3}))$ with
$\pi_2(I_1) = 3$, $\pi_2(I_2) = 2$, $\pi_2(I_3) = H$.
\end{example}

\subfile{../figures/fig-n3-example}

\subsection{Assignments: a bookkeeping device}
\label{sec:proof-assignments}

To recast attribution---and, later, to run rehearsal
(\Cref{sec:rehearsal})---we record one compact device. An
\emph{assignment} captures a coalescence in progress: a partition of
the initial particles into their current clusters, paired with a
bijection matching actors to the resulting entities. This is the
actor-to-entity language of a casting's bijection $\pi$
(\Cref{def:casting}), with all spacetime and path weights forgotten.
Attribution reaches, at the final state, an assignment whose bijection
is exactly that~$\pi$; the construction below builds it one coalescence
at a time, starting from the identity assignment.

\subsubsection{Partitions}
\begin{definition}[Partition]\label{def:partition}
A \emph{partition}~$P$ records a coalescence pattern by:
\begin{itemize}
	\item its \emph{active intervals} $\Active(P)$---a list of half-open
	intervals whose disjoint union is $[1, n+1)$, recording the
	current clusters of initial particles;
	\item its \emph{ghost junctions} $\Ghosts(P) \subseteq \Junctions$,
	recording which interior junctions have been freed by past
	coalescences.
\end{itemize}
\end{definition}
The \emph{entities} of $P$ are $\Active(P) \cup \Ghosts(P)$. (We
reserve the names \emph{heir} and $\Heirs$ for the final state; the active
intervals of a general partition are transient clusters.)

Two extremal partitions appear repeatedly: the \emph{initial
	partition}, whose active intervals are the unit intervals $I_1,
\ldots, I_n$ and whose ghost set is empty; and the
\emph{final-state partition} $P_\FinalState$, whose active intervals
are the heir intervals $\Heirs$ from \Cref{sec:setup-final} and
whose ghost set is $\Ghosts$. The entities of $P_\FinalState$ are
exactly the roles $\Roles$, so the notations
$\Active(P_\FinalState) = \Heirs$ and
$\Ghosts(P_\FinalState) = \Ghosts$ are consistent.

A coalescence merges two adjacent active intervals into their union
and adds the junction between them to the ghost set; all other
intervals and ghosts are unchanged. Every partition in the proof
arises this way, from the initial partition; in particular every
junction lying strictly inside an active interval is a ghost---a
boundary dissolved by an earlier merge.

\subsubsection{Assignments and diagonality}
\fnote{%
  prose={The heir/ghost assignment: a partition together with a bijection
    from actors to its entities.},
  lean={Coalescence.PlainAssignment}}%
\begin{definition}[Assignment]\label{def:assignment}
	An \emph{assignment} is a pair $\pi = (P, \pi)$ where $P$ is a
	partition and
	\[
	\pi\colon \Actors \to \Active(P) \cup \Ghosts(P)
	\]
	is a bijection from actors to entities of $P$. We write $\pi$ both
	for the pair and for the bijection; the underlying partition is
	recovered as $P_\pi$ when needed.
\end{definition}

\fnote{%
  prose={The diagonal (initial actor-to-interval) assignment.},
  lean={Coalescence.DiagonalAssignment},
  py={combinatorics.diagonal_assignment.DiagonalAssignment}}%
\begin{definition}[Diagonal assignment]\label{def:diagonal}
An assignment $\pi = (P, \pi)$ is \emph{diagonal} if, for every active
interval $I = [a, b)$ of $P$, the bijection $\pi$ restricts to a
bijection
\[
\{\,\text{actors with label in } [a, b)\,\}
\;\xrightarrow{\ \sim\ }\;
\{I\} \cup \{\, g \in \Ghosts(P) : a < g < b \,\}
\]
from the $b - a$ actors labelled in $I$ onto $I$ together with the
ghost junctions interior to~$I$. The \emph{heir actor} of $I$
under~$\pi$ is the unique actor sent to~$I$, namely $\pi^{-1}(I)$;
every other actor labelled in~$I$ performs an interior ghost.
\end{definition}

The \emph{identity assignment} pairs the initial partition with the
bijection $I_j \mapsto I_j$. It is trivially diagonal: each unit
interval contains a single actor and no interior junctions, so the
required bijection sends that lone actor to its interval. A bijection
$\pi\colon \Actors \to \Roles$ pairs with the final-state partition
$P_\FinalState$ to give an assignment $\pi = (P_\FinalState, \pi)$;
this is how the $\varepsilon$-candidate bijections of \Cref{def:candidate} appear as
assignments at $P_\FinalState$.

Diagonality thus tiles each active interval: the $b - a$ actors
labelled in $I = [a, b)$ are partitioned into its heir actor and the
performers $\pi^{-1}(g')$ of its $b - a - 1$ interior ghosts (every
junction interior to an active interval is a ghost, by the partition
structure above). The rigidity argument of
\Cref{sec:rigidity-under-planarity} uses this tiling directly.

\subsubsection{The assignment poset}

\fnote{%
  prose={The assignment poset order: $\pi'$ extends $\pi$ by further
    coalescences, fixing the performer of every ghost $\pi$ already has.},
  lean={Coalescence.IntervalPartition.Coarsens}}%
\begin{definition}[The assignment preorder]
	\label{def:assignment-order}
	For assignments $\pi$ and $\pi'$, write $\pi \leq \pi'$ if:
	\begin{enumerate}[label=(\roman*)]
		\item the partition $P_{\pi'}$ is coarser than $P_\pi$
		(equivalently, $\Ghosts(P_\pi) \subseteq \Ghosts(P_{\pi'})$);
		\item for every ghost $g \in \Ghosts(P_\pi)$,
		$\pi'^{-1}(g) = \pi^{-1}(g)$.
	\end{enumerate}
\end{definition}

Reading: $\pi'$ extends $\pi$ by further coalescences without
revising the performer $\pi^{-1}(g)$ of any ghost $\pi$ already has. The
identity assignment is the minimum element; every
assignment lies above it. Strictly, $\leq$ is a preorder---%
antisymmetry can fail for assignments that differ only on their
active-interval images---though on diagonal assignments it is a
genuine partial order. No argument below uses antisymmetry, and we
keep the customary shorthand \emph{assignment poset}.

\subsection{Attribution as a chain}

\fnote{%
  prose={Attribution as an increasing chain of diagonal assignments, one
    binary coalescence per step.},
  lean={Coalescence.attribute_chain},
  py={combinatorics.attribution.attribute_chain}}%
\begin{proposition}[Attribution as a chain of diagonal assignments]
\label{prop:attribution-chain}
Attribution produces an increasing chain in the assignment poset,
from the identity assignment to the diagonal assignment
$(P_\FinalState, \pi)$, advancing by one binary coalescence per
step. Every assignment in the chain is diagonal, and the bijection of the
final assignment is the bijection $\pi$ that the casting
carries (\Cref{def:casting}).
\end{proposition}
\begin{proof}
This is immediate from the construction above. Each binary coalescence
merges two adjacent active intervals $I^-$ and $I^+$ into $H$ and frees
the junction $g$ between them; the far-side principle makes one of the
two incoming actors the performer of the new ghost $g$ and lets the
other carry the merged interval, so the actors labelled in $H$ are
again in bijection with $H$ and its interior ghosts---the assignment
stays diagonal. The step also moves \emph{up} the poset
(\Cref{def:assignment-order}): it only coarsens the partition and adds
the new ghost, while every actor already mapped to a ghost keeps that
image, so the performers already assigned are preserved.
\end{proof}

\subsection{$\varepsilon$-candidacy emerges}

The chain exposes two properties of the casting, one read off each
coalescence. The first is $\varepsilon$-candidacy. The right-hand side of the formula
is the restricted Leibniz expansion (\Cref{prop:restricted-leibniz}), a
sum over \emph{$\varepsilon$-candidate} castings; for the casting attribution produces
to appear there at all---rather than be silently absent---its bijection
must be an $\varepsilon$-candidate. It is.

\fnote{%
  prose={Attribution always yields an $\varepsilon$-candidate; the round-trip
    casting is provably one in
    \texttt{Combinatorics/Operations/Inverse.lean}.},
  lean={Coalescence.canonical_casting}}%
\begin{proposition}[Attribution yields an $\varepsilon$-candidate]\label{prop:attribution-candidate}
Let $\casting = (\pi, \paths)$ be the casting that attribution produces
from a performance. Then $\pi$ is an $\varepsilon$-candidate.
\end{proposition}

\begin{proof}
Fix a ghost $g$, created when the adjacent intervals $I^- = [a, g)$
and $I^+ = [g, c)$ merge. The interval $[a, c)$ is assembled from the
initial actors $I_a, \ldots, I_{c-1}$, so the particle arriving at
this coalescence along $I^-$ traces back to an initial actor whose
label lies in $[a, g)$, and the one arriving along $I^+$ to an
initial actor whose label lies in $[g, c)$. The far-side principle
hands the ghost role to one of the two incoming intervals according
to the ghost sign, and the ghost is performed by that interval's
incoming particle.

If $\varepsilon_g = -1$ (ghost right of heir), the ghost comes from
the left interval $I^- = [a, g)$, so its performer $\pi^{-1}(g)$ has
label in $[a, g)$; thus $\pi^{-1}(g) \ilo g$, the side
\Cref{def:candidate} requires when $\varepsilon_g = -1$. If
$\varepsilon_g = +1$ (ghost left of heir), the ghost comes from the
right interval $I^+ = [g, c)$, so $\pi^{-1}(g)$ has label in
$[g, c)$; thus $g \ilo \pi^{-1}(g)$, the side required when
$\varepsilon_g = +1$. In both cases the performer falls on the side
the ghost sign dictates, so $\pi$ is an $\varepsilon$-candidate.
\end{proof}

\subsection{Positivity}

The second property is the sign. $\varepsilon$-candidacy reads, off each
coalescence, which \emph{side} the ghost is drawn from; the sign is the
companion
reading---each coalescence moves the inversion count only by an even
amount, so the count stays even and the casting is \emph{positive}.

\fnote{%
  prose={Attribution's output is positive (ghost-adjusted sign $+1$); part
    of the layer-1 sign identity.},
  lean={Coalescence.MainTheorem.main_theorem_layer1}}%
\begin{proposition}[Attribution yields a positive casting]
\label{prop:attribution-positive}
The casting $(\pi, \paths)$ that attribution produces from a performance
has ghost-adjusted sign $\gasgnof{\pi} = +1$.
\end{proposition}

We prove this at the end of the subsection, after isolating its one
nontrivial step (\Cref{lem:coalescence-inversion-parity}). The proof
tracks the inversion count $N$ of \eqref{eq:ghost-adjusted-sign} along
the chain, defined for every assignment $\sigma = (P, \rho)$, not
only the final casting: $\iloeps$ extends to the entities of $P$ by the
same rule used on the final roles---active intervals play the part of
heirs, and an active interval and a ghost in its interior are ordered by
the ghost's sign---and $N(\sigma) = \#\{\, I \ilo I' : \rho(I') \iloeps
\rho(I) \,\}$.

\fnote{%
  prose={A single coalescence changes the inversion count $N$ by an even
    amount; the parity step behind positivity.}}%
\begin{lemma}[A coalescence preserves inversion parity]
\label{lem:coalescence-inversion-parity}
Let $\sigma \to \sigma'$ be one step of the attribution chain---a
coalescence merging adjacent active intervals $L = [a, g)$ and
$R = [g, c)$ into $H = [a, c)$ and creating the ghost $g$. Then
$N(\sigma') \equiv N(\sigma) \pmod 2$.
\end{lemma}

\begin{proof}
The actors are the fixed unit intervals, and the step changes only the
images of the two actors $p_L = \sigma^{-1}(L)$ and
$p_R = \sigma^{-1}(R)$, sending them to the new entities $H$ and $g$; every
other actor keeps its image, and the $\iloeps$-order among all surviving
entities is unchanged. So $N$ can change only through pairs of actors
meeting $\{p_L, p_R\}$: the pair $(p_L, p_R)$, and the two pairs each other
actor forms with $p_L$ and $p_R$.

Write $[\,\cdot\,] \in \{0,1\}$ for a truth value. A pair $\{x, p\}$ is an
inversion iff its label order and the $\iloeps$-order of its images
disagree, that is, iff $[x \ilo p] + [\sigma(x) \iloeps \sigma(p)]$ is odd.

\medskip\noindent\emph{The pair $(p_L, p_R)$.} The far-side principle is the
inversion-free choice for $\iloeps$ (\Cref{sec:two-particle}; the
bijections $\lambda_\pm$ are drawn in
\Cref{fig:simplified-coalescence-neg-label,fig:simplified-coalescence-label}),
so $\sigma'(p_L) \iloeps \sigma'(p_R)$ holds just as
$\sigma(p_L) = L \iloeps R = \sigma(p_R)$; with the label order of
$p_L, p_R$ fixed, the pair keeps its inversion status.

\medskip\noindent\emph{The pairs through another actor $x$.} Let
$e = \sigma(x) = \sigma'(x)$ be its unchanged image. 
Write $P(\rho)$ for the number of inversions among $\{x, p_L\}$ and
$\{x, p_R\}$ under an assignment $\rho$. By the rule above,
\[
  P(\rho) \;\equiv\; [x \ilo p_L] + [x \ilo p_R]
  + [e \iloeps \rho(p_L)] + [e \iloeps \rho(p_R)] \pmod 2 .
\]
Here $P(\sigma)$ is the count before the step, with
$\{\sigma(p_L), \sigma(p_R)\} = \{L, R\}$, and $P(\sigma')$ the count after,
with $\{\sigma'(p_L), \sigma'(p_R)\} = \{H, g\}$; the label terms and the
image $e = \sigma(x) = \sigma'(x)$ are common to both. So the change across
the step is
\[
  P(\sigma') - P(\sigma) \;\equiv\;
  [e \iloeps L] + [e \iloeps R] + [e \iloeps H] + [e \iloeps g] \pmod 2 ,
\]
which is zero: \emph{for every entity $e$,}
\begin{equation}\label{eq:parity-identity}
  [e \iloeps L] + [e \iloeps R] \;\equiv\; [e \iloeps H] + [e \iloeps g]
  \pmod 2 .
\end{equation}
Indeed, if $e$ lies to one $\iloeps$-side of the block $[a, c)$, then
$L, R, H, g$ are all on the other side of $e$, and both sides of
\eqref{eq:parity-identity} count $0$ or $2$. Otherwise $e$ is a ghost
inside $L$ or $R$; then $[e \iloeps H] = [e \iloeps(\text{its container})]$,
both settled by $\varepsilon_e$, and
$[e \iloeps g] = [e \iloeps(\text{the other interval})]$, both settled by
position---so \eqref{eq:parity-identity} holds term by term.

Hence every actor $x \neq p_L, p_R$ contributes an even change to $N$, and
$(p_L, p_R)$ contributes none: $N(\sigma') \equiv N(\sigma) \pmod 2$.
\end{proof}

\begin{proof}[Proof of \Cref{prop:attribution-positive}]
Attribution builds $\pi$ as a chain of assignments from the identity
assignment up to $(P_\FinalState, \pi)$ (\Cref{prop:attribution-chain}).
At the identity assignment the entities are the unit intervals, $\iloeps$
restricts to the linear label order, and $I_j \mapsto I_j$ has no
inversions, so $N = 0$. Each coalescence preserves the parity of $N$
(\Cref{lem:coalescence-inversion-parity}), so $N(\pi)$ is even and
\[
\gasgnof{\pi} = (-1)^{N(\pi)} = +1. \qedhere
\]
\end{proof}

$\varepsilon$-candidacy and positivity together fix attribution's role in the proof:
the restricted Leibniz expansion counts each performance exactly once and
with sign $+1$, so every object we want is present, on the positive side.
What they do \emph{not} say is that the expansion counts nothing else: its
$\varepsilon$-candidate sum also carries castings that attribution never produces.
Sorting those out is the work that remains---rehearsal
(\Cref{sec:rehearsal}) and the sign-reversing involution it feeds
(\Cref{sec:involution}).

\section{Rehearsal}\label{sec:rehearsal}

\emph{Rehearsal} reverses attribution: from a casting it tries to recover a
performance that attribution would turn into
that casting. The casting's bijection~$\pi$ is fixed in advance, and
rehearsal rebuilds the performance one coalescence at a time
(\Cref{sec:rehearsal-algorithm}). The run \emph{succeeds} when a
performance is recovered and \emph{fails} at a \emph{spurious}
crossing---a coalescence the reconstruction cannot carry out. This
success/failure split is what the sign-reversing involution
(\Cref{sec:involution}) exploits: the failed castings cancel in signed
pairs, the successful ones survive.

\subsection{The rehearsal algorithm}
\label{sec:rehearsal-algorithm}

\subsubsection{Aspiration and reality}

Rehearsal is read most easily as a climb in the assignment poset
(\Cref{def:assignment-order}). The input is an $\varepsilon$-candidate
casting; its bijection~$\pi$, viewed as the assignment
$(P_\FinalState, \pi)$ (\Cref{def:assignment}), sits at the top of the climb. We call it the
\emph{aspiration}: it records, for each actor~$I$, the entity~$\pi(I)$
that $I$ should reach. Starting from the identity assignment, rehearsal builds a second assignment upward toward it, one binary
coalescence at a time:
\begin{itemize}
\item $\sigma$: the \emph{reality}, a running diagonal assignment
  (\Cref{def:diagonal}) initialized as the identity assignment and updated
  at each accepted coalescence; its partition $P_\sigma$ coarsens by one
  binary merge per step, and its bijection records the performer of each
  ghost fired so far;
\item $\perf$: the performance under construction (coalescence events and
  ghost paths recorded so far).
\end{itemize}
This is the climb of \Cref{prop:attribution-chain} with the two sides
exchanged. Attribution built $\sigma$ freely and let $\pi$ emerge as its
final value; rehearsal holds $\pi$ fixed and asks whether $\sigma$ can
reach it. Throughout the run the reality stays below the aspiration,
$\sigma \leq \pi$ (\Cref{lem:sigma-le-pi}). So $\sigma$ never
overshoots $\pi$: the run either reaches $\pi$ (success) or halts earlier
at a spurious crossing (failure). A third conceivable outcome---the run
completing without failure, no crossing left to process, yet with
$\sigma < \pi$---is ruled out under planarity (\Cref{sec:no-match}).

\medskip

The \emph{representative} (or \emph{survivor}) of an active interval~$J$
is the actor $\sigma^{-1}(J)$ flowing as it---the one actor of~$J$ not yet
consumed by a coalescence.
Under the casting it \emph{aspires} to the entity $\pi(\sigma^{-1}(J))$. The
\emph{$\sigma$-active actors} are the representatives of the active
intervals of~$P_\sigma$.

\medskip

A crossing of two adjacent active intervals $I^- \ilo I^+$ at junction~$g$
forces their merge; the only freedom is how $\sigma$ updates its
bijection, and the requirement $\sigma \leq \pi$ pins it. The two gluings of the
two-particle case are both available---$\lambda_-$ drawing the ghost
from $I^-$, $\lambda_+$ from $I^+$
(\Cref{fig:simplified-coalescence,fig:simplified-coalescence-negative})---but
$\sigma \leq \pi$ admits only the
one whose ghost performer is $\pi^{-1}(g)$. Since $\pi^{-1}(g)$ is a single
actor, at most one of the two representatives can be it: there is one
valid update when $\pi^{-1}(g) \in \{\sigma^{-1}(I^-), \sigma^{-1}(I^+)\}$, and none at
all otherwise---the third outcome, absent from attribution. Several
intervals meeting at one vertex are processed as successive binary
crossings, in label order, matching attribution's nested
pairing (\Cref{sec:proof-high-indegree}); the binary description below
loses no generality.

\fnote{%
  prose={Valid versus spurious crossings; a spurious crossing is one of the
    rehearsal failure witnesses.},
  lean={Coalescence.RehearseSpuriousFail},
  py={combinatorics.rehearsal}}%
\begin{definition}[Valid and spurious crossings]
\label{def:valid-spurious}
A crossing between adjacent active intervals $I^-$ and $I^+$
at junction~$g$ is \emph{valid} if one of them is destined for the
ghost role at~$g$---that is, if $\pi(\sigma^{-1}(I^-)) = g$ or
$\pi(\sigma^{-1}(I^+)) = g$, so that one of the two representatives
aspires to become ghost~$g$. Otherwise the crossing is
\emph{spurious}.
\end{definition}

A valid crossing becomes a coalescence: one path exits as a
ghost path, the other represents the merged interval, and the
system shrinks. A spurious crossing causes rehearsal to
halt: rehearsal terminates with \emph{failure}, reporting the
failure pair~$(I, J)$ at the crossing vertex~$v$. Rehearsal
terminates \emph{successfully} when no crossings remain.

\subsubsection{The procedure}

Rehearsal reads the casting one vertex at a time, in the linear
order~$\le$ (\Cref{def:spacetime-graph}): it passes through the vertices
visited by at most one $\sigma$-active representative and acts only at
crossings, firing a coalescence at each valid one and halting at the first
spurious one.

\ %

\noindent\textbf{Input:} An $\varepsilon$-candidate casting $(\pi, \paths)$, where
$\pi\colon \Actors \to \Roles$ is an $\varepsilon$-candidate bijection (fixed
throughout).

\medskip\noindent\textbf{Rehearsal.}
\begin{enumerate}[label=\textup{(S\arabic*)}, ref=\textup{(S\arabic*)},
                  leftmargin=*]
\item\label{itm:s-init} Initialize: reality $\sigma \gets$ the identity
  assignment (the bottom of the poset), and $\perf \gets$ the empty
  performance.
\item\label{itm:s-loop} Visit the vertices~$v$ of the DAG in the linear
  order~$\le$ (\Cref{def:spacetime-graph}). At most vertices fewer than two
  $\sigma$-active representatives meet, and nothing happens---the sweep
  passes through. Otherwise, while two or more $\sigma$-active
  representatives meet at~$v$:
  \begin{enumerate}
  \item Take the crossing pair: the active intervals of $P_\sigma$ whose
    representatives reach~$v$ form a contiguous block, by consecutivity
    (\Cref{sec:proof-adjacent}); let $I^- \ilo I^+$ be its leftmost
    adjacent pair, with junction~$g$ and representative paths
    $P = P_{\sigma^{-1}(I^-)}$ and $Q = P_{\sigma^{-1}(I^+)}$ read from the
    casting. Let $I = \sigma^{-1}(I^-)$ and $J = \sigma^{-1}(I^+)$ be the
    two representatives.
  \item \textbf{Test:} Is $g \in \{\pi(I), \pi(J)\}$? (Does one of
    the two representatives aspire to ghost~$g$?)
  \begin{itemize}
	  \item If \textbf{no}, the crossing is \emph{spurious}: rehearsal terminates (\emph{failure}), reporting the
	    failure pair $(I, J)$ at vertex~$v$.
	  \item If \textbf{yes} (\emph{valid} crossing)---record a coalescence.
	    One of $I, J$ aspires to~$g$ ($\pi(\cdot) = g$)---the \emph{ghost
	    performer}---and the other is the \emph{survivor}.
	    \begin{itemize}
		    \item \emph{Ghost path:} the suffix from~$v$ of the ghost
		      performer's casting path becomes the ghost path~$\Gamma_g$.
		    \item \emph{Genealogy:} Record $v$ as an internal vertex
		      (merger point) of the genealogy tree; the segments of
		      $P$ and $Q$ from their most recent coalescence vertex
		      (or starting point) to~$v$ become tree edges.
		    \item Form the merged interval $H \gets I^- \cup I^+$
		      and extend $\sigma$ by this binary coalescence: the
		      partition $P_\sigma$ replaces $\{I^-, I^+\}$ by
		      $\{H\}$ and adds~$g$ to its ghost set, while the
		      bijection sends the survivor to the merged
		      interval~$H$ as its heir actor and the ghost performer
		      to the new ghost~$g$. Re-examine~$v$, then continue the
		      sweep.
	    \end{itemize}
    \end{itemize}
  \end{enumerate}
\item\label{itm:s-success} \textbf{Success}: the sweep \ref{itm:s-loop}
  runs to the end of the DAG without a spurious crossing. The remaining
  representative paths (suffixes from the last coalescence to the heir
  endpoints) become the final edges of the genealogy trees.
  Reality has climbed all the way to aspiration: $\sigma$ has
  reached $(P_\FinalState, \pi)$ (verified below in
  \Cref{prop:rehearsal-boundary}).
  Return $(\pi, \paths)$ (a fixed point) together with the
  performance~$\perf$. Rehearsal terminates (\emph{success}).
\end{enumerate}

\subsubsection{Successful castings}

\fnote{%
  prose={A casting that rehearses to the success outcome (one of the five
    rehearsal results).},
  lean={Coalescence.RehearseResult},
  py={combinatorics.rehearsal.RehearseResult}}%
\begin{definition}[Successful casting]\label{def:successful-casting}
An $\varepsilon$-candidate casting is \emph{successful} if rehearsal
processes all crossings as valid coalescences without
encountering a spurious one.
\end{definition}

\begin{remark}[A single obstruction]
A priori, rehearsing a casting could break down in several ways: a
crossing selected by rehearsal might involve non-adjacent intervals;
a crossing might be \emph{spurious} (neither representative aspires to the
ghost); or the run might terminate with reality failing to match
aspiration on some ghost or heir. The planarity assumptions collapse
this list: consecutivity keeps every selected crossing adjacent---at the
start and after each merger (\Cref{sec:proof-adjacent})---and
assignment rigidity rules out a mismatched termination
(\Cref{cor:rehearsal-reaches-aspiration}). The spurious crossing is
therefore the only genuine obstruction---and it is exactly the failure
that segment swap pairs off (\Cref{sec:involution}).
\end{remark}

\subsection{Properties of the rehearsal algorithm}
\label{sec:rehearsal-properties}

This subsection records two facts about the reality $\sigma$ that the
involution argument (\Cref{sec:involution}) uses downstream: throughout
the run $\sigma$ stays below the aspiration in the poset, and the crossing
rehearsal selects is always between adjacent active intervals.

\subsubsection{Reality stays below aspiration}

\fnote{%
  prose={Reality stays below aspiration ($\sigma \le \pi$): the rehearsal
    state never overtakes its target in the assignment poset.},
  lean={Coalescence.IntervalPartition.Coarsens}}%
\begin{lemma}[Reality stays below aspiration]\label{lem:sigma-le-pi}
While rehearsal has not crashed, the running diagonal assignment
$\sigma$ satisfies $\sigma \leq \pi$ in the assignment poset
(\Cref{def:assignment-order}): every ghost $\sigma$ has fired is a ghost
of $\pi$ too, with the same performer.
\end{lemma}

\noindent
This is immediate by induction on the accepted coalescences: each
accepted step fires a ghost~$g$ with $\sigma^{-1}(g) = \pi^{-1}(g)$ (the
validity test passed) and changes nothing else, starting from the
identity assignment.

\subsubsection{Crossings are between consecutive intervals}
\label{sec:proof-adjacent}

At any vertex~$v$, the active intervals whose representatives reach~$v$
form a consecutive run: if the representatives of two active intervals
meet at~$v$, so does the representative of every active interval between
them, whose path is ordered between theirs. This is consecutivity
(\Cref{prop:consecutivity}). A crossing is therefore always between
consecutive active intervals, and a vertex where several meet carries a
contiguous block, which rehearsal processes as successive binary
crossings of adjacent pairs---in label order, matching
attribution's nested pairing (\Cref{sec:proof-high-indegree}). The
property holds at every such vertex, not only the one rehearsal
reaches first, and is inherited after each merger.

\subsection{The no-match obstruction and assignment rigidity}
\label{sec:no-match}

The properties above describe what happens \emph{while} 
rehearsal is running and at a spurious failure. There is one
further outcome we still have to rule out: a successful run
that nevertheless fails to reconstruct a performance, because
the running assignment $\sigma$ stops strictly below the
$\varepsilon$-candidate~$\pi$. This subsection resolves that obstruction
under planarity, via a purely combinatorial rigidity lemma on
assignments.

\subsubsection{The obstruction}

Fix an $\varepsilon$-candidate casting, with bijection~$\pi$, and run
rehearsal on it, producing the running assignment~$\sigma$. Even a
successful run---one that did not crash on a spurious crossing---need not
end with $\sigma$ reaching~$\pi$.
It is conceivable that rehearsal halts at $\sigma < \pi$: some ghost
$g \in \Ghosts(P_\pi)$ is not yet a ghost of $P_\sigma$,
yet no two $\sigma$-active actors' paths cross in the
casting (so rehearsal has nothing further to process).

Such an outcome would break the proof. The casting would
neither be in attribution's image ($\sigma$ never reaches
$\pi$) nor be spuriously failed (no rejection occurred), so it
would have no partner under the involution~$\iota$
(\Cref{thm:involution}) and would survive in the signed sum.

The stall is ruled out in two steps. First,
\Cref{sec:rigidity-under-planarity} isolates the combinatorial
core, \Cref{lem:assignment-rigidity}: a rigidity statement on the
aspiration~$\pi$ and the running assignment~$\sigma$ that makes no
reference to paths or crossings. Then, in
\Cref{sec:rigidity-hypotheses}, planarity~\ref{itm:P1-crossing}
supplies the lemma's hypotheses at a successful termination of
rehearsal; the conclusion is recorded as
\Cref{cor:rehearsal-reaches-aspiration}, the form in which the rest
of the proof uses this subsection.

\subsubsection{The rigidity lemma}
\label{sec:rigidity-under-planarity}

\fnote{%
  prose={Assignment rigidity and heir-ordering; for $V_D$ both are
    discharged from P1.},
  lean={Coalescence.Valuation.assignment_rigidity_property, Coalescence.Valuation.heir_ordering_property},
  py={valuations.predicates}}%
\begin{lemma}[Assignment rigidity]\label{lem:assignment-rigidity}
Let $\pi$ be an assignment and let $\sigma \leq \pi$ be a
diagonal assignment. Suppose:
\begin{enumerate}[label=(\roman*),ref=\roman*]
\item \label{itm:rigidity-heir-pres}
  $\pi$ is \emph{heir-order-preserving}: for active intervals
  $I \ilo J$ of $P_\pi$,
  $\pi^{-1}(I) \ilo \pi^{-1}(J)$;
\item \label{itm:rigidity-G}
  for every ghost
  $g \in \Ghosts(P_\pi) \setminus \Ghosts(P_\sigma)$, writing
  $H = \heir_\pi(g)$ for the unique active interval of $P_\pi$ whose
  interior contains~$g$, exactly one of the following holds:
  \begin{enumerate}[label=\textup{(G\arabic*)},
                     ref=\textup{(G\arabic*)},
                     leftmargin=4em]
  \item \label{itm:ghost-left}
    $\pi^{-1}(H) \ilo \pi^{-1}(g) \ilo g$, or
  \item \label{itm:ghost-right}
    $g \ilo \pi^{-1}(g) \ilo \pi^{-1}(H)$.
  \end{enumerate}
\end{enumerate}
Then $\sigma = \pi$.
\end{lemma}

Two terms, defined relative to the pair $\sigma \leq \pi$, run
through the proof and through \Cref{fig:assignment-rigidity}. A
ghost of~$P_\pi$ is \emph{pinned} if it is also a ghost
of~$P_\sigma$: it then has the same performer under both
assignments, $\sigma^{-1}(g) = \pi^{-1}(g)$
(\Cref{def:assignment-order}), so $\pi$ cannot reassign it. The
remaining ghosts of~$P_\pi$, those in
$\Ghosts(P_\pi) \setminus \Ghosts(P_\sigma)$, are \emph{unmet}.
Hypothesis~(\ref{itm:rigidity-G}) is a \emph{betweenness} condition
on the unmet ghosts: \ref{itm:ghost-left} and \ref{itm:ghost-right}
are the two mutually exclusive ways of saying that the
performer~$\pi^{-1}(g)$ of an unmet ghost lies strictly between the
ghost~$g$ itself and the performer~$\pi^{-1}(H)$ of its heir. The
proof uses the hypothesis mostly through this reading.

\begin{figure}[tp]
\centering
\begin{tikzpicture}[
    x=1.0cm, y=1cm,
    heirbox/.style   = {rounded corners=2pt, thick, draw=colP, fill=colP!12},
    sigmabox/.style  = {rounded corners=2pt, thick, draw=colB, fill=colB!14},
    prefixbox/.style = {rounded corners=2pt, thick, draw=black!35,
                        fill=black!6},
    cellbox/.style   = {rounded corners=1pt, draw=black!30},
    cellpinned/.style= {rounded corners=1pt, draw=black!30, fill=black!12},
    cellfree/.style  = {rounded corners=1pt, draw=colBC, line width=1pt,
                        fill=colBC!18},
    repP/.style   = {circle, fill=colP, inner sep=1.5pt},
    repB/.style   = {circle, fill=colB, inner sep=1.5pt},
    repA/.style   = {circle, fill=colA, inner sep=1.5pt},
    repgray/.style= {circle, fill=black!45, inner sep=1.4pt},
    repfree/.style= {circle, fill=colBC, draw=colBC, line width=0.8pt,
                     inner sep=1.8pt},
    gperfring/.style = {circle, draw=colC, fill=white,
                        line width=1.3pt, inner sep=1.6pt},
    ghostmark/.style   = {circle, draw=colGhost, fill=white,
                          line width=1pt, inner sep=1.7pt},
    ghostfree/.style   = {circle, draw=colC, fill=white,
                          line width=1.4pt, inner sep=1.9pt},
    gsmall/.style     = {circle, draw=colGhost, fill=white,
                         line width=0.7pt, inner sep=1.1pt},
    gfreesmall/.style = {circle, draw=colC, fill=white,
                         line width=1.2pt, inner sep=1.4pt},
    prefixghost/.style = {circle, draw=black!45, fill=white,
                          line width=0.8pt, inner sep=1.5pt},
    bnd/.style       = {colB, line width=1.2pt},   %
    sgarr/.style     = {colBC, -{Latex[length=4pt]}, thick},
    pinarr/.style    = {colGhost, -{Latex[length=3.5pt]}, semithick,
                        densely dotted},
    piarr/.style     = {-{Latex[length=3.2pt]}, semithick},
    guide/.style     = {black!25, densely dotted},
    clash/.style     = {font=\scriptsize\itshape},
    special/.style   = {font=\footnotesize},
    rowlab/.style    = {font=\footnotesize\itshape, anchor=east},
    caselab/.style   = {font=\footnotesize, anchor=west},
    inlab/.style     = {font=\footnotesize},
    sublab/.style    = {font=\scriptsize},
]

\def\gap{0.09}   %
\def\eps{0.22}   %
\def\bh{0.30}    %
\def\ch{0.22}    %
\def\yA{3.15}    %
\def\yB{1.75}    %
\def\yC{0.55}    %
\def\yax{-0.35}  %
\def\gdrop{0.50} %
\def\yGl{-2.55}  %
\def\yGr{-5.30}  %
\pgfmathsetmacro{\gA}{\yA-\gdrop}   %
\pgfmathsetmacro{\gB}{\yB-\gdrop}   %

\def\Ibox#1#2#3#4{\draw[#4] ({#1+\gap},{#3-\bh}) rectangle ({#2-\gap},{#3+\bh});}
\def\Acell#1#2#3{\draw[#3] ({#1+\gap},{#2-\ch}) rectangle ({#1+1-\gap},{#2+\ch});}
\def\Rep#1#2#3{\node[#3] at ({#1+\eps},#2) {};}
\def\GPerf#1#2{\node[gperfring] at (#1,#2) {};
  \fill[colC] (#1,#2) circle[radius=0.85pt];}

\def\xa{5} \def\xc{6} \def\xg{8} \def\xb{11}
\def\xmin{0.7} \def\xmax{11.9}

\fill[black!5] (\xmin,\yax-0.05) rectangle (\xa,\yA+\bh+0.30);
\fill[black!5] (\xmin,{\yGl-\ch-0.10}) rectangle (\xa,{\yGl+\ch+0.10});
\fill[black!5] (\xmin,{\yGr-\ch-0.10}) rectangle (\xa,{\yGr+\ch+0.10});
\draw[black!30, densely dashed] (\xa,\yax) -- (\xa,\yA+\bh+0.30);

\draw[guide] (\xa,\yax-0.62) -- (\xa,\yGr-0.65);
\draw[guide] (\xg,\yax-0.62) -- (\xg,\yGr-0.65);

\node[rowlab] at (\xmin-0.05,\yA) {$P_\pi$};
\node[rowlab] at (\xmin-0.05,\yB) {$P_\sigma$};
\node[rowlab] at (\xmin-0.05,\yC) {actors};

\draw[->, thick] (\xmin,\yax) -- (\xmax,\yax);
\foreach \x in {1,...,11}{ \draw[black!70] (\x,\yax-0.06) -- (\x,\yax+0.06); }
\foreach \x/\lab in {\xa/a, \xc/c, \xg/g, \xb/b}{
  \draw[thick] (\x,\yax-0.10) -- (\x,\yax+0.10);
  \node[special, below] at (\x,\yax-0.10) {$\lab$};
}
\node[special, below] at (1,\yax-0.10) {$1$};
\node[special, below] at (4,\yax-0.10) {$a{-}1$};

\foreach \y in {\yA,\yB}{
  \Ibox{1}{3}{\y}{prefixbox} \Rep{1}{\y}{repgray}
  \Ibox{3}{5}{\y}{prefixbox} \Rep{3}{\y}{repgray}
}
\node[prefixghost] at (2,\gA) {};  \node[prefixghost] at (2,\gB) {};
\node[prefixghost] at (4,\gA) {};  \node[prefixghost] at (4,\gB) {};

\Ibox{\xa}{\xb}{\yA}{heirbox} \Rep{\xa}{\yA}{repP}
\node[inlab] at ({\xa+0.75},\yA) {$H$};
\foreach \x in {6,7,9,10}{ \node[ghostmark] at (\x,\gA) {}; }
\node[ghostfree] (gA) at (\xg,\gA) {};              %
\node[special, anchor=west] at ({\xg+0.16},\gA) {$g$};

\Ibox{\xa}{\xg}{\yB}{sigmabox} \Rep{\xa}{\yB}{repB}
\node[inlab] at ({\xa+0.7},\yB) {$J$};
\Ibox{\xg}{\xb}{\yB}{sigmabox} \Rep{\xg}{\yB}{repB}
\draw[bnd] (\xg,\yB-\bh-0.06) -- (\xg,\yB+\bh+0.06);     %
\node[ghostmark] at (6,\gB) {};
\node[ghostmark] (pinB) at (7,\gB) {};
\node[ghostmark] at (9,\gB) {};  \node[ghostmark] at (10,\gB) {};

\foreach \x in {1,2,3,4,5,7,8,9,10}{ \Acell{\x}{\yC}{cellbox} }
\Acell{6}{\yC}{cellfree}
\foreach \x in {1,2,3,4}{ \Rep{\x}{\yC}{repgray} }
\foreach \x in {5,7,8,9,10}{ \Rep{\x}{\yC}{repA} }
\node[repfree] (cC) at ({\xc+\eps},\yC) {};      %
\node[special, anchor=west] at ({\xc+\eps+0.14},\yC) {$c$};

\draw[sgarr] (cC) to[out=90,in=-90] ({\xa+\eps},\yB);
\draw[pinarr] ({7+\eps},\yC+0.10) to[out=90,in=-90] (pinB);
\node[font=\scriptsize, anchor=west] at (7.35,\gB) {pinned};

\draw[decorate, decoration={brace, amplitude=4pt}, black]
    (\xg-\gap,\yA+\bh+0.16) -- (\xa+\gap,\yA+\bh+0.16)
    node[midway, above, font=\scriptsize, text=black]
    {$\pi$ determined on $[a,g)$ except at~$c$};

\node[caselab] at (\xmin,\yGl+1.28)
  {Case \ref{itm:ghost-left} at~$g$:\quad
   $\pi^{-1}(H) \ilo \pi^{-1}(g) \ilo g$
   \ ---\ two claimants for the free label};
\draw[heirbox] ({\xa+\gap},{\yGl+0.70}) rectangle ({\xb-\gap},{\yGl+0.96});
\node[repP] at ({\xa+\eps},{\yGl+0.83}) {};
\node[sublab] at ({\xa+0.75},{\yGl+0.83}) {$H$};
\foreach \x in {6,7,9,10}{ \node[gsmall] at (\x,{\yGl+0.54}) {}; }
\node[gfreesmall] at (\xg,{\yGl+0.54}) {};
\node[sublab, anchor=west] at ({\xg+0.14},{\yGl+0.54}) {$g$};
\foreach \x in {1,2,3,4,8,9,10}{ \Acell{\x}{\yGl}{cellbox} }
\Acell{5}{\yGl}{cellpinned} \Acell{7}{\yGl}{cellpinned}
\Acell{6}{\yGl}{cellfree}
\node[repP] at (6.32,\yGl) {};
\GPerf{6.68}{\yGl}
\draw[piarr, colP] (6.32,\yGl+0.08) to[out=115,in=-55] (5.60,\yGl+0.68);
\draw[piarr, colC] (6.68,\yGl+0.08) to[out=40,in=195] (7.90,\yGl+0.43);
\node[clash] at (6.5,\yGl-0.50)
  {$\pi^{-1}(H) = c = \pi^{-1}(g)$};

\node[caselab] at (\xmin,\yGr+1.28)
  {Case \ref{itm:ghost-right} at~$g$:\quad
   $g \ilo \pi^{-1}(g) \ilo \pi^{-1}(H)$
   \ ---\ no claimant for the free label};
\draw[heirbox] ({\xa+\gap},{\yGr+0.70}) rectangle ({\xb-\gap},{\yGr+0.96});
\node[repP] at ({\xa+\eps},{\yGr+0.83}) {};
\node[sublab] at ({\xa+0.75},{\yGr+0.83}) {$H$};
\foreach \x in {6,7,9,10}{ \node[gsmall] at (\x,{\yGr+0.54}) {}; }
\node[gfreesmall] at (\xg,{\yGr+0.54}) {};
\node[sublab, anchor=west] at ({\xg+0.14},{\yGr+0.54}) {$g$};
\foreach \x in {1,2,3,4,8,9,10}{ \Acell{\x}{\yGr}{cellbox} }
\Acell{5}{\yGr}{cellpinned} \Acell{7}{\yGr}{cellpinned}
\Acell{6}{\yGr}{cellfree}
\node[clash] at (6.5,\yGr) {?};
\GPerf{8.35}{\yGr}
\node[repP] at (9.5,\yGr) {};
\draw[piarr, colC] (8.35,\yGr+0.08) to[out=110,in=-45] (8.04,\yGr+0.42);
\draw[piarr, colP] (9.5,\yGr+0.08) -- (9.5,\yGr+0.71);
\node[sublab] at (8.35,\yGr-0.50) {$\pi^{-1}(g)$};
\node[sublab] at (9.6,\yGr-0.50) {$\pi^{-1}(H)$};
\node[clash] at (6.5,\yGr-0.50) {no one takes~$c$};

\end{tikzpicture}

\caption{\textbf{Assignment rigidity: the one-free-label
pigeonhole.} \emph{Top three rows:} the entities of~$P_\pi$,
of~$P_\sigma$, and the actors over a shared axis of labels, drawn
as in \Cref{fig:simplified-coalescence}: an interval $[x,y)$ is a
box with a filled circle at its minimum, ghosts are open circles
below their row, and arrows depict the assignment~$\sigma$, from an
actor to the entity it performs. Left of~$a$ the two partitions
agree and all ghosts are pinned (grayed out). The minimal unmet
ghost~$g$ is interior to the $\pi$-heir $H=[a,b)$ but is a boundary
of~$P_\sigma$, so $J=[a,g)$ is an heir of~$P_\sigma$, and $\pi$ is
determined on the window $[a,g)$ except at the free actor
$c=\sigma^{-1}(J)$ (brace; Steps 1--2 of the proof). \emph{Bottom
two rows:} Step~3, each case played against its own copy of the
actors' row and a slim copy of the entities of~$P_\pi$. The arrows
keep their meaning, now for~$\pi$: a filled dot performs~$H$ and the
circled dot performs the unmet ghost~$g$, each placed in the cell of
its performer; the gray cells are taken by the performers of the
pinned ghosts of the window. In case~\ref{itm:ghost-left} both
performers lie in the window, whose only free cell is~$c$: the dots
collide, $\pi^{-1}(H) = c = \pi^{-1}(g)$, impossible for distinct
labels. In case~\ref{itm:ghost-right} every performer avoids the
free cell (prefix entities perform left of~$a$, everything else at
or beyond~$g$): $c$ goes unclaimed, contradicting bijectivity.}
\label{fig:assignment-rigidity}
\end{figure}

\begin{proof}[Proof of \Cref{lem:assignment-rigidity}]
It suffices to prove that every ghost of~$P_\pi$ is pinned, that is,
$\Ghosts(P_\pi) = \Ghosts(P_\sigma)$. Indeed, a partition is
determined by its ghost junctions, so equal ghost sets force
$P_\pi = P_\sigma$; the two assignments then have the same heirs,
agree at every ghost (each being pinned), and hence use the same
heir actors---the labels left over for the heirs. On heirs, both
assignments are order-preserving: $\sigma$ by diagonality
(\Cref{def:diagonal}; the heir actor of an active interval lies
inside it, and the intervals are disjoint), $\pi$ by
hypothesis~(\ref{itm:rigidity-heir-pres}). An order-preserving
bijection from the heirs onto the heir actors is unique---it must
send the $i$-th heir in label order to the $i$-th heir actor---so
the assignments agree on heirs as well, and $\sigma = \pi$.

\medskip

Suppose then, for contradiction, that unmet ghosts exist. Let~$g$
be the minimum unmet ghost in label order, and let
\[
H := \heir_\pi(g) = [a, b)
\]
be its heir, the unique active interval of~$P_\pi$ whose interior
contains~$g$; thus $a < g < b$. From now on we identify each actor
with its integer label and compare labels and junctions by the
usual order on integers, reserving~$\ilo$ for quoting the
hypotheses. Two structural remarks will be used repeatedly. First,
since $\Ghosts(P_\sigma) \subseteq \Ghosts(P_\pi)$, every junction
that is a boundary of~$P_\pi$---a non-ghost---is a boundary
of~$P_\sigma$ as well; and since a junction interior to an active
interval is always a ghost, no active interval of either partition
straddles such a boundary. This applies in particular to~$a$.
Second, by minimality of~$g$, every ghost of~$P_\pi$ with
label~$< g$ is pinned.

The contradiction will be a pigeonhole on a single free label,
traced in \Cref{fig:assignment-rigidity}. Step~1 shows that $\pi$
matches the labels below~$a$ with the entities lying below~$a$; in
particular no heir has its performer in $[a, \pi^{-1}(H))$. Step~2
shows that inside the window $[a, g)$ the pinned ghosts determine
the value of~$\pi$ at every label except one, the free
actor~$c$. Step~3 plays the two orders of
betweenness at~$g$ against this single free label: in
case~\ref{itm:ghost-left} the performers of~$H$ and of~$g$ are both
trapped in the window---two claimants for one free label---while in
case~\ref{itm:ghost-right} no entity at all may claim~$c$, though
the bijection~$\pi$ must use that label. Either way we reach a
contradiction.

\medskip\noindent\textbf{Step 1: $\pi$ matches the prefix.}
Call an entity of~$P_\pi$ a \emph{prefix} entity if it lies
entirely below~$a$: an heir with upper endpoint~$\leq a$, or a
ghost with label~$< a$. Since no heir of~$P_\pi$ straddles~$a$, the
prefix entities tile $[1, a)$---each prefix heir contributes itself
and its interior ghosts, one entity per label---so there are
exactly $a - 1$ of them, as many as there are labels below~$a$. We
claim that every \emph{non-prefix} entity has performer~$\geq a$.
Since $\pi$ is injective, the claim forces $\pi$ to match the
$a - 1$ labels below~$a$ with the $a - 1$ prefix entities; in
particular, every prefix heir then has performer~$< a$. This
matching is the grayed-out region left of the cut at~$a$ in
\Cref{fig:assignment-rigidity}. We verify the claim one kind of
entity at a time.

\smallskip\noindent
\emph{Pinned ghosts.} The performer of a pinned ghost is its
diagonal $\sigma$-performer, a label inside the active interval
of~$P_\sigma$ whose interior contains the ghost
(\Cref{def:diagonal}). That interval does not straddle~$a$. Hence a
pinned ghost with label~$> a$ has performer~$\geq a$; and---needed
in a moment---every prefix ghost, being pinned (its label is
$< a < g$), has performer~$< a$.

\smallskip\noindent
\emph{Heirs.} First, $\pi^{-1}(H) \geq a$. Otherwise
heir-order-preservation~(\ref{itm:rigidity-heir-pres}) would give
every prefix heir, each of which precedes~$H$ in interval order, a
performer $< \pi^{-1}(H) < a$ as well; the $a - 1$ prefix entities
together with~$H$ would then have $a$ distinct performers among the
$a - 1$ labels below~$a$. Next, every non-prefix heir~$K$ either
equals~$H$ or lies entirely to the right of~$H$ (its lower endpoint
is $\geq a$, since it cannot straddle~$a$), so
heir-order-preservation spreads the bound:
$\pi^{-1}(K) \geq \pi^{-1}(H) \geq a$.

\smallskip\noindent
\emph{Unmet ghosts.} An unmet ghost~$g''$ has label $g'' \geq g > a$
by minimality of~$g$, and its heir $H'' = \heir_\pi(g'')$
contains~$g''$ in its interior, so $H''$ is not a prefix heir and
$\pi^{-1}(H'') \geq a$ by the previous paragraph. Betweenness
places $\pi^{-1}(g'')$ strictly between $g''$ and $\pi^{-1}(H'')$,
neither of which lies left of~$a$; hence $\pi^{-1}(g'') \geq a$.

Heirs, pinned ghosts with label~$> a$, and unmet ghosts exhaust the
non-prefix entities, so the claim and the matching hold. We record
the resulting \emph{two-range} bound on heir-performers, the only
form in which Step~3 will use it:
\begin{equation}\label{eq:rigidity-two-range}
\begin{aligned}
  \pi^{-1}(K) &< a
    && \text{when heir $K$ has upper endpoint } \leq a, \\
  \pi^{-1}(K) &\geq \pi^{-1}(H) \geq a
    && \text{for every other heir}~K.
\end{aligned}
\end{equation}

\medskip\noindent\textbf{Step 2: one free label in the window
$[a, g)$.}
Every junction with label in $(a, g)$ is interior to~$H$, hence a
ghost of~$P_\pi$ (a junction interior to an active interval is a
ghost), hence pinned---so in particular a ghost of~$P_\sigma$. The
junction~$g$ itself is unmet, so it is a boundary of~$P_\sigma$;
and $a$ is a boundary of~$P_\sigma$ by the first structural remark
above. Consequently
\[
J := [a, g)
\]
is an active interval of~$P_\sigma$, a strict sub-interval of~$H$.

By $\sigma$-diagonality at~$J$ (\Cref{def:diagonal}), the $g - a$
labels of the window split disjointly as
\[
[a, g) = \{c\} \sqcup
\bigl\{\sigma^{-1}(g') : g' \in \Ghosts(P_\sigma) \cap (a, g)\bigr\},
\qquad c := \sigma^{-1}(J),
\]
the heir actor of~$J$ together with the performers of the pinned
ghosts interior to~$J$. Pinning transfers the second set to~$\pi$:
under~$\pi$ as well, the ghosts in $(a, g)$ consume every label of
the window except~$c$. In other words, the value of~$\pi$ is
already determined at every label of the window but one---the free
actor~$c$ singled out in \Cref{fig:assignment-rigidity}---and an
entity other than a pinned ghost in $(a, g)$ can have its
$\pi$-performer in the window only by taking the one free
label~$c$.

\medskip\noindent\textbf{Step 3: the free label gets two claimants
or none.}
By hypothesis~(\ref{itm:rigidity-G}), one of the two orders of
betweenness holds at~$g$ with heir~$H$; we treat them in turn.

\smallskip\noindent
\textbf{Case~\ref{itm:ghost-left}: two claimants for the free
label.}
Here $\pi^{-1}(H) \ilo \pi^{-1}(g) \ilo g$, that is,
$\pi^{-1}(g) \leq g - 1$ and $\pi^{-1}(H) \leq g - 2$. By
\eqref{eq:rigidity-two-range} we also have $\pi^{-1}(H) \geq a$, so
$\pi^{-1}(H)$ and $\pi^{-1}(g)$ are two \emph{distinct} labels in
the window $[a, g)$. But by Step~2, every label of the window other
than~$c$ performs, under~$\pi$, a pinned ghost in $(a, g)$---and
neither the heir~$H$ nor the unmet ghost~$g$ is such an entity.
Both labels would therefore have to equal the single free
label~$c$: impossible, as they are distinct.

\smallskip\noindent
\textbf{Case~\ref{itm:ghost-right}: no claimant for the free
label.}
Here $g \ilo \pi^{-1}(g) \ilo \pi^{-1}(H)$, that is,
$\pi^{-1}(g) \geq g$ and $\pi^{-1}(H) \geq g + 1$. The free
actor~$c$ performs some entity $\pi(c)$; we rule out every
candidate in turn.
\begin{itemize}
\item \emph{An heir.} With $\pi^{-1}(H) \geq g + 1$, the
  bound~\eqref{eq:rigidity-two-range} confines every heir-performer
  to $[1, a) \cup [g + 1, n + 1)$, and $c \in [a, g)$ lies in
  neither range.
\item \emph{A pinned ghost.} If $\pi(c)$ were a pinned ghost~$g''$,
  then $c = \pi^{-1}(g'') = \sigma^{-1}(g'')$. But $\sigma$
  sends~$c$ to the heir~$J$, not to a ghost.
\item \emph{The ghost~$g$ itself.} Its performer satisfies
  $\pi^{-1}(g) \geq g > c$.
\item \emph{Another unmet ghost~$g'$.} By minimality of~$g$ we have
  $g' > g$. The heir $H' := \heir_\pi(g')$ contains $g' > a$ in its
  interior, so $H'$ is not a prefix heir, and
  \eqref{eq:rigidity-two-range} gives
  $\pi^{-1}(H') \geq \pi^{-1}(H) \geq g + 1$. Betweenness at~$g'$
  would place its performer~$c$ strictly between $g'$ and
  $\pi^{-1}(H')$, both of which lie right of~$g$; but $c < g$.
\end{itemize}
Heirs, pinned ghosts, and unmet ghosts exhaust the entities
of~$P_\pi$: no candidate for $\pi(c)$ remains---the free label goes
unclaimed, contradicting that $\pi$ is a bijection defined at~$c$.

\medskip
Both cases are impossible. Hence no unmet ghost exists,
$\Ghosts(P_\pi) = \Ghosts(P_\sigma)$, and, as shown at the outset,
$\sigma = \pi$.
\end{proof}

\subsubsection{From rehearsal success to the lemma's hypotheses}
\label{sec:rigidity-hypotheses}

\begin{corollary}[Successful rehearsal reaches the aspiration]
\label{cor:rehearsal-reaches-aspiration}
Run rehearsal on an $\varepsilon$-candidate casting with
bijection~$\pi$, producing the running assignment~$\sigma$. Then
the run either fails on a spurious crossing or terminates with
$\sigma = \pi$.
\end{corollary}

\begin{proof}
At a successful termination of rehearsal, the running
assignment~$\sigma$ satisfies $\sigma \leq \pi$
(\Cref{lem:sigma-le-pi}) and is diagonal by construction;
$\pi$ is an $\varepsilon$-candidate by the standing assumption. Successful
termination also means that no two $\sigma$-active actors'
paths cross in the casting. The crossing
property~\ref{itm:P1-crossing} turns this non-crossing into order
preservation: two paths whose sources are in label order,
$x_I \preceq x_K$, but whose endpoints are reversed,
$y_{\pi(K)} \preceq y_{\pi(I)}$, are forced to intersect. So if such
paths do not cross, their endpoints keep the order of the sources:
$I \ilo K$ implies $y_{\pi(I)} \prec y_{\pi(K)}$. This single
observation supplies the two remaining hypotheses of
\Cref{lem:assignment-rigidity}.

\medskip\noindent
\emph{Heir-order-preservation of~$\pi$.} For $\pi$-heirs
$H \ilo H'$, left-to-right order on the endpoints
$y_H, y_{H'}$ (sorted by the heir-position assumption,
\Cref{sec:setup-final}) forces
$\pi^{-1}(H) \ilo \pi^{-1}(H')$.

\medskip\noindent
\emph{The betweenness dichotomy
\ref{itm:ghost-left}/\ref{itm:ghost-right}.} For an unmet
ghost~$g$ of $\pi$---one not yet fired in~$\sigma$---with heir
$H = \heir_\pi(g)$, both performers $\pi^{-1}(g)$ and $\pi^{-1}(H)$ are
$\sigma$-active: the ghost~$g$ has not fired, so $\pi^{-1}(g)$ is still
active; and the performer of a final heir is never consumed (by
$\sigma \leq \pi$, \Cref{lem:sigma-le-pi}, a consumed actor performs a
ghost rather than reaching an heir). Their paths therefore do not cross,
and the order-preservation observation, applied in both label directions,
gives the two-sided
equivalence
\[
\pi^{-1}(g) \ilo \pi^{-1}(H) \iff y_g \prec y_H \iff
\varepsilon_g = +1.
\]
$\varepsilon$-candidacy (\Cref{def:candidate}) is exactly
$\varepsilon_g = +1 \iff g \ilo \pi^{-1}(g)$. Combining the
two equivalences, exactly one of \ref{itm:ghost-left} or
\ref{itm:ghost-right} holds.

\medskip
All hypotheses hold, so \Cref{lem:assignment-rigidity} yields
$\sigma = \pi$.
\end{proof}

\subsection{Successful castings yield valid performances}
\label{sec:successful-to-valid}

\fnote{%
  prose={Rehearsal recovers the prescribed final state; the total, five-way
    outcome classifier.},
  lean={Coalescence.rehearse_chain},
  py={combinatorics.rehearsal}}%
\begin{proposition}[Successful rehearsal recovers the prescribed final state]\label{prop:rehearsal-boundary}
Let $\casting$ be a successful $\varepsilon$-candidate casting. The performance
produced by rehearsal has the prescribed final state~$\FinalState$.
\end{proposition}

\begin{proof}
By \Cref{cor:rehearsal-reaches-aspiration}, successful termination
forces $\sigma = \pi$; in particular the partitions agree at the
end of the run, $P_\sigma = P_\FinalState$.

\medskip\noindent\emph{Ghosts.} Each coalescence adds its junction to
$\Ghosts(P_\sigma)$, so the run fires exactly the junctions of
$\Ghosts(P_\sigma) = \Ghosts$. At the coalescence firing~$g$, the
consumed actor is the aspirant $\pi^{-1}(g)$ (the validity test,
\Cref{def:valid-spurious}), so the ghost path~$\Gamma_g$ is a
suffix of the casting path $P_{\pi^{-1}(g)}$---which ends at
$y_{\pi(\pi^{-1}(g))} = y_g$ by \Cref{def:casting}.

\medskip\noindent\emph{Heirs.} The final active intervals are
$\Active(P_\sigma) = \Heirs$, and the representative of
$H \in \Heirs$ at termination is $\sigma^{-1}(H) = \pi^{-1}(H)$,
never consumed by a coalescence; its casting path
$P_{\pi^{-1}(H)}$ therefore runs to its endpoint
$y_{\pi(\pi^{-1}(H))} = y_H$, closing the genealogy tree of~$H$.

\medskip

The performance thus coalesces at exactly the prescribed ghosts,
with ghost paths ending at the prescribed positions, and its heirs
are exactly~$\Heirs$, with paths reaching the prescribed endpoints:
its final state is~$\FinalState$.
\end{proof}

\section{The performance--casting bijection}
\label{sec:proof-bijection}

We prove that attribution and rehearsal are mutually inverse,
weight-preserving bijections between performances and
\emph{successful} castings. The key observation is that the two
constructions advance the \emph{same} object: the running
assignment~$\sigma$, which starts at the identity assignment and
climbs the assignment poset by one binary coalescence per step, in
the linear order~$\le$ of \Cref{def:spacetime-graph}. The proof
couples the two runs and checks, by induction
over the coalescences, that each step of one is undone by the
corresponding step of the other.
\Cref{sec:attribution-spacetime} prepares the coupling and the
weight argument;
\Cref{sec:proof-bijection-easy,sec:proof-bijection-hard} then verify
the two compositions---one almost definitional, one substantive.

\ifextendedversion
\subsection{Performances and castings as streams}
\else
\subsection{The running assignment as a dictionary}
\fi
\label{sec:attribution-spacetime}

The two-particle rule (\Cref{sec:two-particle}) and its high-indegree
extension (\Cref{sec:proof-high-indegree}) assemble into a single picture,
in which attribution is neither a construction nor a calculation but a
\emph{relabeling}---and the same picture, read backward, is rehearsal.

\ifextendedversion
\subsubsection{The DAG and its flow}

The spacetime DAG is fixed structure: its vertices and edges are settled
before any particle is placed on it. A performance, or a casting, is a
\emph{flow} on this DAG: along each edge something may flow, and what an
edge carries is in general a \emph{multiset} of occurrences, not a
single path. In a \emph{performance} what flows are the entities: at
most one surviving particle per edge, together with any number of
ghosts, which do not interact and so may share an edge with one another
and with the particle (\Cref{fig:intro-21}). In a \emph{casting} what
flows are the actors' paths, any number of which may share
an edge. The flow is what distinguishes one performance, or one
casting, from another; the edges themselves do not change.

It helps to picture the DAG laid out along a time line, time running
from the sources in~$\mathcal{X}$ to the targets in~$\mathcal{Y}$. One
caution is in order: the DAG itself supplies only the \emph{partial}
time ordering~$\lessdot$, under which far-apart parts of the system are
incomparable, so ``the current partition'' at a vertex is not by itself
meaningful. Whenever we speak of an instant---of when an occurrence is
born, which entities are currently alive, how far the partition has
coarsened---we mean a position in the fixed linear extension~$\le$ of
\Cref{def:spacetime-graph}: the state of the system after the vertices
up to that position have been processed. Which linear extension has
been fixed does not matter.

\subsubsection{Events and lives}

Events belong to a \emph{performance}, whose flow records which
meetings merge. An \emph{event} is a coalescence: a vertex at which two
or more surviving particles of the flow meet and merge, one heir
flowing on and one or more ghosts born. Equivalently, an event is
exactly where the partition into active intervals coarsens. A vertex
through which the flow merely passes---carrying a single particle, or
carrying ghosts, which cross everything freely---is not an event;
events are sparse among the vertices the flow visits. The boundary
sets $\mathcal{X}$ and $\mathcal{Y}$ hold no events either: the flow
begins at the sources and ends at the targets. The flow of a
\emph{casting}, by contrast, carries no events of its own: its paths
meet and cross, but which of those meetings are coalescences is not
part of the casting's data---reading it off is precisely rehearsal's
task.

Between events each occurrence persists. The \emph{life} of an
entity---an active cluster or a ghost---is the run of
edges its flow occupies, from its birth (an event, or a source
in~$\mathcal{X}$) to its end (the event that consumes it, or a target
in~$\mathcal{Y}$). A life typically spans many vertices: coalescences fire
elsewhere in the system while the entity flows on untouched. The entities
alive at a given instant are exactly the active intervals together with the
ghosts still in flight.

\subsubsection{Two streams over the vertices}

Reading the flow vertex by vertex turns each object into a \emph{stream},
one entry per vertex, in the linear order~$\le$
(\Cref{def:spacetime-graph}); its substantive entries are the
events---the coalescences---sparse among the pass-through vertices between
them. Performance and casting are two such streams, differing only in how
much each entry records.

A casting reduces to the simplest possible stream: each entry lists which
actors visit that vertex. Nothing more is needed, since an actor's identity
travels with it, so naming the actors present at every vertex pins down
every path. A single entry, however, records only \emph{which} actors meet,
not how their incoming and outgoing path-segments thread through the
vertex---a point that will matter in \Cref{sec:proof-bijection-hard}.

A performance carries a richer entry. Each entry records the partition into
active intervals on entry---which clusters have already coalesced---together
with the entities arriving there: the active intervals that meet and merge,
and the ghosts passing through. This is exactly the
bookkeeping---genealogy and ghost paths---that a performance keeps and a
casting drops.

\subsubsection{The actor behind an edge}

Attribution attaches an actor to each occurrence in the flow, and the
agent that attaches it is the running assignment of the attribution chain
(\Cref{prop:attribution-chain})---the reality~$\sigma$ that climbs from
the identity assignment, written $\sigma_t$ as it advances (the same
$\sigma$ rehearsal runs, \Cref{sec:rehearsal}). At time~$t$ it
names, for each live entity, the actor $\sigma_t^{-1}(\text{entity})$
flowing as it.

As a whole $\sigma_t$ changes at \emph{every} coalescence in the system.
Yet the entry it assigns to a given entity does not move while that entity
is alive. The actor flowing as an occurrence is fixed from the
event that creates the occurrence's entity to the event that consumes it,
and coalescences elsewhere do not touch it; so ``the actor behind an
edge'' is well defined---read $\sigma_t^{-1}$ on the entity flowing there,
at any time during its life, and the answer is the same.

\subsubsection{The dictionary as a cumulative product}

The assignment is built from below. It begins at the identity assignment,
and each coalescence composes in that vertex's local gluing
$\lambda$, the bijection $\{I^-, I^+\} \to \{H, g\}$ of the two-particle
case. So $\sigma_t$ is the cumulative product of these elementary
bijections~$\lambda$ fired up to time~$t$, taken in vertex order, and the
full product is the casting's bijection~$\pi$. This is the algebraic face of
the attribution chain (\Cref{prop:attribution-chain}): a high-indegree
vertex simply contributes its several binary factors in the tie-broken
order (\Cref{sec:proof-high-indegree}), with no separate bookkeeping. The
two properties proved above are one-line readings of the product---each
factor places the performer of its ghost on the side $\varepsilon_g$
dictates, so the product is an $\varepsilon$-candidate
(\Cref{prop:attribution-candidate}); and each factor leaves the inversion
parity unchanged (\Cref{lem:coalescence-inversion-parity}), so the
product is positive (\Cref{prop:attribution-positive}).

\subsubsection{Attribution and rehearsal as stream processors}

Each occurrence therefore has two names: in the performance it is an
entity (the cluster or ghost flowing there), in the casting an actor
(the one whose path flows along it). The assignment is the dictionary
between them: $\sigma_t^{-1}$ converts one into the other, occurrence
by occurrence.
Both attribution and rehearsal run this dictionary as \emph{stream
processors}. Attribution reads the performance stream and emits
the casting stream: at each vertex it consumes the entities arriving there,
advances $\sigma_t$ by the local gluing, and outputs the actors visiting
that vertex; when the stream ends, its terminal state is the casting's
bijection~$\pi$. Rehearsal (\Cref{sec:rehearsal}) is the inverse processor:
its loop~\ref{itm:s-loop} is precisely this vertex sweep, reading the
casting stream and emitting the performance stream while advancing the
\emph{same} $\sigma_t$, now with $\pi$ supplied in advance rather than
produced.

That the two are mutually inverse is then a statement about composing the
processors, checked one coalescence at a time. Since $\sigma_t$ changes only
at coalescences, and both processors start at the identity assignment and
advance it by the same local rule, it is enough to check that a single
coalescence survives the round trip---its entry, pushed through one
processor and back through the other, returns unchanged, leaving both in
the same state $\sigma_{t+1}$---and to induct along the stream. The two orders are not
equally easy: one is
almost definitional (\Cref{lem:reh-attr}); the other must restore the
path-threading that rehearsal discards, and is the substantive half
(\Cref{lem:attr-reh}).

One point makes the composition well-posed, even though $\pi$ is a terminal
output of attribution rather than an initial one. Rehearsal consults $\pi$
only through the single value $\pi^{-1}(g)$ at each junction~$g$---the
performer of the ghost fired there---and that value is settled the moment
$g$ fires and never changes afterward, the chain only growing
(\Cref{prop:attribution-chain}). And $g$ fires at the very vertex whose
crossing makes rehearsal consult $\pi^{-1}(g)$, so the value is available
exactly when required: the processors compose causally, with no genuine
look-ahead.

\else

\subsubsection{What each object records}

A performance names what happens by \emph{entity}: it records which
active intervals merge at each coalescence vertex and which ghost is
born there---the genealogy and the ghost paths---but never how the
path-segments meeting at such a vertex continue one another, a point
that will matter in \Cref{sec:proof-bijection-hard}. A casting names
everything by \emph{actor}: it consists of the actors' paths, any
number of which may share an edge, together with the
bijection~$\pi$, and it records no coalescences at all. Deciding
which crossings are coalescences is rehearsal's task, and only the
\emph{successful} castings describe a performance.

\subsubsection{The coupled induction}

Attribution and rehearsal build the same kind of object: an
increasing chain of diagonal assignments in the assignment poset
(\Cref{def:assignment-order}). Each run starts at the identity
assignment and advances by one binary coalescence per step, taken in
the linear order~$\le$ of \Cref{def:spacetime-graph}; we
write~$\sigma_t$ for the running assignment after the vertices up
to~$t$ have been processed. At every moment $\sigma_t^{-1}$ names,
for each entity, the actor flowing as it: the running assignment is
a dictionary between the two languages above. Attribution uses it to
translate a performance into a casting; rehearsal, when it succeeds,
uses it to translate back. The two lemmas below show that the two
runs share the \emph{same}~$\sigma_t$; the round trip therefore
translates every occurrence out and back with one dictionary, and
returns the original. The two runs differ in how
a step is chosen. Attribution reads a coalescence of the \emph{performance}
and lets the far-side principle pick which of the two incoming
actors performs the new ghost; the terminal value of the chain
defines the output casting's bijection~$\pi$
(\Cref{prop:attribution-chain}). Rehearsal reads a
crossing of two representatives in the \emph{casting} and, when the
validity test passes, extends~$\sigma_t$ by the coalescence whose
ghost performer is~$\pi^{-1}(g)$ (\Cref{def:valid-spurious}). To
prove that the two runs agree, it suffices to consider a single
step, assume the two assignments equal before it, and check that
the steps coincide---then induct. One point deserves care: $\pi$ is
attribution's \emph{terminal} output, yet rehearsal needs it from
the start. There is no circularity, because rehearsal consults~$\pi$
only through the single value $\pi^{-1}(g)$ at the coalescence
firing~$g$, and the chain settles that value at this very vertex and
never revises it.

\fi

\subsubsection{Weight as a regrouped product}
\label{sec:weight-regrouped}

Read as a relabeling, attribution makes its central property---that it
preserves weight---a matter of bookkeeping rather than calculation. The
weight of a performance is a product over its edges grouped by the
genealogy,
\[
w(\perf) = \prod_{\text{tree edges } e} w(e) \cdot
           \prod_{g \in \Ghosts} w(\Gamma_g)
\]
(\Cref{def:performance}); the weight of a casting is the same product
grouped instead by actor, $w(\paths) = \prod_{I} w(P_I)$. The two multiply
the \emph{same} multiset of edge weights---the dictionary only re-sorts
the factors, from grouped-by-entity to grouped-by-actor, and the ``at
most one particle, any number of ghosts'' tally ensures every factor is
claimed exactly once. Hence $w(\perf) = w(\paths)$, the
weight-preservation that underlies the performance--casting bijection
(\Cref{prop:performance-casting-bijection}).

\subsection{The easy direction: rehearsal undoes attribution}
\label{sec:proof-bijection-easy}

\fnote{%
  prose={Attribution then rehearsal is the identity; a round-trip theorem
    in \texttt{Combinatorics/Operations/Inverse.lean}.},
  lean={Coalescence.canonical_casting}}%
\begin{lemma}[Coupling: attribution then rehearsal]\label{lem:reh-attr}%

Let $\perf$ be a performance and $\sigma_t$ attribution's running
assignment along it (\Cref{prop:attribution-chain}), indexed by the
vertices~$t$. Then rehearsal on the casting $\Attribution(\perf)$ holds the
same $\sigma_t$ at every~$t$; in particular it runs without a spurious
crossing and returns~$\perf$, so $\Rehearsal \circ \Attribution = \id$.
\end{lemma}

\begin{proof}
This direction is almost definitional: what attribution emits at a
coalescence, rehearsal reads straight back. We follow a single binary
coalescence and induct over the coalescences in the linear
order~$\le$, both runs starting at the identity assignment.

First, rehearsal advances $\sigma_t$ at exactly the coalescences of~$\perf$.
Between coalescences the casting's active representatives do not cross: an
actor's path runs along its genealogy-tree trajectory; trajectories of one
tree meet only at their mergers, and distinct trees are vertex-disjoint
(\Cref{def:genealogy-forest}); and a ghost path issues only from the
vertex that consumes its performer (\Cref{def:ghost-paths}), an actor
that is from then on no longer a representative. So the only
crossings are the merges themselves.

At such a merge two adjacent active intervals $I^-, I^+$ meet at junction~$g$,
one heir flowing on and the ghost~$g$ born. Attribution assigns the ghost
role to the actor that will follow the ghost path, namely $\pi^{-1}(g)$,
placed on the side $\varepsilon_g$ dictates
(\Cref{fig:simplified-coalescence-negative,fig:simplified-coalescence}); the
casting it emits records which actors visit~$g$, and how it stitches their
paths through the vertex plays no part here. Reading the casting,
rehearsal recovers the same coalescence: the merge of $I^-, I^+$ at~$g$ is
forced, and the ghost performer it reconstructs is the actor $\pi^{-1}(g)$
attribution just named---fixed once~$g$ fires, as the chain only grows
(\Cref{prop:attribution-chain}). No spurious crossing can arise, the casting
issuing from a genuine performance. The two updates of $\sigma_t$
coincide, so the two chains agree throughout.

It remains to check that the performance rehearsal assembles
is~$\perf$ itself, paths included. Attribution built each casting
path by gluing segments of~$\perf$---tree edges and ghost paths---at
the coalescence vertices (\Cref{def:attribution}); rehearsal cuts
the casting paths back at the same vertices, the two runs sharing
their coalescences. The suffix it records as the ghost path of~$g$
is therefore the ghost path~$\Gamma_g$ of~$\perf$, and the segments
between consecutive mergers---with the final suffixes that close the
trees---are the tree edges of~$\perf$. The genealogy trees and the
ghost paths coincide, and rehearsal returns~$\perf$.
\end{proof}

\subsection{The hard direction: attribution undoes rehearsal}
\label{sec:proof-bijection-hard}

\fnote{%
  prose={Rehearsal then attribution is the identity; the other round-trip
    in \texttt{Combinatorics/Operations/Inverse.lean}.},
  lean={Coalescence.canonical_casting}}%
\begin{lemma}[Coupling: rehearsal then attribution]\label{lem:attr-reh}%

Let $\casting = (\pi, \paths)$ be a successful casting, $\sigma_t$
rehearsal's running assignment on it, and $\perf = \Rehearsal(\casting)$
(of final state~$\FinalState$, \Cref{prop:rehearsal-boundary}). Then
attribution on~$\perf$ holds the same $\sigma_t$ at every~$t$; in
particular $\Attribution(\perf) = \casting$, so
$\Attribution \circ \Rehearsal = \id$.
\end{lemma}

\begin{proof}
This is the substantive direction, because rehearsal discards information
that attribution must restore. As in the easy direction, both runs
advance $\sigma_t$ at the same steps: the coalescences of~$\perf$ are
exactly the valid crossings that rehearsal processed
(\Cref{def:successful-casting}), and attribution, reading~$\perf$,
meets them in the same order~$\le$. It therefore suffices to follow a
single such coalescence and induct.

At a junction~$g$ two active representatives cross. The four path-segments
incident to the vertex---two incoming, two outgoing---are paired in the
casting, but rehearsal, recording only the merge and the ghost it frees,
\emph{disconnects} them: the performance it builds no longer says which
incoming segment continued as which outgoing one. This is no accident of the
algorithm but of the representation: a performance names what merges and what
is freed, never the actor-paths themselves (\Cref{sec:attribution-spacetime}).
To return the original casting, attribution must reconnect the four segments
the way rehearsal found them.

There are exactly two reconnections, the two gluings of the two-particle
case: $\lambda_-$ draws the ghost from $I^-$, $\lambda_+$ from $I^+$
(\Cref{fig:simplified-coalescence-negative,fig:simplified-coalescence}).
Rehearsal's choice of ghost performer---$\pi^{-1}(g)$, the only choice
$\sigma \leq \pi$ admits (\Cref{lem:sigma-le-pi})---singles out one of
the two. Attribution chooses
by the ghost sign, drawing the ghost from the side $\varepsilon_g$ dictates.
The two choices agree because the casting is an $\varepsilon$-candidate,
which ties $\varepsilon_g$ to the side of $\pi^{-1}(g)$:
$g \ilo \pi^{-1}(g) \iff \varepsilon_g = +1$ (\Cref{def:candidate}). So the
side $\varepsilon_g$ sends attribution to is the side $\pi^{-1}(g)$ sits on,
and attribution reapplies the very gluing rehearsal used. The lost pairing is
recovered from $\varepsilon_g$, which the casting carries---never stored as
path data. The two updates of $\sigma_t$ coincide, so $\sigma_t$ agrees
throughout and attribution returns~$\casting$.
\end{proof}

The two couplings compose to the identity in both orders, and the
weight bill has already been paid (\Cref{sec:weight-regrouped}):

\fnote{%
  prose={The weight-preserving performance--casting bijection.},
  lean={Coalescence.canonical_casting}}%
\begin{proposition}[Performance--casting bijection]\label{prop:performance-casting-bijection}
For fixed final state~$\FinalState$, attribution and rehearsal are
mutually inverse, weight-preserving bijections between performances and
successful castings: corresponding objects satisfy
$w(\perf) = w(\paths)$.
\end{proposition}

\begin{proof}
Immediate from \Cref{lem:reh-attr,lem:attr-reh}; weight is preserved
because performance and casting multiply the same multiset of edge
weights (\Cref{sec:weight-regrouped}).
\end{proof}

\section{The sign-reversing involution}\label{sec:involution}

Rehearsal (\Cref{sec:rehearsal}) sorts $\varepsilon$-candidate castings into
successful and failed ones. To prove the determinant formula we show
that the failed castings cancel in signed pairs, leaving only the
successful ones. The pairing mechanism is \emph{segment swap}, the
classical first-crossing exchange (\Cref{sec:classic-proof}); the one
new ingredient is \emph{where} to apply it. As previewed in
\Cref{sec:intro-lgv}, the classical Karlin--McGregor /
Lindstr\"om--Gessel--Viennot proof forbids every crossing and swaps at
the first one; here crossings are \emph{prescribed} by the coalescence
pattern, and rehearsal locates the first \emph{spurious} crossing---a
place where the aspiration~$\pi$ calls for a coalescence that the
reality~$\sigma$ cannot carry out. Segment swap edits~$\pi$ there by the
transposition~$(I\;J)$, so the involution pairs each failed casting with
the one differing only in that single spurious crossing; the
$\varepsilon$-candidacy condition guarantees the swapped casting is again
an $\varepsilon$-candidate, and the surviving fixed points are exactly the
castings whose aspiration is realizable throughout.
\Cref{sec:proof-swap} establishes segment swap;
\Cref{sec:proof-global-involution} assembles the involution~$\iota$,
whose fixed points are the successful castings---already identified with
performances in \Cref{sec:proof-bijection}.

\subsection{Segment swap}
\label{sec:proof-swap}

\subsubsection{Definition and basic properties}

Given a casting $(\pi, \paths)$ and two actors $I$, $J$
whose paths cross, segment swap exchanges their suffixes at
the first crossing~$v$: the new path $P_I'$ follows $P_I$
to~$v$, then continues along $P_J$; symmetrically for
$P_J'$. The bijection updates accordingly:
$\pi' = (I\; J) \circ \pi$, exchanging the destinations of
$I$ and~$J$. See \Cref{fig:segment-swap}.

\begin{figure}[t]
\centering
\captionsetup[subfigure]{justification=centering}

\subfloat[Before swap: paths cross at $c$]{%
\begin{tikzpicture}[scale=0.6,
    vertex/.style={circle, fill, inner sep=1.5pt}]

\draw[gray, thick] (0, 0) -- (6, 0);
\draw[gray, thick] (0, 4) -- (6, 4);
\node[right] at (6.2, 0) {\small $t=0$};
\node[right] at (6.2, 4) {\small $t=T$};

\node[vertex, colA] (x1) at (1,0) {};
\node[below] at (x1) {$x_1$};
\node[vertex, colB] (x2) at (3,0) {};
\node[below] at (x2) {$x_2$};

\node[vertex, colA] (y1) at (2,4) {};
\node[above] at (y1) {$y_1$};
\node[vertex, colB] (y2) at (5,4) {};
\node[above] at (y2) {$y_2$};

\node[vertex, black] (v) at (3,2) {};
\node[right=0.2] at (v) {$c$};

\draw[particleA] (x1) -- (2,1) -- (v);
\draw[particleA] (v) -- (y1);

\draw[particleB] (x2) -- (v);
\draw[particleB] (v) -- (4,3) -- (y2);

\begin{scope}[shift={(7,1)}]
  \draw[particleA] (0,1.2) -- (0.7,1.2);
  \node[right] at (0.8,1.2) {\small $P_1$};
  \draw[particleB] (0,0.4) -- (0.7,0.4);
  \node[right] at (0.8,0.4) {\small $P_2$};
\end{scope}

\end{tikzpicture}%
\label{fig:swap-before}%
}
\hfill
\subfloat[After swap: endpoints exchanged]{%
\begin{tikzpicture}[scale=0.6,
    vertex/.style={circle, fill, inner sep=1.5pt}]

\draw[gray, thick] (0, 0) -- (6, 0);
\draw[gray, thick] (0, 4) -- (6, 4);
\node[right] at (6.2, 0) {\small $t=0$};
\node[right] at (6.2, 4) {\small $t=T$};

\node[vertex, colA] (x1) at (1,0) {};
\node[below] at (x1) {$x_1$};
\node[vertex, colB] (x2) at (3,0) {};
\node[below] at (x2) {$x_2$};

\node[vertex, colB] (y1) at (2,4) {};
\node[above] at (y1) {$y_1$};
\node[vertex, colA] (y2) at (5,4) {};
\node[above] at (y2) {$y_2$};

\node[vertex, black] (v) at (3,2) {};
\node[right=0.2] at (v) {$c$};

\draw[particleA] (x1) -- (2,1) -- (v);
\draw[particleA] (v) -- (4,3) -- (y2);

\draw[particleB] (x2) -- (v);
\draw[particleB] (v) -- (y1);

\begin{scope}[shift={(7,1)}]
  \draw[particleA] (0,1.2) -- (0.7,1.2);
  \node[right] at (0.8,1.2) {\small $P'_1$};
  \draw[particleB] (0,0.4) -- (0.7,0.4);
  \node[right] at (0.8,0.4) {\small $P'_2$};
\end{scope}

\end{tikzpicture}%
\label{fig:swap-after}%
}
\caption{\textbf{Segment swap: the sign-reversing involution.}
(a)~Paths $P_1$ (solid) and $P_2$ (double) cross at vertex $c$.
(b)~After the swap, final segments are exchanged: $P'_1$ follows $P_1$
to $c$, then $P_2$'s tail to $y_2$; $P'_2$ follows $P_2$ to $c$, then
$P_1$'s tail to $y_1$. The bijection updates to $\pi' = (1\ 2) \circ \pi$,
reversing the sign. Failed castings cancel in pairs via this operation.}
\label{fig:segment-swap}
\end{figure}

\fnote{%
  prose={Segment swap is involutive, weight-preserving, and sign-reversing.},
  lean={Coalescence.CastingHopChain.segmentSwap, Coalescence.Valuation.segment_swap_property},
  py={valuations.segment_swap_involution}}%
\begin{lemma}[Segment swap is involutive, weight-preserving, and sign-reversing]\label{lem:swap-properties}
Segment swap is:
\begin{enumerate}[label=(\roman*),
                   ref=(\roman*)]
\item \label{itm:swap-involution}
  \textbf{An involution}: swapping twice recovers the original.
\item \label{itm:swap-weight}
  \textbf{Weight-preserving}: same path segments, same total weight.
\item \label{itm:swap-sign}
  \textbf{Sign-reversing}: $\gasgnof{\pi'} = -\gasgnof{\pi}$.
\end{enumerate}
\end{lemma}

\begin{proof}
\ref{itm:swap-involution} Swapping at $v$ twice restores
the original paths and bijection.
\ref{itm:swap-weight} The path segments are redistributed
but the multiset of edges is unchanged, so total weight is
preserved.
\ref{itm:swap-sign} The updated bijection
$\pi' = (I\; J) \circ \pi$ differs from $\pi$ by a transposition, so
$\gasgnof{\pi'} = -\gasgnof{\pi}$ by the transposition rule
(\Cref{cor:gasgn-transposition}).
\end{proof}

\subsubsection{The swap criterion}

Segment swap of an arbitrary $\varepsilon$-candidate casting need not itself
be an $\varepsilon$-candidate. The next lemma pins down exactly when it is,
at the adjacent pair rehearsal crosses.

\fnote{%
  prose={The swap criterion: a swap at a spurious crossing stays an
    $\varepsilon$-candidate.},
  lean={Coalescence.CastingHopChain.segmentSwap},
  py={valuations.segment_swap_involution}}%
\begin{lemma}[Swap criterion]\label{lem:swap-criterion}
At each iteration of rehearsal, let
$I^- \ilo I^+$ be the adjacent active intervals at
junction~$g$, with representative actors
$I = \sigma^{-1}(I^-)$ and $J = \sigma^{-1}(I^+)$ whose paths cross.
Then the segment swap of the casting at $I, J$ is again an
$\varepsilon$-candidate if and only if
\[
  g \notin \{\pi(I), \pi(J)\},
\]
that is, unless one of the two crossing representatives aspires to
ghost~$g$.
\end{lemma}

\begin{proof}
Segment swap exchanges $\pi$ at $I$ and~$J$ and leaves
$\pi$ unchanged elsewhere; the swapped casting fails
$\varepsilon$-candidacy precisely when the swap moves some ghost's
performer to the wrong side. We check each ghost.

\medskip\noindent\emph{The current junction $h = g$.}
The performer at~$g$ changes under the swap if and only if
$g \in \{\pi(I), \pi(J)\}$; in that case the new performer
is the other actor of~$\{I, J\}$. Since $I \in I^- = [a, g)$
and $J \in I^+ = [g, c)$ lie on opposite sides of~$g$, the
performer of~$g$ switches sides, so $\varepsilon$-candidacy at~$g$ flips. If
$g \notin \{\pi(I), \pi(J)\}$ the performer at~$g$ is
unchanged and $\varepsilon$-candidacy at~$g$ is preserved.

\medskip\noindent\emph{Ghosts $h \neq g$ outside $I^- \cup I^+$.}
Both $I$ and~$J$ lie inside the contiguous interval
$I^- \cup I^+$, hence on the same side of~$h$. Whether or
not the performer at~$h$ changes, it stays on the same side
of~$h$, so $\varepsilon$-candidacy at~$h$ is preserved.

\medskip\noindent\emph{Ghosts $h \neq g$ inside $I^- \cup I^+$.}
Such an $h$ lies interior to the active interval $I^-$ or $I^+$, so it is
already a ghost of $P_\sigma$. By $\sigma \leq \pi$
(\Cref{lem:sigma-le-pi}) its performer is pinned,
$\pi^{-1}(h) = \sigma^{-1}(h)$---the actor consumed when $h$ fired, no
longer a survivor. The crossing representatives $I$ and $J$ are survivors,
so $\pi^{-1}(h) \notin \{I, J\}$; the swap does not affect the performer
at~$h$, and $\varepsilon$-candidacy at~$h$ is preserved.

\medskip

Combining the three cases, the swapped casting fails
$\varepsilon$-candidacy if and only if $g \in \{\pi(I), \pi(J)\}$;
equivalently, it remains an $\varepsilon$-candidate exactly when
$g \notin \{\pi(I), \pi(J)\}$.
\end{proof}

\fnote{%
  prose={A segment swap at a spurious failure stays an $\varepsilon$-candidate.},
  lean={Coalescence.segmentSwapAtFailure},
  py={valuations.segment_swap_involution}}%
\begin{corollary}[Segment swap at a failure pair lands in an $\varepsilon$-candidate]%
\label{cor:swap-at-failure-fires}
If rehearsal on an $\varepsilon$-candidate casting~$\casting$ halts at the
failure pair~$(I, J)$, then the segment swap of $\casting$ at~$(I, J)$
is itself an $\varepsilon$-candidate casting.
\end{corollary}

\begin{proof}
At a spurious crossing $g \notin \{\pi(I), \pi(J)\}$, so the swap
remains an $\varepsilon$-candidate by \Cref{lem:swap-criterion}.
\end{proof}

\subsection{The involution}
\label{sec:proof-global-involution}

Combining rehearsal (the classifier) with segment swap (the
pairing mechanism) defines the involution on
$\varepsilon$-candidate castings.

\fnote{%
  prose={The involution $\iota$: rehearse, then segment-swap at the failure
    pair.},
  lean={Coalescence.segmentSwapAtFailure},
  py={valuations.segment_swap_involution}}%
\begin{definition}[The involution $\iota$]
\label{def:global-involution}
For each $\varepsilon$-candidate casting~$\casting$, set
\[
  \iota(\casting) =
    \begin{cases}
      \casting
        & \text{if rehearsal on $\casting$ succeeds,} \\[2ex]
      \text{segment swap of $\casting$ at $(I, J)$}
        & \text{if rehearsal on $\casting$} \\
        & \qquad\text{halts at the
                failure pair $(I, J)$.}
    \end{cases}
\]
\end{definition}

\fnote{%
  prose={Failed castings cancel pairwise under $\iota$; the successful
    castings are its fixed points (layer~1).},
  lean={Coalescence.MainTheorem.main_theorem_layer1, Coalescence.segmentSwapAtFailure},
  py={valuations.segment_swap_involution}}%
\begin{theorem}[Successful castings are the fixed points of $\iota$]\label{thm:involution}
The map $\iota$ is a weight-preserving, sign-reversing
involution on $\varepsilon$-candidate castings. Its fixed points are
exactly the successful castings.
\end{theorem}

\begin{proof}
\emph{Well-defined on $\varepsilon$-candidates.} If rehearsal succeeds
on~$\casting$ then $\iota(\casting) = \casting$ is $\varepsilon$-candidate.
If rehearsal halts at $(I, J)$, the segment swap at $(I, J)$ lands
in an $\varepsilon$-candidate (\Cref{cor:swap-at-failure-fires}).

\medskip\noindent\emph{Fixed points.} By construction,
$\iota(\casting) = \casting$ iff rehearsal succeeds
on~$\casting$, iff $\casting$ is successful
(\Cref{def:successful-casting}).

\medskip\noindent\emph{Involution.} We treat the two branches
of $\iota$'s definition separately. On a successful casting,
$\iota$ is the identity, so $\iota^2 = \id$ trivially.

On a failed casting~$\casting$ with failure pair $(I, J)$ at
vertex~$v$ and junction~$g$, set
$\casting'$ to be the segment swap of $\casting$ at $(I, J)$. Here the
failure vertex~$v$ is the \emph{first} crossing of $P_I$ and~$P_J$, so
the segment swap at~$(I, J)$ (\Cref{sec:proof-swap}) acts exactly
at~$v$: since $I$ and~$J$ are both $\sigma$-active when rehearsal
reaches~$v$, any earlier crossing of their paths would---by
consecutivity (\Cref{sec:proof-adjacent}) and rehearsal's per-vertex
processing---already have fired a coalescence consuming one of $I, J$ or
halted rehearsal spuriously before~$v$, contradicting that both
reach~$v$ active. We must show that
$\iota(\casting') = \casting$. We claim that rehearsal on~$\casting'$
also halts at~$v$ with failure pair~$(I, J)$; then
$\iota(\casting')$ is the segment swap of $\casting'$ at $(I, J)$,
which is $\casting$ by involutivity of segment
swap (\Cref{lem:swap-properties}\ref{itm:swap-involution}).

\medskip\noindent\emph{Locality of the swap.} Segment swap at~$(I, J)$
modifies only the suffixes of $P_I$ and $P_J$ after~$v$;
every walk prefix up to and including~$v$, and every other
walk, is identical in $\casting$ and $\casting'$. Rehearsal's crossings
strictly before~$v$ therefore appear in the same chronological order on
$\casting$ and
$\casting'$ with identical geometry. Through every valid
crossing before~$v$, the same pair of active intervals
merges at the same junction, the same actor is consumed as
a ghost performer, and reality~$\sigma$ advances identically, so the
reality $\sigma$ and the partial performance built so far when rehearsal
reaches~$v$ are the same on $\casting$ and $\casting'$. The same applies
to any valid crossings processed at~$v$ itself before the pair $(I, J)$:
their consumed performers are no longer active, hence are neither $I$
nor~$J$, so the swap alters neither their tests nor their outcomes.

\medskip\noindent\emph{Invariance of the test at~$v$.} At~$v$ the paths of
$I$ and $J$ still cross. The validity test asks whether
$g \in \{\pi(I), \pi(J)\}$. The swap transposes
$\pi$ at $I$ and~$J$, exchanging $\pi(I)$ and
$\pi(J)$; the unordered pair
$\{\pi(I), \pi(J)\}$ is unchanged. The test
still fails, so rehearsal on~$\casting'$ halts at the same
vertex~$v$ with failure pair~$(I, J)$.

\medskip\noindent\emph{Properties.} Weight-preservation and
sign-reversal follow from \Cref{lem:swap-properties}.
\end{proof}

\section{Proof of the coalescence formula}\label{sec:proof-main}

With the bijection between performances and successful castings
and the involution~$\iota$ in place, two pieces remain:
the sign identity, showing that every successful casting is
\emph{positive} for the ghost-adjusted sign, and the final assembly, which
reads $\det M$ off the restricted Leibniz expansion
(\Cref{prop:restricted-leibniz}) and cancels the failed castings against
one another.

\subsection{The sign identity}
\label{sec:proof-sign}

\fnote{%
  prose={Successful castings all carry ghost-adjusted sign $+1$.},
  lean={Coalescence.MainTheorem.main_theorem_layer1}}%
\begin{proposition}[Successful castings are positive]\label{prop:sign-identity}
For any successful casting $\casting$, the ghost-adjusted sign is
\[
\gasgnof{\pi} = +1
\qquad\text{(equivalently, } \sgn \pi = \prod_{g \in \Ghosts}
\varepsilon_g \text{).}
\]
\end{proposition}

\begin{proof}
Every successful casting is the attribution of a performance---the
surjectivity of attribution, the $\Attribution \circ \Rehearsal =
\id$ direction of \Cref{prop:performance-casting-bijection}---so
$\gasgnof{\pi} = +1$ by \Cref{prop:attribution-positive}.
\end{proof}

So attribution lands among the positive castings: each successful
casting contributes exactly $+1 \cdot w(\paths)$ to the determinant.

\subsection{Completing the proof}
\label{sec:proof-completion}

\begin{proof}[Proof of \Cref{thm:coalescence}]
The restricted Leibniz expansion (\Cref{prop:restricted-leibniz}) writes
the determinant as a signed sum over $\varepsilon$-candidate castings,
\[
\det M = \sum_{\substack{(\pi, \paths) \\ \pi \in \Candidates}}
         \gasgnof{\pi} \, w(\paths).
\]
The involution $\iota$ (\Cref{thm:involution}) partitions the
$\varepsilon$-candidate castings into its fixed points---the successful
castings---and two-element orbits of failed castings. Segment swap
preserves the weight and negates the ghost-adjusted sign
(\Cref{lem:swap-properties}), so the two castings of each failed orbit
cancel. Every surviving fixed point is a successful casting, which is
positive by \Cref{prop:sign-identity}; hence
\[
\det M
= \sum_{\text{successful castings}} \gasgnof{\pi}\, w(\paths)
= \sum_{\text{successful castings}} w(\paths).
\]
Finally, the weight-preserving bijection between successful castings
and performances
(\Cref{prop:performance-casting-bijection}) identifies the right-hand
side with the total weight~$Z$. Therefore $Z = \det M$.
\end{proof}

\subsection{Example: full coalescence of three particles}
\label{sec:proof-example}

We return to \Cref{ex:attribution-n3}: three particles $I_1, I_2, I_3$
coalesce into one heir $H = [1,4)$,
with $\varepsilon_2 = +1$ and $\varepsilon_3 = -1$.
The $\varepsilon$-candidacy condition requires:
\begin{itemize}
\item $\pi^{-1}(2) \igo 2$, so $\pi^{-1}(2) \in \{I_2, I_3\}$;
\item $\pi^{-1}(3) \ilo 3$, so $\pi^{-1}(3) \in \{I_1, I_2\}$.
\end{itemize}

Three bijections satisfy these constraints:
\begin{align*}
\pi_1 &\colon I_1 \mapsto H,\; I_2 \mapsto 3,\; I_3 \mapsto 2, \\
\pi_2 &\colon I_1 \mapsto 3,\; I_2 \mapsto 2,\; I_3 \mapsto H, \\
\pi_3 &\colon I_1 \mapsto 3,\; I_2 \mapsto H,\; I_3 \mapsto 2.
\end{align*}

\subsubsection{Which candidates come from a performance}

The bijections $\pi_1$ and $\pi_2$ each arise from a coalescence order
whose attribution produces it---the two orders
worked out in \Cref{ex:attribution-n3} (\Cref{fig:n3-pi1,fig:n3-pi2}),
in which junction~$3$ fires first for $\pi_1$ and junction~$2$ first for
$\pi_2$. Both are positive, $\gasgnof{\pi_1} = \gasgnof{\pi_2} = +1$, as
guaranteed by \Cref{prop:sign-identity}: here
$\sgn \pi_1 = \sgn \pi_2 = -1$ and
$\prod_g \varepsilon_g = \varepsilon_2\varepsilon_3 = -1$ multiply
to $+1$.

The bijection $\pi_3$, by contrast, arises from no coalescence order: it
satisfies the $\varepsilon$-candidacy constraint but admits no successful casting
(\Cref{fig:unrealizable-bijection}). The middle particle $I_2$ must
coalesce with a neighbor before reaching the final time. If $I_2$
coalesces first with $I_1$, then ghost~$2$ is created---but the
ghost must be $I_1$ or $I_2$, not $I_3$ as $\pi_3$ requires.
If $I_2$ coalesces first with $I_3$, then ghost~$3$ is
created---but the ghost
must be $I_2$ or $I_3$, not $I_1$ as $\pi_3$ requires.

\begin{figure}[t]
\centering

\subfloat[]{
\begin{tikzpicture}[scale=1.0]

\def\width{5.5}
\def\height{3}

\tikzset{
    styleA/.style={line width=2pt},
    styleB/.style={double, double distance=2pt, line width=0.6pt},
    styleC/.style={line width=1.5pt, decorate, decoration={zigzag, amplitude=1pt, segment length=4pt}},
}

\draw[gray, thick] (-\width/2, 0) -- (\width/2, 0);
\draw[gray, thick] (-\width/2, \height) -- (\width/2, \height);
\node[right] at (\width/2, 0) {$\mathcal{X}$};
\node[right] at (\width/2, \height) {$\mathcal{Y}$};

\pgfmathsetmacro{\xA}{-1.8}
\pgfmathsetmacro{\xB}{0}
\pgfmathsetmacro{\xC}{1.8}
\pgfmathsetmacro{\xGtwo}{-1.5}
\pgfmathsetmacro{\xS}{0}
\pgfmathsetmacro{\xGthree}{1.5}

\coordinate (I1) at (\xA, 0);
\coordinate (I2) at (\xB, 0);
\coordinate (I3) at (\xC, 0);

\coordinate (g2) at (\xGtwo, \height);
\coordinate (heir) at (\xS, \height);
\coordinate (g3) at (\xGthree, \height);

\draw[styleA, colA] (I1) -- (g3);
\draw[styleB, colB] (I2) -- (heir);
\draw[styleC, colC] (I3) -- (g2);

\node[circle, fill=black, inner sep=2pt] at (I1) {};
\node[circle, fill=black, inner sep=2pt] at (I2) {};
\node[circle, fill=black, inner sep=2pt] at (I3) {};

\node[circle, fill=black, inner sep=2.5pt] at (heir) {};
\node[circle, draw=black, fill=white, line width=1pt, inner sep=1.5pt] at (g2) {};
\node[circle, draw=black, fill=white, line width=1pt, inner sep=1.5pt] at (g3) {};

\node[below=0.1cm] at (I1) {$I_1$};
\node[below=0.1cm] at (I2) {$I_2$};
\node[below=0.1cm] at (I3) {$I_3$};
\node[above=0.1cm] at (g2) {$2$};
\node[above=0.1cm] at (heir) {$H$};
\node[above=0.1cm] at (g3) {$3$};

\node[font=\fontsize{24}{24}\selectfont\bfseries, text=red] at (\width/2 - 0.2, \height/2) {!};

\end{tikzpicture}
\label{fig:unrealizable-bijection-a}
}
\hfill
\subfloat[]{
\begin{tikzpicture}[scale=0.8]

\def\axislen{6}
\def\height{3.75}
\def\rectheight{0.4}
\def\spacing{1.3}

\draw[->, thick] (-0.3, 0) -- (\axislen, 0);
\foreach \x in {1,2,3,4} {
    \draw (\spacing*\x, -0.1) -- (\spacing*\x, 0.1);
    \node[below] at (\spacing*\x, -0.2) {\small $\x$};
}

\draw[->, thick] (-0.3, \height) -- (\axislen, \height);
\foreach \x in {1,2,3,4} {
    \draw (\spacing*\x, \height-0.1) -- (\spacing*\x, \height+0.1);
    \node[above] at (\spacing*\x, \height+0.2) {\small $\x$};
}

\draw[fill=colA!30, draw=colA, thick, rounded corners=3pt]
    (\spacing*1+0.07, 0.15) rectangle (\spacing*2-0.07, 0.15+2*\rectheight);
\node[circle, fill=colA, inner sep=1.5pt] (I1circ) at (\spacing*1+0.2, 0.15+\rectheight) {};
\node at (\spacing*1.5, 0.15+\rectheight) {\small $I_1$};

\draw[fill=colB!30, draw=colB, thick, rounded corners=3pt]
    (\spacing*2+0.07, 0.15) rectangle (\spacing*3-0.07, 0.15+2*\rectheight);
\node[circle, fill=colB, inner sep=1.5pt] (I2circ) at (\spacing*2+0.2, 0.15+\rectheight) {};
\node at (\spacing*2.5, 0.15+\rectheight) {\small $I_2$};

\draw[fill=colC!30, draw=colC, thick, rounded corners=3pt]
    (\spacing*3+0.07, 0.15) rectangle (\spacing*4-0.07, 0.15+2*\rectheight);
\node[circle, fill=colC, inner sep=1.5pt] (I3circ) at (\spacing*3+0.2, 0.15+\rectheight) {};
\node at (\spacing*3.5, 0.15+\rectheight) {\small $I_3$};

\draw[fill=colP!15, draw=colP, thick, rounded corners=3pt]
    (\spacing*1+0.07, \height-0.15-2*\rectheight) rectangle (\spacing*4-0.07, \height-0.15);
\node[circle, fill=colP, inner sep=1.5pt] (Hcirc) at (\spacing*1+0.3, \height-0.15-\rectheight) {};
\node at (\spacing*2.5, \height-0.15-\rectheight) {$H$};

\node[circle, draw=colC, fill=white, line width=1.5pt, inner sep=2pt] (g2circ)
    at (\spacing*2, \height-0.15-2*\rectheight-0.35) {};
\node[right] at (\spacing*2+0.15, \height-0.15-2*\rectheight-0.35) {\small $2$};

\node[circle, draw=colA, fill=white, line width=1.5pt, inner sep=2pt] (g3circ)
    at (\spacing*3, \height-0.15-2*\rectheight-0.35) {};
\node[right] at (\spacing*3+0.15, \height-0.15-2*\rectheight-0.35) {\small $3$};

\draw[colA, thick, dashed, ->] (I1circ) to[out=90, in=-90] (g3circ);
\draw[colB, thick, ->] (I2circ) to[out=90, in=-90] (Hcirc);
\draw[colC, thick, dashed, ->] (I3circ) to[out=90, in=-90] (g2circ);

\node[font=\fontsize{20}{20}\selectfont\bfseries, text=red]
    at (\axislen + 0.5, \height/2) {\textsf{X}};

\end{tikzpicture}
\label{fig:unrealizable-bijection-b}
}

\caption{\textbf{A candidate with no successful casting: $\pi_3$.}
The bijection $\pi_3$: $I_1 \mapsto 3$, $I_2 \mapsto H$, $I_3 \mapsto 2$
satisfies the candidacy condition, yet every one of its castings fails.
(a)~Any path family has a spurious crossing---the middle particle $I_2$
must collide before reaching $H$.
(b)~Bijection $\pi_3$, with $\gasgn \pi_3 = -1$; its castings all fail
and cancel under the involution. Compare with
\Cref{fig:n3-pi1,fig:n3-pi2}.}
\label{fig:unrealizable-bijection}
\end{figure}

\subsubsection{The involution pairs failed castings}

Every $\pi_3$-casting fails at its first crossing (either
$P_1 \cap P_2$ or $P_2 \cap P_3$). The segment swap at this crossing
converts $\pi_3$-castings into $\pi_1$- or $\pi_2$-castings with
wrong crossing order. Conversely, $\pi_1$- and $\pi_2$-castings with
wrong crossing order swap back to $\pi_3$-castings. The pairings
preserve weight and flip sign, so all failed castings cancel. Only
correctly-ordered $\pi_1$- and $\pi_2$-castings survive---exactly
those arising from performances.

\section{Continuous processes}\label{sec:continuous}

The discrete framework extends to continuous time and space. The
determinant identity of \Cref{thm:coalescence} was proved for
discrete-time, discrete-space walks; we now show that the same identity
holds for any Markov process satisfying the Karlin--McGregor
assumptions---in particular Brownian motion---by reading the discrete
proof as a statement about measures rather than counts
(\Cref{thm:continuous}). The combinatorial core ($\varepsilon$-candidates, segment
swap, the sign-reversing involution) is untouched; only the bookkeeping
of outcomes changes from summation to integration.

The dictionary with the discrete proof is direct. In the discrete proof a
\emph{spacetime point} is a vertex $(x,t)$---space horizontal, time
upward (\Cref{fig:intro-21})---and the sign-reversing involution runs along
$\le$, which on these vertices $(x,t)$ we realize as the lexicographic order
(a linear extension of the time order of \Cref{def:spacetime-graph}),
comparing the time coordinate~$t$ first. Here the spacetime points are arbitrary $(x,t)$
in space-time, and the same order serves unchanged. The one combinatorial
subtlety---several particles meeting at a single spacetime point---is
handled exactly as in the discrete construction, by resolving the multiple
meeting into a sequence of binary collisions (the construction below).

\subsection{Karlin--McGregor assumptions}
\label{sec:continuous-km}

Let $\Omega$ denote the probability space for $n$ non-interacting
particles with initial positions $x_1 \leq \cdots \leq x_n$. Each
$\omega \in \Omega$ specifies trajectories up to time $T$, with
$\mathbf{X}_T(\omega) = (X_T^1(\omega), \ldots, X_T^n(\omega))$
denoting the final positions.

The \emph{Karlin--McGregor assumptions}~\cite{KM1959} are:
\begin{enumerate}[label=\textup{(KM\arabic*)}, ref=\textup{KM\arabic*},
                  leftmargin=*]
\item \label{itm:km-first}\label{itm:km-markov}\emph{Strong Markov property}: the
  $n$-particle system $(X^1, \ldots, X^n)$ is strong Markov---for
  any stopping time $\tau$ of its joint filtration, the post-$\tau$
  system is conditionally independent of the pre-$\tau$ system given
  the state at~$\tau$;
\item \label{itm:km-identical}\emph{Identical, independent dynamics}:
the particles are independent ($\PP$ is the product law), each
following the same Markov process (differing only in initial
position);
\item \label{itm:km-order}\emph{Order preservation}: adjacent particles
cannot change their relative order without first occupying the same
state;
\item \label{itm:km-last}\label{itm:km-meeting}\emph{Meeting times are stopping times}: for
  particles $I < J$, the first meeting time
  $\tau_{I,J} = \inf\{t : X^I_t = X^J_t\}$ is a stopping time.
\end{enumerate}
Identical, independent dynamics~(\ref{itm:km-identical}) ensures that
when two particles meet, their post-meeting trajectories are
exchangeable (conditioned on the meeting state): this is the
measure-theoretic analog of weight-preserving segment swap. For
independent copies the joint strong Markov
property~(\ref{itm:km-markov}) is not automatic; it holds under mild
regularity---continuous paths and a jointly continuous transition
function~\cite[Section~6, Theorem~2]{KM1959}.
These hold for Brownian motion and other continuous-path
diffusions, as well as for skip-free birth-death chains. For
discrete-time $\pm 1$ walks on~$\ZZ$, order preservation
requires that all particles share the same parity (as enforced
by the checkerboard lattice of \Cref{ex:main-examples}).

\subsection{The finite coalescing system}
\label{sec:continuous-coalescing}

The formula below (\Cref{thm:continuous}) is an identity for $\PP_{\mathrm{int}}$, the law of the
\emph{coalescing system}. Before stating it we must say what that system is:
for particles moving in continuous space and time the very meaning of
coalescence calls for a construction, since there is no smallest time step at
which to resolve a collision. With finitely many particles this construction
is elementary, and the ghost method makes it transparent.

\subsubsection{The construction}

Fix a priority order on the $n$ particles and run the $n$ trajectories of
$\Omega$, the non-interacting system of \Cref{sec:continuous-km}. At the
first instant at which two particles occupy a common state, an heir and a
ghost emerge: the higher-priority particle takes the \emph{heir} role and
carries the merged
group forward along its own trajectory, while the lower-priority particle
takes the \emph{ghost} role. Order preservation~(\ref{itm:km-order}) forces
each collision to be between two particles currently adjacent in the order, so
every collision creates exactly one ghost and retires one active particle;
hence at most $n-1$ collisions occur and the construction terminates. Should
three or more particles meet at a single instant, the meeting is resolved as a
sequence of binary collisions in priority order, exactly as a high-indegree
vertex is in the discrete construction (\Cref{sec:proof-high-indegree}). Meeting
times are stopping times~(\ref{itm:km-meeting}) and the strong Markov
property~(\ref{itm:km-markov}) makes the evolution after each collision a
fresh instance of the same dynamics, so the recursion is well defined---it is
the very stopping-time bookkeeping used in \Cref{sec:continuous-measurability} to make
the casting space measurable.

\subsubsection{Retention makes existence immediate}

The standard constructions \emph{truncate} the loser at each
collision, as in Arratia's construction for Brownian
motion~\cite{Arratia1979}; the ghost method instead \emph{keeps} the
truncated tail as a ghost running on along its own independent
trajectory, so no path is ever stopped. All $n$ trajectories of
$\Omega$ survive intact, and coalescence becomes a relabeling of which
coordinate is read as an heir. Consequently $\PP_{\mathrm{int}}$ is
the pushforward of the non-interacting law $\PP$ under this
deterministic relabeling, and its existence is immediate.

\subsubsection{Generality}

The construction invokes only the Karlin--McGregor
assumptions~(\ref{itm:km-first}--\ref{itm:km-last}), never any special
feature of Brownian motion, so it covers the whole class of
\Cref{sec:continuous-km}. The law of the coalescing system does not
depend on the chosen priority order: identical, independent
dynamics~(\ref{itm:km-identical}) makes the post-collision trajectories
exchangeable---for Brownian motion, exactly Arratia's observation that
the coalescing flow does not depend on the precedence
rule~\cite{Arratia1979}. Passing to \emph{infinitely} many particles
requires machinery---coming down from infinity, a topology on path
families---that the finite systems of this paper do not
need~\cite{EvansMorrisSen2013,TothWerner1998,FINR2004}.

\subsection{The continuous-time ghost formula}
\label{sec:continuous-formula}

\fnote{%
  prose={Not formalized: a continuous-time object (admissible sets of real
    final positions). The discrete counterpart is the $\varepsilon$-candidate
    set; see \Cref{sec:lean-scope}.}}%
\begin{definition}[$\varepsilon$-admissible sets]\label{def:eps-admissible-set}
Fix a coalescence pattern, hence the role set $\Roles$, and a sign
vector~$\varepsilon$. An \emph{$\varepsilon$-admissible set} is a measurable
set $A \subseteq \mathbb{R}^{\Roles}$ of final positions such that every
$(y_f)_{f \in \Roles} \in A$ satisfies:
\begin{itemize}
\item \emph{ghost--heir signs:} for each ghost~$g$ we require
  $y_g \preceq y_{\heir(g)}$ when $\varepsilon_g = +1$ and
  $y_{\heir(g)} \prec y_g$ when $\varepsilon_g = -1$;
\item \emph{heir order:} $y_H \preceq y_{H'}$ whenever $H \ilo H'$.
\end{itemize}
Here the weak inequality~$\preceq$ for $\varepsilon_g = +1$ versus the
strict~$\prec$ for $\varepsilon_g = -1$ assigns a tie
$y_g = y_{\heir(g)}$ to $\varepsilon_g = +1$, as permitted by
\Cref{def:ghost-sign}.
\end{definition}

The heir-order condition is no restriction: as in the discrete setting
(\Cref{sec:setup-final}), order preservation~(\ref{itm:km-order}) keeps the
surviving particles in this order, so a configuration violating it has
probability zero under the coalescing system.

\fnote{%
  prose={Not formalized: the continuous-time, probabilistic formula. The
    Lean development imports no measure theory and certifies the discrete
    identity this is obtained from by marginalization; see
    \Cref{sec:lean-scope}.}}%
\begin{theorem}[Continuous-time ghost formula]\label{thm:continuous}
Let the underlying process satisfy the Karlin--McGregor assumptions.
For any $\varepsilon$-admissible set $A$:
\[
  \PP_{\mathrm{int}}(\text{final positions} \in A) =
  \sum_{\pi \in \Pi_A} \gasgnof{\pi} \,
  \PP\bigl( \mathbf{X}_T \in A_\pi \bigr),
\]
where $\PP_{\mathrm{int}}$ denotes the probability for the
coalescing system, $\PP$ denotes the probability for
non-interacting particles, and
\[
  A_\pi = \bigl\{(y_{\pi(1)}, \ldots, y_{\pi(n)}) :
  (y_1, \ldots, y_n) \in A\bigr\}
\]
is the set $A$ with coordinates permuted by $\pi$.
\end{theorem}

The candidate set $\Pi_A$ (\Cref{def:candidate}) depends only on the
coalescence pattern and~$\varepsilon$, and the ghost-adjusted sign
$\gasgnof{\pi}$ (\Cref{eq:ghost-adjusted-sign}) only on~$\pi$
and~$\varepsilon$; neither depends on the positions in~$A$, so the
right-hand side is a well-defined finite signed sum.

\begin{remark}[Scope]\label{rem:continuous-scope}
The theorem is stated for a single $\varepsilon$-admissible set $A$---one
coalescence pattern together with one sign vector~$\varepsilon$. An
arbitrary event on the final positions is recovered by partitioning it
according to coalescence pattern and~$\varepsilon$ and summing, so this
one identity determines the entire law of the final state of the
coalescing system.
\end{remark}

The formula is an identity between measures, making no reference to
densities or mass functions; it is the starting point for the Brownian
specializations developed in the companion
papers~\cite{Sniady2026coalescenceApplications,Sniady2026pfaffian}.

\subsection{Reduction to the discrete proof}
\label{sec:continuous-proof}

The proof of \Cref{thm:continuous} occupies the rest of this section. The
whole argument is one idea: the continuous identity is the discrete
determinant identity of \Cref{thm:coalescence} applied cell by cell, with
the passage between the two carried by a reformulation of rehearsal.

\subsubsection{The lazy reformulation}

As stated in \Cref{sec:rehearsal-algorithm}, rehearsal is \emph{eager}:
it sweeps through every spacetime vertex in the linear order, acting
only at the rare crossings. Nothing in the run depends on the inert
steps, so rehearsal has an equivalent \emph{lazy} form: from the
current state, jump straight to the first vertex at which two active
representatives meet, act there, and repeat. A continuum of spacetime
offers no next vertex to enumerate, but the first meeting of two
active representatives is still a well-defined instant---a stopping
time, as verified below---so it is the lazy description that survives
the passage to continuous spacetime.

\subsubsection{Finitely many cells}

Read this way, rehearsal makes only finitely many decisions: one at each
\emph{coalescence} of two active representatives (simultaneous meetings of
three or more reduce to binary collisions, as above). 
Each valid
coalescence retires one active representative, so the count of active
particles strictly decreases and at most $n-1$ coalescences can occur
before the run ends. The potential 
continuity of the paths is deceptive but harmless
here: two trajectories may share a position infinitely often---near a
meeting time, Brownian paths do so almost surely---yet only the
\emph{first} meeting of two still-active representatives is a decision,
after which one of them is retired, so the remaining infinitely many
coincidences are invisible to rehearsal. What records the run is therefore
not the (possibly infinite) set of meetings but its \emph{combinatorial
type}: which adjacent pairs coalesce, in what order, and with which
valid/spurious verdicts, forgetting the actual times and positions.
Pairing the finite candidate set $\Pi_A$ with $\Omega$ and grouping by
combinatorial type partitions the casting space into \emph{finitely many}
combinatorial cells.

Their finiteness is structural: a combinatorial type is
an increasing chain in the assignment poset
(\Cref{prop:attribution-chain})---the reality assignment $\sigma$ 
climbs from the
identity at the bottom to $(P_\FinalState, \pi)$, one cover step per
coalescence---and a finite poset has only finitely many chains, however
wildly the paths oscillate within each cell. From the dual side the same
finiteness is even more immediate: attribution reads a coalescing
performance \emph{produced} with at most $n-1$ coalescence points---the
collision instants where the ghosts emerge
(\Cref{sec:continuous-coalescing})---so the finite skeleton sits in the
object itself, before any sweep reads it.

\subsubsection{The discrete identity, cell by cell}

On each combinatorial cell every combinatorial object is constant: the
candidate $\pi$, the time-evolution of the assignment $\sigma$ up to
reparametrization, and the ghost-adjusted sign $\gasgnof{\pi}$. The
discrete proof uses only these order-theoretic data, never the numerical
positions, so it applies to each cell verbatim and the determinant
identity holds there unchanged. The \emph{only} ingredient genuinely new
in continuous time is that segment swap is measure-preserving and so
permutes the combinatorial cells; this is the single place where the
strong Markov property and the exchangeability of post-meeting
trajectories enter.

\subsubsection{What remains to verify}

For the integrals it suffices to work with a coarser partition: the
first-spurious-crossing partition used below groups castings by which pair
of actors owns the first spurious crossing (or by success), and each of
its blocks is a union of combinatorial cells, so the two share the same
measurability. The remainder of the proof makes the reduction
precise and verifies it against the discrete argument of
\Cref{sec:proof-setup,sec:proof-attribution,sec:rehearsal,sec:proof-bijection,sec:involution,sec:proof-main}:
that those blocks are measurable, that the swap is measure-preserving and
sign-reversing, and that the successful castings account for
$\PP_{\mathrm{int}}(\text{final positions} \in A)$.

\subsection{The casting space and its partition}

\subsubsection{The casting space}

\fnote{%
  prose={Not formalized: the continuous-time casting space. The discrete
    castings are \texttt{Casting} (\Cref{def:casting}); see
    \Cref{sec:lean-scope}.}}%
\begin{definition}[Casting space]\label{def:casting-space}
For an $\varepsilon$-admissible set $A$, the \emph{casting space} is
\[
  \CastingSpace_A = \bigl\{ (\pi, \omega) \in \Pi_A \times \Omega :
    \mathbf{X}_T(\omega) \in A_\pi \bigr\}.
\]
A \emph{casting} $(\pi, \omega) \in \CastingSpace_A$ pairs an $\varepsilon$-candidate
bijection with $n$ trajectories $(P_I)_{I \in \Actors}$; its
\emph{ghost-adjusted sign} is $\gasgnof{\pi}$ (\Cref{eq:ghost-adjusted-sign}). The
casting space carries the
product measure $\mu$: counting measure on $\Pi_A$ times the probability
measure $\PP$ on $\Omega$.
\end{definition}

\subsubsection{Partition into blocks}

The involution $\iota: \CastingSpace_A \to \CastingSpace_A$ acts by segment swap
at the first spurious crossing. Partition the casting space according
to which pair of actors has the first spurious crossing:
\[
  \CastingSpace_A =
  \CastingSpace_A^{\mathrm{succ}} \sqcup
  \bigsqcup_{I < J} \CastingSpace_A^{(I,J)},
\]
where:
\begin{itemize}
\item $\CastingSpace_A^{\mathrm{succ}}$ consists of \emph{successful}
  castings (no spurious crossing); the involution acts as the identity;
\item $\CastingSpace_A^{(I,J)}$ consists of castings whose first spurious
  crossing involves actors $I$ and $J$; the involution is the
  segment swap of paths $P_I, P_J$.
\end{itemize}

\subsubsection{Measurability}
\label{sec:continuous-measurability}

The blocks must be measurable for the integrals below to make sense. Fix
an $\varepsilon$-candidate $\pi \in \Pi_A$ and run rehearsal on the casting
$(\pi, (P_I(\omega))_I)$; its decisions are stopping times of the
filtration $(\mathcal{F}_t)$ generated by the $n$ trajectories.

By order preservation~(\ref{itm:km-order}) two representatives reverse
order only by first meeting, so the first crossing is the earliest
meeting of two active representatives; by consecutivity
(\Cref{sec:proof-adjacent}) this earliest meeting is automatically
between \emph{adjacent} ones. At the start the adjacent pairs are fixed
by the initial order and~$\pi$, so $\rho_1 = \min \tau_{I,J}$, the
minimum of the meeting times~(\ref{itm:km-meeting}) over this finite,
deterministic set, is a stopping time. If two or more adjacent pairs meet at the
same instant (at distinct positions), the lexicographic
order~$\leq$---time first, then
position---selects a unique pair to process first; this keeps the first
crossing well defined in every setting, and for Brownian motion and
non-degenerate diffusions the tie is moreover almost surely vacuous.

The meeting pair at $\rho_1$ is $\mathcal{F}_{\rho_1}$-measurable, and
whether the crossing is valid or spurious---the condition
$g \in \{\pi(\sigma^{-1}(I^-)),\; \pi(\sigma^{-1}(I^+))\}$---depends only on $\pi$ and
that pair, hence is $\mathcal{F}_{\rho_1}$-measurable. If the crossing
is valid, rehearsal fires a ghost and the surviving representative
carries the merged interval onward along its own trajectory, restricted
to $[\rho_1, T]$. The record of which ghost fired is
$\mathcal{F}_{\rho_1}$-measurable, so the new set of adjacent active
pairs is $\mathcal{F}_{\rho_1}$-measurable; and by the strong Markov
property~(\ref{itm:km-markov}) the surviving representatives form a
post-$\rho_1$ system again satisfying~(\ref{itm:km-identical})
and~(\ref{itm:km-meeting}). Hence $\rho_2$, the earliest post-$\rho_1$
meeting over this random but $\mathcal{F}_{\rho_1}$-measurable finite
family of adjacent pairs, is again a stopping time, and likewise each
subsequent~$\rho_i$. Each valid crossing removes one active actor, so
the recursion halts after at most $n - 1$ steps, producing stopping
times $\rho_1 \leq \cdots \leq \rho_k$.

The first spurious crossing occurs at the first $\rho_i$ whose crossing
is spurious---itself a stopping time, being the first member of a finite
chain of stopping times at which an $\mathcal{F}_{\rho_i}$-measurable
event holds; if none is spurious, the run is successful. Hence the event
$\{\omega : (\pi,\omega) \in \CastingSpace_A^{(I,J)}\}$ (first spurious
crossing the pair $(I,J)$) and the event
$\{\omega : (\pi,\omega) \in \CastingSpace_A^{\mathrm{succ}}\}$ are
$\mathcal{F}_T$-measurable; as $\Pi_A$ is finite, the blocks of the
partition are measurable. The construction uses only the Karlin--McGregor
assumptions, so it covers every process in the theorem's scope; for
processes with jumps, where two or more adjacent pairs may meet simultaneously
with positive probability, the lexicographic order~$\leq$---not a
null-set argument---keeps each $\rho_i$ unambiguous.

\subsection{The measure-preserving involution}

\subsubsection{Measure-preservation of the swap}

The segment swap is measure-preserving on every block. On
$\CastingSpace_A^{(I,J)}$ the swap acts at the first spurious crossing
$\rho_\star$ (the first spurious~$\rho_i$ above), at which the two
representatives occupy the same position. Because the underlying
particles are independent---$\PP$ is their product law on~$\Omega$,
by~(\ref{itm:km-identical})---the
post-$\rho_\star$ trajectories of the two representatives are
conditionally independent of each other given that common state; the
strong Markov property~(\ref{itm:km-markov}) lets each restart afresh
from the meeting state, independently of the pre-$\rho_\star$ past, and
identical dynamics~(\ref{itm:km-identical}) makes them identically
distributed. In
particular, the joint law of the pair
$(P_I^{\mathrm{post}}, P_J^{\mathrm{post}})$ is exchangeable: swapping
the two post-meeting segments produces a pair with the same
distribution. It also maps $\CastingSpace_A$ into itself: where it acts as the
identity this is immediate, and on the block $\CastingSpace_A^{(I,J)}$
the first spurious crossing is $(I,J)$, so $g \notin \{\pi(I), \pi(J)\}$
and by the swap criterion
(\Cref{lem:swap-criterion,cor:swap-at-failure-fires}) the swapped
bijection $\pi' = (I\; J) \circ \pi$ is again an $\varepsilon$-candidate, with
$\mathbf{X}_T(\omega') \in A_{\pi'}$. Since the swap acts by
exchanging exactly these segments while leaving all other paths
unchanged, exchangeability together with this codomain-invariance shows
that it preserves the product measure~$\mu$.

\subsubsection{Cancellation of failed castings}

On the failed part $\CastingSpace_A \setminus \CastingSpace_A^{\mathrm{succ}}$,
the involution $\iota$ is sign-reversing: $\gasgnof{\iota(\pi, \omega)} =
-\gasgnof{\pi}$. This holds because the segment swap exchanges the roles of
actors $I$ and $J$, so the new bijection $\pi' = (I\; J) \circ \pi$
differs from $\pi$ by a transposition
(\Cref{lem:swap-properties}\ref{itm:swap-sign}).

Since $\iota$ is also measure-preserving:
\[
  \int_{\CastingSpace_A \setminus \CastingSpace_A^{\mathrm{succ}}} \gasgnof{\pi} \, d\mu
  = \int_{\CastingSpace_A \setminus \CastingSpace_A^{\mathrm{succ}}}
    \gasgnof{\iota(\pi, \omega)} \, d\mu
  = -\int_{\CastingSpace_A \setminus \CastingSpace_A^{\mathrm{succ}}} \gasgnof{\pi} \, d\mu.
\]
Therefore $\int_{\text{failed}} \gasgnof{\pi} \, d\mu = 0$, and
\[
  \int_{\CastingSpace_A} \gasgnof{\pi} \, d\mu
  = \int_{\CastingSpace_A^{\mathrm{succ}}} \gasgnof{\pi} \, d\mu.
\]

\subsection{Completing the proof}

\subsubsection{Contribution from successful castings}

For successful castings, the ghost-adjusted sign is $\gasgnof{\pi} = +1$
(\Cref{prop:sign-identity}). Therefore
\[
  \int_{\CastingSpace_A^{\mathrm{succ}}} \gasgnof{\pi} \, d\mu
  = \mu(\CastingSpace_A^{\mathrm{succ}})
  = \PP_{\mathrm{int}}(\text{final positions} \in A),
\]
where the last equality uses the bijection between successful castings
and performances (\Cref{prop:performance-casting-bijection}). The bijection is purely
combinatorial---it depends only on the paths, not on the probability
measure---so the discrete proof applies verbatim to each
outcome~$\omega$: the lexicographic tie-break makes rehearsal
deterministic, so for \emph{every} $\omega \in \Omega$ with final
positions in~$A$ there is exactly one $\pi \in \Pi_A$ such that
$(\pi, \omega)$ is a successful casting. No exceptional set is
needed---the selection is deterministic for every outcome, including
those where two adjacent pairs meet simultaneously
(\Cref{sec:continuous-measurability}).

\subsubsection{The determinant identity}

The integral over the casting space equals
\[
  \int_{\CastingSpace_A} \gasgnof{\pi} \, d\mu
  = \sum_{\pi \in \Pi_A} \gasgnof{\pi} \, \PP\bigl( \mathbf{X}_T \in A_\pi \bigr).
\]
Combining with the cancellation of failed castings and the positivity
of successful ones gives
\[
  \PP_{\mathrm{int}}(\text{final positions} \in A)
  = \sum_{\pi \in \Pi_A} \gasgnof{\pi} \, \PP\bigl( \mathbf{X}_T \in A_\pi \bigr),
\]
which is \Cref{thm:continuous}.

\section{Integrating out the ghosts}\label{sec:determinant}

As previewed in \Cref{sec:intro-results}, integrating out the ghost
positions turns the coalescence formula (\Cref{thm:coalescence}) into a
determinant in the heir positions alone. The heirs are the surviving
particles, which we also call \emph{survivors}. The result is a
closed-form determinantal
formula---the \emph{coalescence determinant}---that, like
\Cref{thm:continuous}, is an identity between measures, valid for both
discrete and continuous state spaces. Its matrix definition and theorem
statement use only the coalescence pattern and transition kernels; the
ghost machinery enters only in the derivation.

\subsection{Transition kernels}
\label{sec:det-kernels}

Let $S$ denote the state space ($\RR$ or $\ZZ$), and fix a time
horizon $T > 0$. Write $X_T$ for the position of a single particle
at time~$T$, and let $P_x$ denote the transition kernel: the
distribution of $X_T$ when the particle starts at~$x$.
For a measurable set $B$, write $P_x(B) = \PP_x(X_T \in B)$.
Let $\nu$ denote the reference measure (Lebesgue measure on~$\RR$,
counting measure on~$\ZZ$), and write $p_x(y)$ for the density
of $P_x$ with respect to~$\nu$: the transition density (continuous
case) or transition probability (discrete case). Define the
cumulative distribution
\[
F_x(y) = P_x\bigl((-\infty, y]\bigr) = \PP_x(X_T \leq y).
\]

\subsection{The coalescence matrix}
\label{sec:det-matrix}

Fix a coalescence pattern: a composition
$c_1{+}\cdots{+}c_k = n$, where the first~$c_1$ particles
merge into survivor~$1$ at position~$y_1$, the next~$c_2$ into
survivor~$2$ at~$y_2$, and so on. The $l$th block of the
composition---the initial particles merging into
survivor~$l$---has indices $c_1{+}\cdots{+}c_{l-1}{+}1$
through $c_1{+}\cdots{+}c_l$.

\fnote{%
  prose={The ghost-free matrix $\tilde M$ is the integrated, continuous
    form; its discrete ghost analogue $M$ is \texttt{M\_V}
    (\Cref{prop:restricted-leibniz}). $\tilde M$ itself is not formalized;
    see \Cref{sec:lean-scope}.}}%
\begin{definition}[Coalescence matrix]\label{def:coalescence-matrix}
Both rows and columns of the $n \times n$ \emph{coalescence matrix}
$\tilde{M}$ are indexed by $\{1, \ldots, n\}$. The entry in
row~$i$, column~$j$ (where $j$ lies in the $l$th block, with
survivor position~$y_l$) is
\[
\tilde{M}_{ij} = \begin{cases}
p_{x_i}(y_l)
  & \text{if $j$ is the first index in its block}, \\
F_{x_i}(y_l) - [i < j]
  & \text{otherwise},
\end{cases}
\]
where $[i < j]$ denotes the Iverson bracket.
The first column of each block contains transition densities
(or probabilities); the remaining $c_l - 1$ columns contain
cumulative distributions with a ``staircase'' shift $-[i < j]$:
entries with $i < j$ are shifted by~$-1$ (the staircase region is the
same as in \Cref{fig:matrix-structure}, there shown for the ghost
matrix~$M$).
\end{definition}

\subsection{The coalescence determinant}
\label{sec:det-formula}

\fnote{%
  prose={Not formalized: the continuous, probabilistic density form (no
    measure theory in Lean). The Lean development certifies the discrete
    algebraic core this is derived from; see \Cref{sec:lean-scope}.}}%
\begin{theorem}[Coalescence determinant]\label{thm:ghost-free}
Label the survivors in increasing spatial order, and let
\[
  W_k = \{(y_1, \ldots, y_k) \in S^k : y_1 \prec \cdots \prec y_k\}
\]
be the set of strictly ordered survivor positions. For a measurable
set $A \subseteq W_k$,
\[
\PP_{\mathrm{int}}\bigl(\text{survivor positions} \in A\bigr)
= \int_A \det\bigl(\tilde{M}(y_1, \ldots, y_k)\bigr) \, d\nu^{\otimes k},
\]
where $\PP_{\mathrm{int}}$ denotes the probability for the
coalescing system (\Cref{sec:continuous}).
\end{theorem}

The restriction to~$W_k$ is the usual Karlin--McGregor caveat: the
matrix~$\tilde{M}$ is built from the ordered positions
$y_1 \prec \cdots \prec y_k$, and off the ordered chamber the
determinant need not even be nonnegative. On~$W_k$ the formula is an
identity between measures: $\det \tilde{M}$ is the Radon--Nikodym
derivative of the survivor-position distribution with respect
to~$\nu$. For continuous state spaces, $\det \tilde{M}$ is a
probability density on~$W_k$; for discrete state spaces, it is a
probability mass function. The companion
papers~\cite{Sniady2026coalescenceApplications}
and~\cite{Sniady2026pfaffian} develop applications of the
coalescence determinant.

\begin{proof}
Start from \Cref{thm:coalescence}: the total weight of performances
with a fixed sign vector $\varepsilon$ is $Z_\varepsilon = \det M$, whose
ghost columns are already specialized to~$\varepsilon$ through the
Iverson brackets $[\varepsilon_g = \pm 1]$. To
marginalize over ghost positions, integrate each ghost position
$y_g$ over the half-line on the side of its heir selected by its
sign---left, $y_g \preceq y_{\heir(g)}$, when $\varepsilon_g = +1$ and
right, $y_{\heir(g)} \prec y_g$, when $\varepsilon_g = -1$ (the
$\varepsilon$-admissible region of \Cref{def:eps-admissible-set})---and
sum over all sign vectors $\varepsilon \in \{+1,-1\}^{\Ghosts}$. For continuous state spaces this
starting identity is the density form of \Cref{thm:continuous}, with
$\det M$ the joint density of the final positions; the marginalization
below is then integration against that density.

Using multilinearity of the determinant in columns---it is linear in
each column separately, and the marginalization touches one ghost
column at a time---the summation and integration act column by
column. Heir columns contain
$W(x_i \to y_H)$ with no ghost-sign dependence, so they pass through
unchanged; under continuous state spaces, $W(x_i \to y_H)$
becomes the transition density $p_{x_i}(y_H)$.

For a ghost column~$g$ with heir $H = \heir(g)$, the junction index~$g$
separates the rows by which side of the junction each particle starts
on: those with $i \geq g$ start at or to the right of the junction,
those with $i < g$ to its left. The matrix entry is
$[\varepsilon_g=+1] \cdot W(x_i \to y_g)$ when $i \geq g$ and
$-[\varepsilon_g=-1] \cdot W(x_i \to y_g)$ when $i < g$. For a fixed sign, the matching
bracket equals~$1$ and the opposite bracket equals~$0$, so each
row keeps only its matching entry. Summing over both signs
$\varepsilon_g \in \{+1,-1\}$ and integrating over $y_g$, the
ghost column entry in row~$i$ becomes:
\begin{itemize}
\item If $i \geq g$: the entry $[\varepsilon_g=+1] \cdot W(x_i \to y_g)$
  contributes when $\varepsilon_g = +1$ (ghost left of heir,
  $y_g \leq y_H$). Integrating:
  \[\displaystyle \int_{y_g \leq y_H} p_{x_i}(y_g) \, d\nu(y_g)
  = F_{x_i}(y_H).\]
\item If $i < g$: the entry $-[\varepsilon_g=-1] \cdot W(x_i \to y_g)$
  contributes when $\varepsilon_g = -1$ (ghost right of heir,
  $y_g > y_H$). The surviving entry is $-W(x_i \to y_g)$.
  Integrating:
  \[-\!\int_{y_g > y_H} p_{x_i}(y_g) \, d\nu(y_g)
  = -(1 - F_{x_i}(y_H)) = F_{x_i}(y_H) - 1.\]
\end{itemize}
This is precisely the staircase pattern in
\Cref{def:coalescence-matrix}.
\end{proof}

\subsection{Example}
\label{sec:det-examples}

\begin{example}[Pattern $2{+}1$: three particles, first two coalesce]
\label{ex:ghost-free-3}
Continuing \Cref{sec:intro-example}: particles $1$ and $2$ coalesce
(survivor at $y_1$) while particle $3$ survives alone ($y_2$).
The composition is $2{+}1$: the first block has columns~$1$
and~$2$, the second block has column~$3$. The coalescence matrix
is:
\[
\tilde{M} = \begin{pmatrix}
p_{x_1}(y_1) & F_{x_1}(y_1) - 1 & p_{x_1}(y_2) \\
p_{x_2}(y_1) & F_{x_2}(y_1) & p_{x_2}(y_2) \\
p_{x_3}(y_1) & F_{x_3}(y_1) & p_{x_3}(y_2)
\end{pmatrix}.
\]
The columns follow the same staircase logic as the intro example
(\Cref{sec:intro-example}), with cumulative distributions replacing
transition probabilities: column~$2$ carries the shift $-[i<j]$, so
row~$1$ gives $F - 1$ and rows~$2$ and~$3$ give~$F$.

For any measurable set $A \subseteq W_2 = \{(y_1, y_2) : y_1 \prec
y_2\}$ of ordered survivor positions:
\[
\PP_{\mathrm{int}}\bigl((y_1, y_2) \in A\bigr)
= \int_A \det \tilde{M} \, d\nu^{\otimes 2}.
\]
For Brownian motion ($S = \RR$), $d\nu^{\otimes 2} = dy_1\, dy_2$ and
$\det \tilde{M}$ is a joint density on~$W_2$. For simple random walk
($S = \ZZ$), $\nu$ is counting measure and $\det \tilde{M}$ is
the probability mass at $(y_1, y_2)$.
\end{example}

\ifextendedversion\subfile{sections/appendix-lean}\fi

\section*{Acknowledgments}

We thank Theodoros Assiotis, Bal\'azs B\'ar\'any, Maciej Dołęga,
Sho Matsumoto, B\'alint T\'oth,
Oleg Zaboronski, and Karol Życzkowski
for stimulating discussions and helpful literature suggestions.

P.~\'Sniady was supported by the National Science Centre, Poland,
grant number~2025/59/B/ST1/01258.

\section*{Declaration of generative AI and AI-assisted technologies \\
  in the manuscript preparation process}

During the preparation of this work P.~\'Sniady used Claude (Anthropic)
in order to draft and edit the text and to produce the figures (TikZ
code). After using this tool, the authors reviewed and edited the
content as needed and take full responsibility for the content of the
published article.

\printbibliography

\end{document}